\documentclass[smallextended]{svjour3_GD}
\smartqed
\usepackage{amsmath,latexsym,amsfonts,amssymb,mathrsfs,ulem}
\usepackage{graphicx,verbatim}
\usepackage{a4wide}
\allowdisplaybreaks
\usepackage{mathtools} 
\usepackage{stackengine} 
\usepackage[active]{srcltx}
\usepackage[colorlinks,linkcolor={blue},
citecolor={blue},urlcolor={red},]{hyperref}
\usepackage{color}

\newcommand{\oldReps}{\ocirc{R}_\eps}
\newcommand{\newReps}{{R}_\eps}
\newcommand{\newMepshat}{\widehat{{M}}_\eps}
\newcommand{\newMepstilde}{\widetilde{M}_\eps}
\newcommand{\pa}{\partial}

\newtheorem{assumption}[theorem]{Assumption}
\numberwithin{equation}{section}

\newcommand\delc[1]{}
\newcommand{\rA}{\mathrm{ A}}
\newcommand{\dom}{\mathcal{O}}

\newcommand{\rH}{\mathrm{ H}}
\newcommand{\rV}{\mathrm{ V}}
\newcommand{\rU}{\mathrm{ U}}
\newcommand{\rD}{\mathrm{D}}
\def\v{{\mathrm{v}}}
\def\u{{\mathrm{u}}}
\def\w{{\mathrm{w}}}
\newcommand{\test}{{\oldReps \phi}}
\newcommand{\tMe}{\widetilde{M}_\eps}

\newcommand{\oMe}{\ocirc{M}_\eps}
\newcommand{\tNe}{\widetilde{N}_\eps}
\newcommand{\tue}{\widetilde{u}_\eps}
\newcommand{\tWe}{\wtd{W}_\eps}
\newcommand{\tge}{\widetilde{g}_\eps}
\newcommand{\tfe}{\widetilde{f}_\eps}
\newcommand{\hE}{\widehat{\E}}
\newcommand{\inteps}{\int_{1}^{1+\eps}}
\newcommand{\LS}{{\mathbb{L}^2(\bbS)}}

\newcommand{\tale}{\widetilde{\alpha}_\eps}
\newcommand{\tble}{\widetilde{\beta}_\eps}
\newcommand{\ale}{{\alpha}_\eps}
\newcommand{\ddivS}{{\rm{div}^\prime}}
\newcommand{\curlS}{{\rm{curl}^\prime}}
\newcommand{\deltaS}{{\Delta^\prime}}
\newcommand{\nablaS}{\nabla^\prime}
\newcommand{\tu}{\widetilde{u}}
\newcommand{\hu}{\widehat{u}}

\newcommand{\hae}{\widehat{\alpha}_\eps}
\newcommand{\hp}{\widehat{\mathbb{P}}}
\newcommand{\homega}{\widehat{\Omega}}
\newcommand{\bbW}{\mathbb{W}}

\newcommand{\hbe}{\widehat{\beta}_\eps}
\newcommand{\be}{\begin{equation}}
\newcommand{\ee}{\end{equation}}
\newcommand{\ba}{\begin{array}}
\newcommand{\ea}{\end{array}}
\newcommand{\beas}{\begin{eqnarray*}}
\newcommand{\eeas}{\end{eqnarray*}}
\newcommand{\bea}{\begin{eqnarray}}
\newcommand{\eea}{\end{eqnarray}}
\newcommand{\nn}{\nonumber}
\newcommand{\lb}{\label}
\newcommand{\hH}{\mathbb{H}}
\newcommand{\lL}{\mathbb{L}}
\newcommand{\N}{\mathbb{N}}
\newcommand{\bbF}{\mathbb{F}}
\newcommand{\bbP}{\mathbb{P}}
\newcommand{\bbS}{{\mathbb{S}^2}}
\newcommand{\E}{\mathbb{E}}
\newcommand{\bbE}{\mathbb{E}}
\newcommand{\R}{\mathbb{R}}
\newcommand{\Rp}{\mathbb{R}_+}

\newcommand{\calA}{\mathcal{A}}
\newcommand{\bcal}{\mathcal{B}}
\newcommand{\calB}{\mathcal{B}}
\newcommand{\ccal}{\mathcal{C}}
\newcommand{\calF}{\mathcal{F}}
\newcommand{\calG}{\mathcal{G}}
\newcommand{\calH}{\mathcal{H}}
\newcommand{\kcal}{\mathcal{K}}
\newcommand{\calT}{\mathcal{T}}
\newcommand{\calV}{\mathcal{V}}
\newcommand{\calZ}{\mathcal{Z}}
\newcommand{\ddiv}{{\rm div\;}}
\newcommand{\curl}{{\rm curl\;}}
\newcommand{\eps}{\varepsilon}
\newcommand{\wtd}{\widetilde}
\newcommand{\what}{\widehat}
\newcommand{\bx}{{\bf x}}
\newcommand{\by}{{\bf y}}
\newcommand{\bz}{{\bf z}}
\newcommand{\vecn}{{\vec{n}}}
\newcommand{\DDelta}{\boldsymbol{\mathrm{\Delta}}^\prime \, }
\newcommand{\LQeps}{{\lL^2(Q_\eps)}}
\newcommand\ocirc[1]{\ensurestackMath{\stackon[1pt]{#1}{\mkern2mu\circ}}}

\begin{document}
\title{Stochastic Navier--Stokes equations on a thin spherical domain
\thanks{The research of all three authors is partially supported by Australian Research Council Discover Project grant DP180100506, ``Uncertainty on Spheres and Shells: Mathematics and Methods for Applications''. Zdzis\l aw Brze{\'z}niak has been supported by the Leverhulme project grant ref no RPG-2012-514 and by Australian Research Council Discover Project grant DP160101755. The research of Gaurav Dhariwal was supported by Department of Mathematics, University of York and is partially supported by the Austrian Science Fund (FWF) grants P30000, W1245, and F65.
}}
\titlerunning{SNSE on a thin spherical domain}
\author{Zdzis\l aw Brze\'{z}niak \and Gaurav Dhariwal \and Quoc Thong Le Gia}
\institute{Z. Brze\'{z}niak \at Department of Mathematics, University of York, Heslington, York, YO10 5DD, UK \\ \email{zdzislaw.brzezniak@york.ac.uk} \and 
G.~Dhariwal \at Institute of Analysis and Scientific Computing, Vienna University of Technology, Vienna, Austria \\ \email{gaurav.dhariwal@tuwien.ac.at} \and
Q.~T.~Le Gia \at School of Mathematics and Statistics, University of New South Wales, Sydney, NSW 2052, Australia \\
\email{qlegia@unsw.edu.au}
}

\authorrunning{Brze{\'z}niak et al.}
\date{Published Online: July 11, 2020}
\maketitle

\begin{abstract}
Incompressible Navier--Stokes equations on a thin spherical domain $Q_\varepsilon$ along with free boundary conditions under a random forcing are considered. The convergence of the martingale solution of these equations to the martingale solution of the stochastic Navier--Stokes equations on a sphere $\mathbb{S}^2$ as the thickness converges to zero is established.
\end{abstract}

\section{Introduction}
\label{sec:intro}

For various motivations, partial differential equations in thin domains have been studied extensively in the last few decades;
e.g. Babin and Vishik \cite{[BV83]}, Ciarlet \cite{[Ciarlet90]}, Ghidaglia and Temam \cite{[GT91]}, Marsden \textit{et.al.} \cite{[MRR95]} and references there in. The study of the Navier--Stokes equations (NSE) on thin domains originates in a series of papers by Hale and Raugel \cite{[HR92a]}--\cite{[HR92c]} concerning the reaction-diffusion and damped wave equations on thin domains. Raugel and Sell \cite{[RS93],[RS94]} proved the global existence of strong solutions to NSE on thin domains for large initial data and forcing terms, in the case of purely periodic and periodic-Dirichlet boundary conditions. Later,
by applying a contraction principle argument and carefully analysing the dependence of the solution on the first eigenvalue of the corresponding Laplace operator, Arvin \cite{[Arvin96]} showed global existence of strong solutions of the Navier--Stokes equations on thin three-dimensional domains for large data. Temam and Ziane \cite{[TZ96]} generalised the results of \cite{[RS93],[RS94]} to other boundary conditions. Moise \textit{et.al.} \cite{[MTZ97]} proved global existence of strong solutions for initial data larger than in \cite{[RS94]}. Iftimie \cite{[Iftime99]} showed the existence and uniqueness of solutions for less regular initial data which was further improved by Iftimie and Raugel \cite{[IR01]} by reducing the regularity and increasing the size of initial data and forcing.

In the context of thin spherical shells,
large-scale atmospheric dynamics that play an important role in global climate models and weather
prediction can be described by the 3-dimensional Navier--Stokes equations in a thin rotating spherical
shell \cite{LionTemamWang92,LionTemamWang92a}.
Temam and Ziane in \cite{[TZ97]} gave the mathematical justification for
the primitive equations of the atmosphere and the oceans which are known to be the fundamental equations of
meteorology and oceanography \cite{[LTW95],[Pedlosky87]}. The atmosphere is a compressible fluid occupying a
thin layer around the Earth and whose dynamics can be described by the 3D compressible Navier--Stokes equations
in thin layers. In \cite{[TZ97]} it was assumed that the atmosphere is incompressible and hence a 3D incompressible
NSE on thin spherical shells could be used as a mathematical model.
They proved that the averages in the radial direction of the strong solutions
(whose existence for physically relevant initial data was established in the same article)
to the NSE on the thin spherical shells converge to the solution of the NSE on the sphere as
the thickness converges to zero.
In a recent paper Saito \cite{[Saito05]} studied the 3D Boussinesq equations in thin spherical domains and
proved the convergence of the average of weak solutions of the 3D Boussinesq equations to a 2D problem.
More recent work on incompressible viscous fluid flows
in a thin spherical shell was carried out in \cite{IbraPelin2009,Ibra2011,IbraIbra2011}.

For the deterministic NSE on the sphere, Il'in and Filatov \cite{[Ilin91]} - \cite{[IF89]} considered the existence
and uniqueness of solutions while Temam and Wang \cite{[TW93]} considered inertial forms of NSE on spheres.
Brze\'{z}niak \textit{et. al.} proved the existence and uniqueness of the solutions to the stochastic NSE on the
rotating two dimensional sphere and also proved the existence of an asymptotically compact random dynamical
system \cite{[BGL15]}. Recently, Brze\'{z}niak \textit{et. al.} established \cite{[BGL18]} the existence of
random attractors for the NSE on two dimensional sphere under random forcing irregular in space and time
deducing the existence of an invariant measure.

The main objective of this article is to establish the convergence of the martingale solution of the stochastic
Navier--Stokes equations (SNSE) on a thin spherical domain $Q_\eps$, whose existence can be established as in the forthcoming paper \cite{[BDL19]} to the martingale solution of the stochastic Navier--Stokes equations
on a two dimensional sphere $\bbS$ \cite{[BGL15]} as thickness $\eps$ of the spherical domain converges to zero. In this way we also give another proof for the existence of a martingale solution for stochastic NSE on the unit sphere $\bbS$.

We study the stochastic Navier--Stokes equations (SNSE) for incompressible fluid \begin{align}
 d \tue - [ \nu \Delta \tue - (\tue \cdot \nabla) \tue - \nabla \widetilde{p}_\eps]dt = \tfe dt +  \widetilde{G}_\eps\, d\tWe(t) &\quad \text{ in } Q_\eps \times (0,T), \label{eq:0.1}\\
\ddiv \tue = 0 &\quad \text{ in } Q_\eps \times (0,T),\label{eq:0.2}
\end{align}
in thin spherical shells
\begin{equation}\label{def:Qeps}
Q_\eps := \left\{{\bf y} \in \R^3 : 1 \le |{\bf y}| \le 1+ \eps\right\}, \text{ where } 0< \eps < 1/2,
\end{equation}
along with free boundary conditions
\begin{align}
\tue \cdot \vec{n} = 0, \quad \curl \tue \times \vec{n} = 0 &\quad\text{ on } \partial Q_\eps \times (0,T), \label{eq:0.3} \\
\tue(0, \cdot) = \tu_0^\eps &\quad\text{ in } Q_\eps. \label{eq:0.4}
\end{align}
In the above, $\tue=(\tue^r, \tue^\lambda, \tue^\varphi)$ is the fluid velocity field, $p$ is the pressure, $\nu>0$ is a (fixed) kinematic viscosity, $\tu_0^\eps$ is a divergence free
vector field on $Q_\eps$ and $\vecn$ is the unit outer normal vector to the boundary $\partial Q_\eps$ and $\tWe(t)$, $t \ge 0$ is an $\R^N$-valued Wiener process in some probability space $\left(\Omega, \mathcal{F}, \mathbb{F}, \mathbb{P}\right)$ to be defined precisely later.

The main result of this article is Theorem~\ref{thm:main_thm}, which establishes the convergence of the radial averages of the martingale solution (see Definition~\ref{defn5.2}) of the 3D stochastic equations \eqref{eq:0.1}--\eqref{eq:0.4}, as the thickness of the shell $\eps \rightarrow 0$, to a martingale solution $u$ (see Definition~\ref{defn_mart_SNSE_sphere}) of the following stochastic Navier--Stokes equations on
the unit sphere $\bbS \subset {\mathbb R}^3$:
\begin{align}
du - \left[ \nu \DDelta u - (u \cdot \nablaS) u  - \nablaS p \right]dt = f dt + G\,dW(t) & \quad\text{ in } \bbS \times (0,T), \label{eq:0.5}\\
\ddivS u = 0 & \quad \text{ in } \bbS \times (0,T),\label{eq:0.6} \\
u(0, \cdot ) = u_0 & \quad\text{ in } \bbS,\label{eq:0.7}
\end{align}
where $u=(u_\lambda, u_\varphi)$ and $\DDelta$, $\nablaS$ are
the Laplace--de Rham operator and the surface gradient on $\bbS$ respectively. Assumptions on initial data and external forcing will be specified later.

The paper is organised as follows. We introduce necessary functional spaces in Section~\ref{sec:prelim}. In Section~\ref{sec:operators}, we define some averaging operators and give their properties. Navier--Stokes equations on thin shells $Q_\eps$ and on the unit sphere $\bbS$ driven by a deterministic forcing are introduced in Section~\ref{sec:detNSE} and Section~\ref{sec:NSE_S} respectively. Moreover, in Section~\ref{sec:NSE_S} we show the convergence of the radial average of the weak solution of NSE on thin spherical domain to unique solution of NSE on sphere, which indeed is Temam and Ziane's result \cite{[TZ97]}. Stochastic Navier--Stokes equations on thin spherical domains are introduced in Section~\ref{sec:SNSE_shell} and a priori estimates for the radially averaged velocity are obtained which are later used to prove the convergence of the radial average of a martingale solution of stochastic NSE on thin spherical shell (see \eqref{eq:0.1}--\eqref{eq:0.4}) to a martingale solution of the stochastic NSE on the sphere (see \eqref{eq:0.5}--\eqref{eq:0.7}) with vanishing thickness.

\section{Preliminaries}
\label{sec:prelim}
A point $\by \in Q_\eps$ could be represented by the Cartesian coordinates $\by = \left(x, y, z\right)$
or $\by = \left(r, \lambda, \varphi\right)$ in spherical coordinates, where
\[x = r \sin \lambda \cos \varphi \quad y = r \sin \lambda \sin \varphi \quad z = r \cos \lambda,\]
for $r \in (1, 1+\eps)$, $\lambda \in [0,\pi]$ and $\varphi \in [0, 2 \pi)$.

For $p \in [1, \infty)$, by $L^p(Q_\eps)$, we denote the Banach space of (equivalence-classes of) Lebesgue measurable $\R$-valued $p$th power integrable functions on $Q_\eps$. The $\R^3$-valued $p$th power integrable vector fields will be denoted by $\mathbb{L}^p(Q_\eps)$. The norm in $\mathbb{L}^p(Q_\eps)$ is given by
\[\|u\|_{\mathbb{L}^p(Q_\eps)} := \left(\int_{Q_\eps}|u(\by)|^p\,d\by\right)^{1/p}, \qquad u \in \mathbb{L}^p(Q_\eps).\]
If $p=2$, then $\mathbb{L}^2(Q_\eps)$ is a Hilbert space with the inner product given by
\[\left(u, \v \right)_{\LQeps} := \int_{Q_\eps} u(\by) \cdot \v(\by)\,d\by, \qquad u, \v \in \LQeps.\]
By $\mathbb{H}^1(Q_\eps) = \mathbb{W}^{1,2}(Q_\eps)$, we will denote the Sobolev space consisting of all $u \in \LQeps$ for which there exist weak derivatives $D_i u \in \LQeps$, $i = 1,2,3$. It is a Hilbert space with the inner product given by
\[ \left( u, \v \right)_{\mathbb{H}^1(Q_\eps)} := \left(u, \v \right)_\LQeps + \left(\nabla u, \nabla \v \right)_\LQeps, \qquad u, \v \in \mathbb{H}^1(Q_\eps),\]
where
\[\left(\nabla u, \nabla \v \right)_\LQeps = \sum_{i=1}^3\int_{Q_\eps} D_i u(\by) \cdot D_i \v(\by)\,d\by.\]
The Lebesgue and Sobolev spaces on the sphere $\bbS$ will be denoted by $\mathbb{L}^p(\bbS)$ and $\mathbb{W}^{s,q}(\bbS)$ respectively for $p, q \ge 1$ and $s \ge 0$. In particular, we will write $\mathbb{H}^1(\bbS)$ for $\mathbb{W}^{1,2}(\bbS)$.

\subsection{Functional setting on the shell \texorpdfstring{$Q_\eps$}{}}
\label{sec:functional_setting_shell}

We will use the following classical spaces on $Q_\eps \colon$
\begin{align*}
\rH_\eps &= \left\{ u \in \LQeps \colon \ddiv u = 0 \mbox{ in } Q_\eps,\,\, u \cdot \vec{n} = 0\,\, \mbox{on } \partial Q_\eps\right\},\\
\rV_\eps &= \left\{ u \in \hH^1(Q_\eps) \colon \ddiv u = 0 \mbox{ in } Q_\eps,\,\, u \cdot \vec{n} = 0\,\, \mbox{on } \partial Q_\eps\right\},\\
& = \hH^1(Q_\eps) \cap \rH_\eps.
\end{align*}
On $\rH_\eps$, we consider the inner product and the norm inherited from $\LQeps$ and denote them by $\left( \cdot, \cdot \right)_{\rH_\eps}$ and $\|\cdot\|_{\rH_\eps}$ respectively, that is
\[\left(u, \v \right)_{\rH_\eps} := \left(u, \v \right)_\LQeps, \qquad \|u\|_{\rH_\eps} := \|u\|_\LQeps,\qquad u,\v \in \rH_\eps.\]

Let us define a bilinear map $\mathrm{a}_\eps:\rV_\eps \times \rV_\eps \to \R$ by
\begin{equation}
\label{eq:1.4}
\mathrm{a}_\eps(u,\v) := \left(\curl u, \curl \v \right)_{\lL^2(Q_\eps)},\qquad u,\v \in \rV_\eps,
\end{equation}
where
\[\curl u = \nabla \times u,\]
and for $u \in \rV_\eps$, we define
\begin{equation}
\label{eq:1.5}
\|u\|_{\rV_\eps}^2 := \mathrm{a}_\eps(u,u) = \|\curl u\|^2_{\LQeps}.
\end{equation}
Note that for $u \in \rV_\eps$, $\|u\|_{\rV_\eps} = 0$ implies that $u$ is a constant vector and $u\cdot \vec{n} = 0$ on $\partial Q_\eps$ i.e., $u$ is tangent to $Q_\eps$ for every $\by \in \partial Q_\eps$, and thus must be $0$. Hence $\|\cdot\|_{\rV_\eps}$ is a norm on $\rV_\eps$ (other properties can be verified easily). Under this norm $\rV_\eps$ is a Hilbert space with the inner product given by
\[\left( u, \v \right)_{\rV_\eps} := \left(\curl u, \curl \v \right)_{\LQeps}, \qquad u, \v \in \rV_\eps.\]
We denote the dual pairing between $\rV_\eps$ and $\rV^\prime_\eps$ by $\langle \cdot, \cdot \rangle_\eps$, that is $\langle \cdot, \cdot \rangle_\eps := {}_{\rV_\eps^\prime}\langle \cdot, \cdot \rangle_{\rV_\eps}$. By the Lax--Milgram theorem, there exists a unique bounded linear operator $\mathcal{A}_\eps \colon \rV_\eps \to \rV_\eps^\prime$ such that we have the following equality$\colon$
\begin{equation}
\label{eq:1.6}
\langle \mathcal{A}_\eps u, \v \rangle_\eps = \left(u,\v\right)_{\rV_\eps}, \qquad u, \v \in \rV_\eps.
\end{equation}
The operator $\mathcal{A}_\eps$ is closely related to the Stokes operator $\rA_\eps$ defined by
\begin{equation}
\label{eq:1.7}
\begin{split}
\mathrm{D}(\rA_\eps) &= \left\{ u \in \rV_\eps : \mathcal{A}_\eps u \in \rH_\eps, \; \curl u \times \vec{n} = 0 \mbox{ on } \partial Q_\eps\right\},\\
\mathrm{A}_\eps u &= \mathcal{A}_\eps u, \qquad u \in \mathrm{D}(\rA_\eps).
\end{split}
\end{equation}
The Stokes operator $\rA_\eps$ is a non-negative self-adjoint operator in $\rH_\eps$ (see Appendix~\ref{sec:curl-stokes}). Also note that
\[\rD(\rA_\eps) = \left\{u \in \mathbb{H}^2(Q_\eps) \colon \ddiv u = 0 \mbox{ in } Q_\eps,\,\, u \cdot \vec{n} = 0 \mbox{ and } \curl u \times \vec{n} = 0 \mbox{ on } \partial Q_\eps\right\}.\]
We recall the Leray--Helmholtz projection operator $P_\eps$, which is the orthogonal projector of $\LQeps$ onto $\rH_\eps$. Using this, the Stokes operator $\rA_\eps$ can be characterised as follows$\colon$
\begin{equation}
\label{eq:1.8}
\rA_\eps u = P_\eps (-\Delta u), \qquad u \in \rD(\rA_\eps).
\end{equation}
We also have the following characterisation of the Stokes operator $\rA_\eps$ \cite[Lemma~1.1]{[TZ97]}$\colon$
\begin{equation}
\label{eq:1.9}
\rA_\eps u =\mathrm{curl} \left(\curl u\right), \qquad u \in \rD(\rA_\eps).
\end{equation}
For $u \in \rV_\eps$, $\v \in \rD(\rA_\eps)$, we have the following identity (see Lemma~\ref{lem:curl-stokes})
\begin{equation}
    \label{eq:1.9a}
    \left(\curl u, \curl \v\right)_{\LQeps} = \left(u, \rA_\eps \v \right)_{\LQeps}.
\end{equation}

Let $b_\eps$ be the continuous trilinear from on $\rV_\eps$ defined by$\colon$
\begin{equation}
\label{eq:1.10}
b_\eps(u, \v, \w) = \int_{Q_\eps}\left(u \cdot \nabla \right) \v \cdot \w \,d\by,\qquad u,\v,\w \in \rV_\eps.
\end{equation}
We denote by $B_\eps$ the bilinear mapping from $\rV_\eps\times\rV_\eps$ to $\rV_\eps^\prime$ by
\[\left\langle B_\eps(u,\v), \w \right\rangle_\eps = b_\eps(u,\v,\w), \qquad u,\v, \w \in \rV_\eps,\]
and we set
\[B_\eps(u) = B_\eps(u,u).\]
Let us also recall the following properties of the form $b_\eps$, which directly follows from the definition of $b_\eps$ $\colon$
\begin{equation}
\label{eq:1.11}
b_\eps(u,\v,\w) = - b_\eps(u,\w,\v),\qquad u,\v, \w \in \rV_\eps.
\end{equation}
In particular,
\begin{equation}
\label{eq:1.12}
\left\langle B_\eps(u,\v), \v \right\rangle_\eps = b_\eps(u, \v, \v) = 0, \qquad u,\v \in \rV_\eps.
\end{equation}

\subsection{Functional setting on the sphere \texorpdfstring{$\bbS$}{}}
\label{sec:functional_setting_sphere}
Let $s\ge0$. The Sobolev space $H^s(\bbS)$ is the space of all scalar functions $\psi \in L^2(\bbS)$ such that $(-\deltaS)^{s/2} \psi \in L^2(\bbS)$, where $\deltaS$ is the Laplace--Beltrami operator on the sphere (see \eqref{eq:C.1.2}). We similarly define $\hH^s(\bbS)$ as the space of all vector fields $u \in \lL^2(\bbS)$ such that $(-\DDelta)^{s/2} u \in \LS$, where $\DDelta$ is the Laplace--de Rham operator on the sphere (see \eqref{eq:C.1.5}).

For $s \ge 0$, $\left(H^s(\bbS), \|\cdot\|_{H^{s}(\bbS)}\right)$ and $\left(\hH^s(\bbS), \|\cdot\|_{\hH^{s}(\bbS)}\right)$ are Hilbert spaces under the respective norms, where
\begin{equation}\label{def:Hs}
  \|\psi\|^2_{H^s(\bbS)} = \|\psi\|^2_{L^2(\bbS)} + \| (-\deltaS)^{s/2} \psi\|^2_{L^2(\bbS)}, \qquad \psi \in H^s(\bbS)
\end{equation}
and
\begin{equation}\label{def:hHs}
  \|u\|^2_{\hH^s(\bbS)} = \|u\|^2_{\lL^2(\bbS)} + \|(-\DDelta)^{s/2} u\|^2_{\lL^2(\bbS)}, \qquad u \in \hH^s(\bbS).
\end{equation}

By the Hodge decomposition theorem~\cite[Theorem 1.72]{Aub98} the space of $C^\infty$
smooth  vector fields on $\bbS$ can be decomposed into three components:
\begin{equation} \label{Hodge}
  C^\infty(T\bbS) =   \calG \oplus  \calV \oplus \calH,
\end{equation}
where
\begin{equation} \label{orth_space}
   \calG = \{\nabla^\prime \psi: \psi \in C^\infty(\bbS)\},\quad
   \calV = \{{\rm curl}^\prime \psi: \psi \in C^\infty(\bbS)\},
\end{equation}
and $\calH$ is the finite-dimensional space of harmonic vector fields. Since
the sphere is simply connected, $\calH = \{0\}$.
We introduce the following spaces
\beas
    \rH &=& \mbox{ closure of } \calV
          \mbox{ in } \lL^2(\bbS), \\
    \rV &=& \mbox{ closure of } \calV
          \mbox{ in } \hH^1(\bbS).
\eeas
Note that it is known (see \cite{[Temam00]})
\beas
    \rH &=& \{u \in  \lL^2(\bbS): {\rm div}^\prime u =0\}, \\
    \rV &=& \rH\cap  \hH^1(\bbS).
\eeas

Given a tangential vector field $u$ on $\bbS$, we can find vector field $\tilde{u}$ defined
on some neighbourhood of $\bbS$ such that their restriction to $\bbS$ is equal to $u$,
that is $\tilde{u}\vert_{\bbS} = u \in T\bbS$. Then we define
\begin{equation}\label{def:curl}
  {\rm curl}^\prime u (\bx) := (\bx \cdot (\nabla \times \tilde{u}))\vert_{\bbS}
                     = (\bx \cdot {\rm curl\;} \tilde{u})\vert_{\bbS}.
\end{equation}
Since $\bx$ is orthogonal to the tangent plane $T_{\bx} \bbS$, ${\rm curl}^\prime\; u$ is the normal
component of $\nabla \times \tilde{u}$. It could be identified with a normal vector field when
needed.

We define the bilinear form $a: \rV \times \rV \rightarrow \R$ by
\[
   a(u,\v) := ({\rm curl}^\prime u, {\rm curl}^\prime \v)_{\LS}, \qquad u, \v \in \rV.
\]
The bilinear from $a$ satisfies $a(u,\v) \le \|u\|_{\hH^1(\bbS)} \|\v\|_{\hH^1(\bbS)}$ and hence is
continuous on $\rV$. So by the Riesz representation theorem, there exists a unique operator
$\calA : \rV \rightarrow \rV^\prime$ such that $a(u, \v) = {}_{\rV^\prime}\langle\calA u, \v\rangle_{\rV}$ for $u, \v \in \rV$.
Using the Poincar\'{e} inequality, we also have $a(u,u) \ge \alpha \|u\|^2_{\rV}$, for some
positive constant $\alpha$, which means $a$ is coercive in $\rV$. Hence, by the Lax--Milgram theorem, the operator $\calA : \rV \to \rV^\prime$ is an isomorphism.

Next we define an operator $\rA$ in $\rH$ as follows:
\begin{equation}
\begin{cases}
  \rD(\rA) &:= \{ u \in \rV:  \calA u \in \rH\}, \\
  \rA u &:= \calA u, \quad u \in \rD(\rA).
\end{cases}
\end{equation}
By Cattabriga \cite{[Cattabriga61]}, see also Temam \cite[p.~56]{[Temam97]}, one can show that $\rA$ is a non-negative self-adjoint operator in $\rH$. Moreover, $\rV = \rD(\rA^{1/2})$, see \cite[p.~57]{[Temam97]}.

Let $P$ be the orthogonal projection from $\LS$ to $\rH$, called the Leray--Helmholtz projection. It can be shown, see \cite[p.~104]{[Grigoryan00]}, that
\begin{equation}
\label{eq:stokes_sphere}
\rD(\rA) = \hH^2(\bbS) \cap \rH, \qquad \rA u =  P\left(-\DDelta u\right), \quad u \in \rD(\rA).
\end{equation}
$\rD(\rA)$ along with the graph norm
\[\|u\|_{\rD(\rA)}^2 := \|u\|_{\LS}^2 + \|\rA u\|^2_{\LS}, \quad u \in\rD(\rA),\]
forms a Hilbert space with the inner product
\[\langle u, \v \rangle_{\rD(\rA)} := \left( u, \v\right)_\LS + \left( \rA u, \rA \v\right)_\LS, \quad u, \v \in \rD(\rA).\]
Note that $\rD(\rA)$-norm is equivalent to $\hH^2(\bbS)$-norm. For more details about the Stokes operator on the sphere and fractional power $\rA^{s}$ for $s \ge 0$, see \cite[Sec. 2.2]{[BGL15]}.

Given two tangential vector fields $u$ and $\v$ on $\bbS$, we can find vector fields
$\tilde{u}$ and $\tilde{\v}$ defined on some neighbourhood of $\bbS$ such that
their restrictions to $\bbS$ are equal to, respectively, $u$ and $\v$. Then we define
the covariant derivative
\[
  [\nabla^\prime_{\v} u](\bx) =
  \pi_{\bx} \left( \sum_{i=1}^3 \tilde{\v}_i(\bx) \partial_i \tilde{u}(\bx)\right)
    = \pi_\bx( (\tilde{\v}(\bx) \cdot \nabla) \tilde{u}(\bx) ), \quad \bx \in \bbS,
\]
where $\pi_\bx$ is the orthogonal projection from $\R^3$ onto the tangent space $T_\bx \bbS$ to $\bbS$
at $\bx$. By decomposing $\tilde{u}$ and $\tilde{\v}$ into tangential and normal components and using
orthogonality, one can show that
\begin{equation}\label{eq:identity1}
\pi_\bx( \tilde{u} \times \tilde{\v}) = u \times ( (\bx \cdot \v) x) + (\bx \cdot u)\bx \times \v =
         u \times ( (\bx \cdot v) \bx), \quad \bx \in \bbS,
\end{equation}
where in the last equality, we use the fact that $\bx \cdot \v = 0$ for any tangential vector $\v$.

We set $\v = u$ and use the formula
\[
 (\tilde{u} \cdot \nabla)\tilde{u} = \nabla \frac{|\tilde{u}|^2}{2} - \tilde{u} \times (\nabla \times \tilde{u})
\]
to obtain
\[
 [\nabla^\prime_u u](\bx)
   = \nabla^\prime \frac{|u|^2}{2} - \pi_\bx \left(\tilde{u} \times (\nabla \times \tilde{u})\right).
\]
Using \eqref{eq:identity1} for the vector fields $\tilde{u}$ and
$\tilde{\v} = \nabla \times \tilde{u} ={\rm curl}\;\tilde{u}$, we have
\begin{align*}
\pi_\bx ( \tilde{u} \times (\nabla \times \tilde{u})) = u \times ( (\bx \cdot {\rm curl } u)\bx ) = u \times {\rm curl }^\prime u.
\end{align*}
Thus
\[
  \nabla^\prime_u u = \nabla^\prime \frac{|u|^2}{2}  - u \times {\rm curl }^\prime u.
\]
We consider the trilinear form $b$ on $\rV \times \rV \times \rV$, defined by
\begin{equation}
  b(\v, \w , z) = (\nabla^\prime_{\v} \w,z) = \int_{\bbS} \nabla^\prime_{\v} \w \cdot z\,d\sigma(\bx),
     \quad \v, \w, z \in \rV,
\end{equation}
where $d \sigma(\bx)$ is the surface measure on $\mathbb{S}^2$.
\section{Averaging operators and their properties}
\label{sec:operators}

In this section we recall the averaging operators which were first introduced by Raugel and Sell \cite{[RS93],[RS94]} for thin domains. Later, Temam and Ziane \cite{[TZ97]} adapted those averaging operators to thin spherical domains, introduced some additional operators and proved their properties using the spherical coordinate system. Recently, Saito \cite{[Saito05]} used these averaging operators to study Boussinesq equations in thin spherical domains. We closely follow \cite{[Saito05],[TZ97]} to describe our averaging operators and provide proofs for some of the properties mentioned below.

Let $M_\eps \colon \ccal(Q_\eps, \R) \to \ccal(\bbS, \R)$
be a map that projects functions defined on $Q_\eps$ to functions defined on $\bbS$ and is defined by
\be
\lb{eq:2.1}
M_\eps \psi (\bx) := \dfrac{1}{\eps} \int_1^{1+\eps} r \psi(r \bx)\,dr, \qquad \bx \in \bbS.
\ee

\begin{remark}
\lb{rem2.1}
We will use the Cartesian and spherical coordinates interchangeably in this paper.
For example, if $\bx \in \bbS$ then we will identify it by $\bx = (\lambda, \varphi)$
where $\lambda \in [0, \pi]$ and $\varphi \in [0, 2 \pi)$.
\end{remark}

\begin{lemma}
\lb{lemma2.2}
The map $M_\eps$ as defined in \eqref{eq:2.1} is continuous (and linear) w.r.t norms $L^2(Q_\eps)$ and $L^2(\bbS)$. Moreover,
\begin{equation}
\lb{eq:2.2}
\|M_\eps \psi\|_{L^2(\bbS)}^2 \le \frac{1}{\eps}\|\psi\|^2_{L^2(Q_\eps)}, \qquad \psi \in L^2(Q_\eps).
\end{equation}
\end{lemma}

\begin{proof}
Take $\psi \in \ccal({Q}_\eps)$ then by the definition of $M_\eps$ we have
\[M_\eps \psi(\bx) = \dfrac{1}{\eps} \int_1^{1+\eps} r \psi(r\bx)\,dr.\]
Thus, using the Cauchy-Schwarz inequality we have
\begin{align*}
\|M_\eps \psi\|_{L^2(\bbS)}^2 = \int_\bbS |M_\eps \psi(\bx)|^2\,d\sigma(\bx) &= \int_\bbS \left|\dfrac{1}{\eps} \int_{1}^{1+\eps} r \psi(r \bx)\,dr\right|^2\,d\sigma(\bx) \\
& \le \dfrac{1}{\eps^2} \int_\bbS\left( \inteps r^2 |\psi(r \bx)|^2\,dr \inteps dr \right)\,d \sigma(\bx) \\
& = \dfrac{1}{\eps^2} \cdot \eps \inteps \int_\bbS r^2 |\psi(r \bx)|^2\,dr\,d\sigma(\bx) \\
& = \frac{1}{\eps} \|\psi\|^2_{L^2(Q_\eps)},
\end{align*}
where the last equality follows from the fact that
\[\int_{Q_\eps} d\by = \inteps \int_\bbS r^2\,dr\,d\sigma(\bx)\]
is the volume integral over the spherical shell $Q_\eps$ in spherical coordinates, with \[d\sigma(\bx) = \sin{\lambda}\,d\lambda\,d\varphi,\]
being the Lebesgue measure over a unit sphere. Therefore, we obtain
\begin{equation}
\lb{eq:2.3}
\|M_\eps \psi\|^2_{L^2(\bbS)} \le \frac{1}{\eps}\|\psi\|^2_{L^2(Q_\eps)},
\end{equation}
and hence the map is bounded and we can infer \eqref{eq:2.2}.
\end{proof}

\begin{corollary}
\lb{cor2.3}
The map $M_\eps$ as defined in \eqref{eq:2.1} has a unique extension,
which without the abuse of notation will be denoted by the same symbol $M_\eps \colon L^2(Q_\eps) \to L^2(\bbS)$.
\end{corollary}
\begin{proof}
Since $\ccal(Q_\eps)$ is dense in $L^2(Q_\eps)$ and $M_\eps \colon L^2(Q_\eps) \to L^2(\bbS)$ is a bounded map thus by the Riesz representation theorem there exists a unique extension.
\end{proof}

\begin{lemma}
\label{lemma2.4}
The following map
\begin{align}
\label{eq:newReps}
\newReps \colon L^2(\bbS) \ni \psi \mapsto \frac{1}{|\cdot|} \psi\left(\frac{\cdot}{|\cdot|}\right) \in L^2(Q_\eps)
\end{align}
is bounded and
\[\|\newReps\|^2_{\mathcal{L}(L^2(\bbS), L^2(Q_\eps))} = \eps.\]
\end{lemma}

\begin{proof}
It is sufficient to consider $\psi \in \ccal(\bbS)$. For $\psi \in \ccal(\bbS)$, we have
\[\|\newReps \psi\|_{L^2(Q_\eps)}^2 = \int_{Q_\eps}|\v(\by)|^2\,d\by = \inteps r^2 \int_{\bbS} |\v(r\bx)|^2\,d\sigma(\bx)\,dr,\]
where $\v(\by) = \frac{1}{|\by|}\psi\left(\frac{\by}{|\by|}\right)$. But, for $\bx \in \bbS$
\[\v(r\bx) = \frac{1}{|r\bx|}\psi\left(\frac{r\bx}{|r\bx|}\right) = \frac{1}{r}\psi(\bx).\]
So
\begin{align*}
\|\newReps \psi\|_{L^2(Q_\eps)}^2 & = \inteps r^2 \frac{1}{r^2} \int_\bbS |\psi(\bx)|^2\,d \sigma(\bx)\,dr \\
& = \eps \|\psi\|^2_{L^2(\bbS)},
\end{align*}
thus, showing that the map $\newReps$ is bounded w.r.t. $L^2(\bbS)$ and $L^2(Q_\eps)$ norms.
\end{proof}

\begin{lemma}
\label{lem:grad_ret}
Let $\psi \in W^{1,p}(\bbS)$ for $p \ge 2$. Then for $\eps \in (0,1)$ there exists a constant $C > 0$ independent of $\eps$ such that
\[
\|\nabla \newReps \psi\|^p_{L^p(Q_\eps)} \le C \eps \| \psi\|_{W^{1,p}(\bbS)}^p.
\]
\end{lemma}

\begin{proof}
By the definition of the map $\newReps$ (see \eqref{eq:newReps}), and identities \eqref{eq:grad}, \eqref{eq:C.1.1} for the scalar function $\psi \in W^{1,p}(\bbS)$, we have for $Q_\eps \ni \by = r\bx$, $r \in (1,1+\eps)$ and $\bx \in \bbS$,
\begin{align*}
    \nabla (\newReps[\psi](\by)) & = \frac{\pa}{\pa r} \left(\newReps[\psi](\by)\right) \widehat{e_r} + \frac{1}{r} \frac{\pa}{\pa \lambda} \left(\newReps[\psi](\by)\right) \widehat{e_\lambda} + \frac{1}{r \sin \lambda} \frac{\pa}{\pa \varphi} \left(\newReps[\psi](\by)\right) \widehat{e_\varphi} \\
    & = \frac{\pa}{\pa r} \left(\frac{\psi(\bx)}{r}\right)\widehat{e_r} + \frac{1}{r} \frac{\pa}{\pa \lambda} \left(\frac{\psi(\bx)}{r}\right) \widehat{e_\lambda} + \frac{1}{r \sin \lambda} \frac{\pa}{\pa \varphi} \left(\frac{\psi(\bx)}{r}\right) \widehat{e_\varphi} \\
    & = - \frac{1}{r^2} \psi(\bx) \widehat{e_r} + \frac{1}{r^2} \frac{\pa}{\pa \lambda} \left(\psi(\bx)\right) \widehat{e_\lambda} + \frac{1}{r^2 \sin \lambda} \frac{\pa}{\pa \varphi} \left(\psi(\bx)\right) \widehat{e_\varphi}.
\end{align*}
Hence,
\begin{align*}
    \|\nabla \newReps \psi\|^p_{L^p(Q_\eps)} & = \int_{Q_\eps} |\nabla(\newReps[\psi](\by))|^p\,d\by \\
    & = \inteps \int_\bbS \frac{1}{r^{2p}}\left(|\psi(\bx)|^p + |\nablaS \psi(\bx)|^p\right)r^2 d\sigma(\bx)dr \\
    & = -\frac{r^{3 - 2p}}{2p - 3}\bigg\vert^{1+\eps}_1\left(\|\psi\|^p_{L^p(\bbS)} + \|\nablaS \psi\|^p_{L^p(\bbS)}\right)\\
    & \le C(p) \eps \|\psi\|^p_{W^{1,p}(\bbS)}.
\end{align*}
\end{proof}

\begin{lemma}
\label{lem:lap_ret}
Let $\psi \in H^2(\bbS)$. Then for $\eps \in (0,1)$
\begin{equation}
    \label{eq:lap_ret1}
    \|\Delta \newReps\psi \|^2_{L^2(Q_\eps)} \le \eps \|\Delta' \psi\|^2_{L^2(\bbS)}.
\end{equation}
\end{lemma}

\begin{proof}
Let $\psi \in H^2(\bbS)$, then
\[[\newReps \psi](\by)  = \frac{1}{|\by|} \psi\left(\frac{\by}{|\by|}\right), \quad \by \in Q_\eps.\]
Therefore, for every $Q_\eps \ni \by = r\bx$, $r \in (1, 1+\eps)$ and $x \in \bbS$, we have (see \eqref{eq:lap-bel} and \eqref{eq:C.1.2} for the definition of Laplace--Beltrami operator)
\begin{align*}
    \Delta([\newReps \psi](\by)) & = \frac{\partial^2}{\pa r^2} \left(\newReps \psi(\by)\right) + \frac{2}{r} \frac{\partial}{\pa r} \left(\newReps \psi (\by)\right) + \frac{1}{r^2} \Delta^\prime \left(\newReps \psi(\by)\right) \\
    & = \frac{\partial^2}{\partial r^2}\left(\frac{\psi(\bx)}{r}\right) + \frac{2}{r} \frac{\partial }{\partial r}\left(\frac{\psi(\bx)}{r}\right) + \frac{1}{r^2} \Delta^\prime \left(\frac{\psi(\bx)}{r}\right)\\
    & = \frac{2}{r^3}\psi(\bx) - \frac{2}{r^3}\psi(\bx) + \frac{1}{r^3}\Delta^\prime \psi(\bx)\\
    & = \frac{1}{r^3}\Delta^\prime \psi(\bx).
\end{align*}
Hence
\begin{align*}
    \|\Delta\left(\newReps \psi\right)\|^2_{L^2(Q_\eps)} &= \int_{Q_\eps} |\Delta\left([\newReps \psi](\by)\right)|^2 d\by = \inteps \int_\bbS \frac{1}{r^6} |\Delta^\prime \psi(\bx)|^2 r^2 d\sigma(\bx)\,dr \\
    & = -\frac{1}{3r^3}\Big|_{1}^{1+\eps}\|\Delta^\prime \psi\|^2_{L^2(\bbS)} = \frac{\left((1+\eps)^3 - 1\right)}{3(1+\eps)^3}\|\Delta^\prime \psi\|^2_{L^2(\bbS)} \\
    &= \frac{\eps^3 + 3\eps^2 + 3\eps}{(1+\eps)^3}\|\Delta^\prime \psi\|^2_{L^2(\bbS)}.
\end{align*}
Since $\eps \in (0,1)$, the inequality \eqref{eq:lap_ret1} holds.
\end{proof}

\begin{remark}
\label{rem_averaging_operator_dual}
It is easy to check that the dual operator $M_\eps^\ast \colon L^2(\bbS) \to L^2(Q_\eps)$ is given by
\begin{equation}
\label{eq:avgop_dual}
\left(M_\eps^\ast \psi \right)(\by) = \frac{1}{\eps}(R_\eps \psi)(\by),
\qquad \by \in Q_\eps .
\end{equation}
\end{remark}

Next we define another map
\begin{align}
\label{eqn-newMepshat}
\newMepshat &= \newReps \circ M_\eps\colon L^2(Q_\eps)  \to L^2(Q_\eps).
\end{align}
Courtesy of Corollary~\ref{cor2.3} and Lemma~\ref{lemma2.4}, $\newMepshat$ is well-defined and bounded. Using definitions of maps $R_\eps$ and $M_\eps$ , we have
\begin{align*}
\newMepshat \colon \psi \mapsto \left\{\by \mapsto \frac{1}{|\by|} \left(M_\eps \psi\right)\left(\frac{\by}{|\by|}\right)\right\}.
\end{align*}

\begin{lemma}
\lb{lemma2.6}
Let $\psi \in L^2(Q_\eps)$, then we have the following scaling property
\begin{equation}
\lb{eq:2.4}
\|\newMepshat \psi\|^2_{L^2(Q_\eps)} = \eps \|M_\eps \psi\|^2_{L^2(\bbS)}, \qquad \psi \in L^2(Q_\eps).
\end{equation}
\end{lemma}

\begin{proof}
Let $\psi \in L^2(Q_\eps)$. Then by the defintion of the map $\newMepshat$, we have
\begin{align*}
\|\newMepshat \psi\|^2_{L^2(Q_\eps)} & = \int_{Q_\eps}|\newMepshat \psi(\by)|^2\,d\by \\
& = \inteps \int_\bbS \frac{1}{r^2}|M_\eps \psi(\bx)|^2 r^2 d\sigma(\bx)\,dr = \eps \|M_\eps \psi\|^2_{L^2(\bbS)}.
\end{align*}
\end{proof}

The normal component of a function $\psi$ defined on $Q_\eps$ when projected to $\bbS$ is given by the map $\widehat{N}_\eps$ which is defined by
\begin{equation}
\label{eq:2.5}
\widehat{N}_\eps = {\rm{Id}} - \newMepshat,
\end{equation}
i.e.
\[
\widehat{N}_\eps \colon L^2(Q_\eps) \ni \psi \mapsto \psi - \newMepshat \psi \in L^2(Q_\eps).
\]
The following result establishes an important property of the map $\widehat{N}_\eps$.
\begin{lemma}
\label{lemma2.7}
Let $\psi \in L^2(Q_\eps)$, then
\begin{equation}
\label{eq:2.6}
\inteps r \widehat{N}_\eps \psi\,dr = 0, \qquad \mbox{a.e. on } \bbS.
\end{equation}
\end{lemma}

\begin{proof}
Let us choose and fix $\psi \in L^2(Q_\eps)$.  Then by the definitions of the operators involved we have the following equality in $L^2(\bbS)$:
\[\inteps r \widehat{N}_\eps \psi\,dr = \eps M_\eps \widehat{N}_\eps \psi.\]
Therefore ,we deduce that in order to prove equality \eqref{eq:2.6}, it is sufficient to show that
\[M_\eps \widehat{N}_\eps = 0.\]
Hence, by taking into account definitions \eqref{eq:2.5} of $\widehat{N}_\eps$ and \eqref{eqn-newMepshat} of
$\newMepshat$ , we infer that  it is sufficient to prove that
\[M_\eps = M_\eps \circ \newReps \circ M_\eps.\]
Let us choose $\psi \in \ccal(Q_\eps)$  and put  $\phi = M_\eps \psi$, i.e.
\[ \phi(\bx) = \frac{1}{\eps}\inteps r \psi(r \bx)\,dr. \]
Note that
\[\newReps \phi(\rho \bx) = \frac{1}{\rho} \phi(\bx), \quad \rho \in (1, 1+ \eps),\, \bx \in \bbS.\]
Thus, we infer that
\begin{align*}
M_\eps \left[\newReps \phi\right](\bx) & = \frac{1}{\eps}\inteps \rho \newReps \phi(\rho \bx)\,d\rho \\
& = \frac{1}{\eps}\inteps \rho \frac{1}{\rho} \phi(\bx)\,dr  = \phi(\bx) = M_\eps \psi (\bx), \;  \bx \in \bbS.
\end{align*}
Thus, we proved $M_\eps \circ \newReps\circ M_\eps \psi = M_\eps \psi$ for every $\psi \in \ccal({Q}_\eps)$.
Since $\ccal({Q}_\eps)$ is dense in $L^2(Q_\eps)$ and the maps $M_\eps$ and $M_\eps \circ \newReps\circ M_\eps$ are bounded in $L^2(Q_\eps)$, we conclude that we have  proved \eqref{eq:2.6}.
\end{proof}

\begin{lemma}
\label{lemma2.11}
For all $\psi, \xi \in L^2(Q_\eps)$, we have
\begin{equation}
\label{eq:2.13}
\left( \newMepshat \psi , \widehat{N}_\eps \xi \right)_{L^2(Q_\eps)}  = 0.
\end{equation}
\end{lemma}

\begin{proof}
Let $\psi, \xi \in L^2(Q_\eps)$, then
\begin{align*}
\left( \newMepshat \psi , \widehat{N}_\eps \xi \right)_{L^2(Q_\eps)} & = \int_{Q_\eps} \newMepshat \psi(\by) \cdot \widehat{N}_\eps \xi (\by)\,d \by \\
& = \int_\bbS \inteps \newMepshat \psi(r\bx) \cdot \widehat{N}_\eps \xi (r\bx) r^2\,dr\,d\sigma(\bx).
\end{align*}
By the definition \eqref{eqn-newMepshat}  of the map $\newMepshat$  and by Lemma~\ref{lemma2.7}, we infer that
\[\left( \newMepshat \psi , \widehat{N}_\eps \xi \right)_{L^2(Q_\eps)} = \int_\bbS M_\eps \psi (\bx) \underbrace{\left(\inteps r \widehat{N}_\eps \xi (r\bx)\,dr \right)}_{= 0\,\, \mbox{\tiny a.e. on } \bbS \mbox{ \tiny by Lemma~\ref{lemma2.7}}} \,d\sigma(\bx) = 0.\]
\end{proof}

Next we define projection operators for $\R^3$-valued vector fields using the above maps (for scalar functions), as follows $\colon$
\begin{align}
\label{eq:2.7}
\newMepstilde \colon \LQeps \ni u  = \left(u_r, u_\lambda, u_\varphi\right) & \mapsto \left(0, \newMepshat u_\lambda, \newMepshat u_\varphi \right) \in \LQeps,\\
\label{eq:2.8}
\widetilde{N}_\eps&= \rm{Id} -  \newMepstilde  \in \mathcal{L}( \LQeps).
\end{align}

\begin{lemma}
\label{lemma2.8}
Let $u \in \LQeps$. Then
\[\newMepstilde u \cdot \vec{n} = 0 \qquad \mbox{on }\, \partial Q_\eps.\]
Moreover, if $u$ satisfies the boundary condition $u \cdot \vec{n} = 0$,
then
\[\widetilde{N}_\eps u \cdot \vec{n} = 0 \qquad \mbox{on }\, \partial Q_\eps.\]
\end{lemma}

\begin{proof}
The normal vector $\vec{n}$ to $\partial Q_\eps$ is given by $\vecn = \left(1, 0, 0\right)$. Thus by the definition of $\newMepstilde$ we have
\[\newMepstilde u \cdot \vecn = 0 \qquad \mbox{on }\, \partial Q_\eps.\]
Now for the second part, from the definition of $\widetilde{N}_\eps$ we have
\[\widetilde{N}_\eps u \cdot \vecn = \underbrace{u \cdot \vecn}_{=\,0\,\, \mbox{\tiny from b.c.}} - \underbrace{\newMepstilde u\cdot \vecn}_{=\,0\,\, \mbox{\tiny from first part}} = 0.\]
\end{proof}

We also have the following generalisation of Lemma \ref{lemma2.7}.

\begin{lemma}
\label{lemma2.7b}
Let $u \in \mathbb{L}^2(Q_\eps)$, then
\begin{equation}
\label{eq:2.6b}
\inteps r \widetilde{N}_\eps u\,dr = 0, \qquad \mbox{a.e. on } \bbS.
\end{equation}
\end{lemma}

The following Lemma makes sense only for vector fields.

\begin{lemma}
\label{lemma2.9}
Let $u \in \rH_\eps$, then
\[\ddiv \newMepstilde u = 0 \quad \mbox{and} \quad \ddiv \widetilde{N}_\eps u = 0 \qquad \mbox{in }\, Q_\eps.\]
\end{lemma}

\begin{proof}
Let $u \in {\rm{U}}_\eps := \left\{\v \in \ccal^\infty({Q_\eps}) \colon \ddiv \v = 0\, \mbox{ in }\, Q_\eps\, \mbox{ and }\, \v \cdot \vec{n} = 0\, \mbox{ on } \partial Q_\eps \right\}$, then using the definition of divergence for a vector field $\v = \left(\v_r, \v_\lambda, \v_\varphi\right)$ in spherical co-ordinates (see \eqref{eq:diver}), we get for $Q_\eps \ni \by = r\bx$, $\bx \in \bbS$, $r \in (1, 1+\eps)$,
\begin{align}
\label{eq:2.9}
&\ddiv \left((\newMepstilde u) (\by)\right)  = 0 + \frac{1}{r \sin \lambda} \frac{\partial}{\partial \lambda} \left((\newMepshat u_\lambda)(\by) \sin \lambda\right) + \frac{1}{r \sin \lambda}\frac{\partial}{\partial \varphi} \left((\newMepshat u_\varphi)(\by)\right) \nn \\
&\quad = \frac{1}{r \sin \lambda} \left[\frac{\partial}{\partial \lambda}\left(\frac{1}{|\by|}(M_\eps u_\lambda)\left(\frac{\by}{|\by|}\right)\sin \lambda\right) + \frac{\partial}{\partial \varphi}\left(\frac{1}{|\by|}(M_\eps u_\varphi)\left(\frac{\by}{|\by|}\right)\right)\right]\nn \\
&\quad = \frac{1}{r \sin \lambda} \left[ \frac{\partial}{\partial \lambda} \left(\frac{\sin \lambda}{r \eps} \inteps \rho u_\lambda(\rho, \lambda, \varphi)\,d\rho \right) + \frac{\partial}{\partial \varphi} \left( \frac{1}{r \eps} \inteps \rho u_\varphi(\rho, \lambda, \varphi)\,d\rho \right)\right] \nn \\
&\quad =: \frac{1}{r^2 \sin \lambda}\left[ I + II \right].	
\end{align}
Now considering each of the terms individually, we have
\begin{equation}
\label{eq:2.10}
II = \frac{1}{ \eps} \frac{\partial}{\partial \varphi} \inteps \rho u_\varphi\,d\rho = \frac{1}{\eps} \inteps \rho \frac{\partial u_\varphi}{\partial\varphi}\,d\rho.
\end{equation}
\begin{align}
\label{eq:2.11}
I & = \frac{1}{\eps} \frac{\partial}{\partial \lambda} \left( \sin \lambda \inteps \rho u_\lambda\,d\rho \right) = \frac{1}{\eps} \frac{\partial}{\partial \lambda} \inteps \rho u_\lambda \sin \lambda\,d\rho \nn \\
& = \frac{1}{ \eps} \inteps \rho \frac{\partial }{\partial \lambda} \left(u_\lambda \sin \lambda\right)\,d\rho.
\end{align}
Using \eqref{eq:2.10} and \eqref{eq:2.11} in the equality \eqref{eq:2.9}, we obtain
\begin{equation}
\label{eq:2.12}
\ddiv \newMepstilde u = \frac{1}{r^2 \sin \lambda} \left( \frac{1}{\eps} \inteps \rho \left[\frac{\partial}{\partial \lambda} \left(u_\lambda \sin \lambda\right) + \frac{\partial u_\varphi}{\partial \varphi} \right]\,d\rho \right).
\end{equation}
Since $u \in \rU_\eps$, $\ddiv u = 0$ in $Q_\eps$, which implies
\[\frac{1}{\rho^2} \frac{\partial}{\partial \rho} \left(\rho^2 u_\rho \right) + \frac{1}{\rho \sin \lambda} \frac{\partial}{\partial \lambda} \left(u_\lambda \sin \lambda\right) + \frac{1}{\rho \sin \lambda} \frac{\partial u_\varphi}{\partial \varphi} = 0.\]
Using this in \eqref{eq:2.12}, we get
\begin{align*}
\ddiv \newMepstilde u & = - \frac{1}{r^2 \sin \lambda} \frac{1}{\eps} \inteps \rho \frac{\sin \lambda}{\rho} \frac{\partial}{\partial \rho}\left(\rho^2  u_\rho \right)\,d\rho \\
& = - \frac{1}{\eps r^2} \inteps \frac{\partial}{\partial \rho} \left(\rho^2 u_\rho\right)\,ds = - \frac{1}{\eps r^2} \rho^2 u_\rho\bigg|_1^{1+\eps} \\
& = -  \frac{1}{\eps r^2} \left[ (1+\eps)^2 u_\rho(1+\eps, \cdot, \cdot) - u_\rho(1, \cdot, \cdot) \right] = 0 \quad \mbox{(since $u\cdot \vec{n} = 0$)}.
\end{align*}
Thus, we have proved that $\ddiv \newMepstilde u = 0$, for every $u \in \rU_\eps$. Since, $\rU_\eps$ is dense in $\rH_\eps$, it holds true for every $u \in \rH_\eps$ too.
The second part follows from the definition of $\widetilde{N}_\eps$ and $\rH_\eps$.
\end{proof}

From Lemma~\ref{lemma2.8} and Lemma~\ref{lemma2.9}, we infer the following corollary$\colon$
\begin{corollary}
\lb{cor2.10}
If $u \in \rH_\eps$ then $\newMepstilde u$ and $\widetilde{N}_\eps u$ belong to $\rH_\eps$.
\end{corollary}

Using the definition of maps $\newMepstilde$ and $\tNe$ and Lemma~\ref{lemma2.11}, we conclude:

\begin{proposition}
\label{prop2.12}
For all $u, \v \in \LQeps$, we have
\begin{equation}
\label{eq:2.14}
\left(\newMepstilde u , \widetilde{N}_\eps \v \right)_{\LQeps} = 0.
\end{equation}
Moreover,
\begin{equation}
\label{eq:2.15}
\|{{u}}\|^2_{\LQeps} = \|\newMepstilde u\|^2_{\LQeps} + \|\widetilde{N}_\eps {{u}}\|^2_{\LQeps}, \qquad {{u}} \in \rH_\eps.
\end{equation}
\end{proposition}

Finally we define a projection operator that projects $\R^3$-valued vector fields defined on $Q_\eps$ to the ``tangent'' vector fields on sphere $\bbS$.
\begin{equation}
\label{eq:2.16}
\begin{split}
\ocirc{M}_\eps &\colon \LQeps \to \mathbb{L}^2(\bbS),\\
u &\mapsto \newMepstilde u \big|_\bbS = \left(0, M_\eps u_\lambda, M_\eps u_\varphi \right).
\end{split}
\end{equation}

\begin{lemma}
\label{lemma2.13}
Let $u \in \LQeps$, then
\[\|\newMepstilde u\|^2_{\LQeps} = \eps \|\ocirc{M}_\eps u\|^2_{\mathbb{L}^2(\bbS)}.\]
\end{lemma}
\begin{proof}
Let $u \in \LQeps$, then
\begin{align*}
\|\newMepstilde u\|^2_{\LQeps} & = \int_{Q_\eps} |\newMepstilde u (\by)|^2\,d\by \\
& = \int_{Q_\eps} \left( \frac{|M_\eps u_\lambda\left(\frac{\by}{|\by|}\right)|^2}{|\by|^2} + \frac{|M_\eps u_\varphi\left(\frac{\by}{|\by|}\right)|^2}{|\by|^2}\right)d\by \\
& = \inteps \int_{\bbS}r^2 \frac{|\ocirc{M}_\eps u(\bx)|^2}{r^2}\,d\sigma(\bx)\,dr
= \eps \|\ocirc{M}_\eps u\|^2_{\mathbb{L}^2(\bbS)}.
\end{align*}
\end{proof}

\begin{remark}
\label{rem_Mcirc_dual}
Similar to the scalar case, one can prove that the dual operator $\ocirc{M}^\ast_\eps : \LS \to \LQeps$ is given by
\begin{equation}
\label{eq:Mcirc_dual}
\left[\ocirc{M}_\eps^\ast u\right](\by) = \left(0, \left[M_\eps^\ast u_\lambda\right](\by) , \left[M_\eps^\ast u_\varphi\right](\by) \right), \quad \by \in Q_\eps.
\end{equation}
Indeed, for $u \in \LQeps, \v \in \LS$
\begin{align*}
\left(\ocirc{M}_\eps u, \v \right)_\LS & = \left(M_\eps u_\lambda, \v_\lambda\right)_{L^2(\bbS)} + \left(M_\eps u_\varphi, \v_\varphi\right)_{L^2(\bbS)} \\
& = \left(u_r, 0\right)_{L^2(Q_\eps)} + \left(u_\lambda, M_\eps^\ast \v_\lambda\right)_{L^2(Q_\eps)} + \left(u_\varphi, M_\eps^\ast \v_\varphi \right)_{L^2(Q_\eps)}.
\end{align*}
\end{remark}

Using the identities \eqref{eq:2.17}--\eqref{eq:2.19}, we can show that for a divergence free smooth vector field $\u$
\begin{equation}
\label{eq:2.20}
\left(-\Delta \u , \u \right)_{\LQeps} = \left(\curl \u, \curl \u \right)_{\LQeps} = \|\curl \u\|^2_{\LQeps}.
\end{equation}

We define a weighted $L^2$-product on $\rH_\eps$ by
\begin{equation}
\label{eq:2.21}
\left(u, \v \right)_r = \int_{Q_\eps} r^2\,u\cdot\v\,d\by, \qquad u, \v \in \rH_\eps,
\end{equation}
and the corresponding norm will be denoted by $\|\cdot\|_r$ which is equivalent to $\|\cdot\|_{\LQeps}$, uniformly for $\eps \in (0, \frac12)$ $\colon$
\begin{equation}
\label{eq:2.22}
\|u\|^2_{\LQeps} \le \|u\|_{r}^2 \le \frac94\|u\|^2_{\LQeps}, \qquad u \in \LQeps.
\end{equation}

We end this section by recalling a lemma and some Poincar\'e type inequalities from \cite{[TZ97]}.
\begin{lemma} \cite[Lemma~1.2]{[TZ97]}
\label{lemma2.14}
For $u, \v \in \rV_\eps$, we have
\[\left(\curl \newMepstilde u, \curl \widetilde{N}_\eps \v \right)_r = 0, \qquad u, \v \in \rV_\eps.\]
Moreover,
\begin{equation}
\label{eq:2.23}
\|\curl u\|^2_r = \|\curl \newMepstilde u\|^2_{r} + \|\curl \widetilde{N}_\eps u\|^2_r,\qquad u \in \rV_\eps.
\end{equation}
\end{lemma}

\begin{corollary}
\label{cor_grad_alpha}
Let $\eps \in (0, \frac12)$ and $u \in \rV_\eps$. Then
\begin{align*}
\|\curl \newMepstilde u \|_{\LQeps} & \le \frac94 \|\curl u \|^2_\LQeps,\\
\|\curl \tNe u\|_\LQeps &\le \frac94 \|\curl u\|^2_\LQeps.
\end{align*}
\end{corollary}
\begin{proof}
Let $\eps \in (0, \frac12)$ and $u \in \rV_\eps$. Then, by relation \eqref{eq:2.20}, equivalence of norms \eqref{eq:2.22} and Eq.~\eqref{eq:2.23}, we have
\begin{align*}
\|\curl u\|^2_\LQeps \ge \frac49 \|\curl u\|^2_r \ge \frac49 \|\curl \newMepstilde u\|^2_r \ge \frac49 \|\curl \newMepstilde u\|^2_\LQeps.
\end{align*}
The second inequality can be proved similarly.
\end{proof}
The following two lemmas are taken from \cite{[TZ97]}. For the sake of completeness and convenience of the reader we have provided the proof in Appendix~\ref{sec:appLadyzhenskaya}.
\begin{lemma}[Poincar\'e inequality in thin spherical shells]
\label{lemma2.15} \cite[Lemma~2.1]{[TZ97]}
For $0 < \eps \le \tfrac12$, we have
\begin{equation}
\label{eq:2.24}
\|\widetilde{N}_\eps u\|_{\LQeps} \le 2 \eps \|\curl \widetilde{N}_\eps u\|_{\LQeps},\qquad \forall\, u\, \in\, \rV_\eps.
\end{equation}
\end{lemma}

\begin{lemma}[Ladyzhenskaya's inequality]
\label{lemma2.16} \cite[Lemma~2.3]{[TZ97]}
There exists a constant $c_1$, independent of $\eps$, such that
\begin{equation}
\label{eq:2.25}
\|\widetilde{N}_\eps u\|_{\mathbb{L}^6(Q_\eps)} \le c_1 \|\widetilde{N}_\eps u\|_{\rV_\eps} ,\qquad \forall\, u\, \in\, \rV_\eps.
\end{equation}
\end{lemma}

\begin{corollary}
\label{cor_L3ineq}
For $\eps \in (0, \frac12)$, there exists a constant $c_2 > 0$ such that
\begin{equation}
\label{eq:2.26}
\|\widetilde{N}_\eps u\|^2_{\mathbb{L}^3(Q_\eps)} \le c_2 \eps \|\widetilde{N}_\eps u\|^2_{\rV_\eps} ,\qquad \forall\, u\, \in\, \rV_\eps.
\end{equation}
\end{corollary}
\begin{proof}
Let $u \in \rV_\eps$, then by the H\"older inequality, we have
\begin{align*}
\|\tNe u\|^3_{\mathbb{L}^3(Q_\eps)} & = \int_{Q_\eps}|\tNe u(\by)|^3\,d\by = \int_{Q_\eps}|\tNe u(\by)|^{3/2}|\tNe u(\by)|^{3/2}\,d\by \\
& \le \left(\int_{Q_\eps}|\tNe u(\by)|^6\,d\by\right)^{1/4}\left(\int_{Q_\eps}|\tNe u(\by)|^2\,d\by\right)^{3/4} \\
& = \|\tNe u\|^{3/2}_{\mathbb{L}^6(Q_\eps)}\|\tNe u\|^{3/2}_\LQeps.
\end{align*}
Thus, by Lemmas~\ref{lemma2.15} and \ref{lemma2.16}, we get
\begin{align*}
\|\tNe u\|^2_{\mathbb{L}^3(Q_\eps)} \le c_1\|\tNe u\|_{\rV_\eps} 2\eps \|\tNe u\|_{\rV_\eps} = c_2 \eps \|\tNe u\|_{\rV_\eps}^2.
\end{align*}
\end{proof}

In the following lemma we enlist some properties of operators $\newMepshat$, $\widehat{N}_\eps$, $\newMepstilde$ and $\widetilde{N}_\eps$.

\begin{lemma}
\label{lemma2.17}
Let $\eps > 0$. Then
\begin{itemize}
\item[i)] for $\psi \in L^2(Q_\eps)$
\begin{equation}
\label{eq:2.27}
\newMepshat \left(\newMepshat \psi\right) = \newMepshat \psi,
\end{equation}
\begin{equation}
\label{eq:2.28}
\widehat{N}_\eps \left(\widehat{N}_\eps \psi \right) = \widehat{N}_\eps \psi,
\end{equation}
\begin{equation}
\label{eq:2.29}
\newMepshat \left(\widehat{N}_\eps \psi\right) = 0, \qquad \mbox{and} \qquad \widehat{N}_\eps \left(\newMepshat \psi\right) = 0,
\end{equation}
\item[ii)] and for $u \in \LQeps$
\begin{equation}
\label{eq:2.30}
\newMepstilde \left(\newMepstilde u\right) = \newMepstilde u,
\end{equation}
\begin{equation}
\label{eq:2.31}
\widetilde{N}_\eps \left(\widetilde{N}_\eps u \right) = \widetilde{N}_\eps u,
\end{equation}
\begin{equation}
\label{eq:2.32}
\newMepstilde \left(\widetilde{N}_\eps u\right) = 0, \qquad \mbox{and} \qquad \widetilde{N}_\eps \left(\newMepstilde u\right) = 0.
\end{equation}
\end{itemize}
\end{lemma}
\begin{proof}
Let $\psi \in \ccal({Q}_\eps)$. Put
\[\phi = \newMepshat \psi,\]
i.e. for $\by \in Q_\eps$
\begin{align*}
\phi(\by) &  = \frac{1}{|\by|} \left(M_\eps \psi\right)\left(\frac{\by}{|\by|}\right) \\
& = \frac{1}{|\by|} \frac{1}{\eps} \inteps r \psi\left(r \frac{\by}{|\by|}\right)\,dr.
\end{align*}
Next for $\bz \in Q_\eps$
\begin{align*}
\vert \bz \vert  \left(\newMepshat \phi \right)(\bz) & = \left(M_\eps \phi \right)\left(\frac{\bz}{|\bz|}\right) = \frac{1}{\eps} \inteps \rho \phi\left(\rho \frac{\bz}{|\bz|}\right)\,d\rho \\
& = \frac{1}{\eps^2} \inteps \rho \frac{1}{\left|\rho \frac{\bz}{|\bz|}\right|} \left[ \inteps r \psi \left(r \frac{\rho \frac{\bz}{|\bz|}}{\left|\rho \frac{\bz}{|\bz|}\right|} \right)\,dr\right]\,d\rho \\
& = \frac{1}{\eps^2} \inteps \left[ \inteps r \psi \left(r \frac{\bz}{|\bz|}\right)\,dr \right]\,d\rho \\
& = \frac{1}{\eps} \inteps r \psi \left(r \frac{\bz}{|\bz|} \right)\,dr = \left(M_\eps \psi \right)\left(\frac{\bz}{|\bz|}\right) \\
& = \vert \bz \vert  \left(\newMepshat \psi \right)(\bz).
\end{align*}
Hence, we proved \eqref{eq:2.27} for every $\psi \in \ccal({Q}_\eps)$. Since $\ccal({Q}_\eps)$ is dense in $L^2(Q_\eps)$, it holds true for every $\psi \in L^2(Q_\eps)$.

\noindent \textbf{Proof of first part of} \eqref{eq:2.29}. Let $\psi \in \ccal({Q}_\eps)$.
Put $\phi = \widehat{N}_\eps \psi \in \ccal({Q}_\eps)$. By Lemma~\ref{lemma2.7}
\[\inteps r \phi(\by)\,dr = 0, \qquad \by \in Q_\eps.\]
Therefore, for $\by \in Q_\eps$,
\[\left(M_\eps \phi \right)\left(\frac{\by}{|\by|}\right) = \frac{1}{\eps}\inteps r \phi(\by)\,dr = 0.\]
Therefore, we infer that
\[\left(\newMepshat \phi \right)(\by) = \frac{1}{|\by|}\left(M_\eps \phi \right)\left(\frac{\by}{|\by|}\right) = 0,\]
for all $\by \in Q_\eps$. Thus, we have established first part of \eqref{eq:2.29} for all $\psi \in \ccal({Q}_\eps)$. Using the density argument, we can prove it for all $\psi \in L^2(Q_\eps)$.

Now for \eqref{eq:2.28}, by the definition of $\widehat{N}_\eps$ and \eqref{eq:2.29}, we obtain
\[\widehat{N}_\eps \left(\widehat{N}_\eps \psi \right) = \widehat{N}_\eps \psi - \newMepshat \left(\widehat{N}_\eps \psi\right) = \widehat{N}_\eps \psi.\]
Again using the definition of $\widehat{N}_\eps$ and Eq.\eqref{eq:2.27}, we have
\[\qquad \widehat{N}_\eps \left(\newMepshat \psi\right) = \newMepshat \psi - \newMepshat \left(\newMepshat \psi\right) = \newMepshat \psi - \newMepshat \psi =0,\]
concluding the proof of second part of \eqref{eq:2.29}.

\noindent \textbf{Proof of} \eqref{eq:2.30}. Let $u \in \ccal({Q}_\eps, \R^3)$. Write $u = \left(u_r, u_\lambda, u_\varphi\right)$. Put $\v = \newMepstilde u$, i.e.
\[\v = \left(0, \v_\lambda, \v_\varphi\right),\]
where
\[\v_\lambda = \newMepshat u_\lambda,\quad \mbox{and} \quad \v_\varphi = \newMepshat u_\varphi.\]
Thus, by the definition of $\newMepstilde$ and identity \eqref{eq:2.27}
\begin{align*}
\newMepstilde\left(\newMepstilde u \right) & = \newMepstilde \v  = \left(0, \newMepshat \v_\lambda, \newMepshat \v_\varphi\right) \\
 & = \left(0, \newMepshat \left(\newMepshat  u_\lambda\right), \newMepshat \left(\newMepshat  u_\varphi\right)\right) \\
& = \left(0, \newMepshat u_\lambda, \newMepshat u_\varphi\right) = \v = \newMepstilde u .
\end{align*}
We can extend this to $u \in \LQeps$ by the density argument. The remaining identities can be also established similarly as in the case of scalar functions.
\end{proof}

Later in the proof of Theorem~\ref{thm:main_thm}, in order to pass to the limit we will use an operator
$\oldReps$ defined by
\begin{equation}
\label{eq:2.36}
[\oldReps u](\by) = \left(0, [\newReps u_\lambda](\by), [\newReps u_\varphi](\by)\right), \quad Q_\eps \ni \by =r \bx
\end{equation}
where
\[\LS \ni u = \left(0, u_\lambda(\bx), u_\varphi(\bx)\right), \qquad \bx \in \bbS.\]
Using the definition of map $\newReps$ from Lemma~\ref{lemma2.4}, we can rewrite $\oldReps u$ as
\begin{equation}
    \label{eq:oldReps2}
    [\oldReps u](\by) = \left(0, \frac{1}{|\by|} u_\lambda\left(\frac{\by}{|\by|}\right), \frac{1}{|\by|}u_\varphi\left(\frac{\by}{|\by|}\right)\right).
\end{equation}
Note that $\oldReps$ is a bounded linear map from $\LS$ to $\LQeps$.

This operator $\oldReps$ is  retract of $\oMe$, i.e. a map $\oldReps : \LS \to \LQeps$ such that
\begin{equation}
\label{eq:2.35}
\oMe \circ \oldReps = \mathrm{Id} \; \mbox{ on } \LS.
\end{equation}

One can easily show that if $u \in \rD(\rA)$ then $\oldReps u \in \rD(\rA_\eps)$. In particular, for $u \in \rH$, $\oldReps u \in \rH_\eps$. Next we establish certain scaling properties for the map $\oldReps$.

\begin{lemma}
\label{lemma2.22}
Let $\eps > 0$, then
\begin{equation}
\label{eq:2.37}
\|\oldReps u\|_{\LQeps}^2 = \eps \|u\|_{\LS}^2, \quad u \in \LS.
\end{equation}
\end{lemma}

\begin{proof}
Let $\eps > 0$ and consider $\LS \ni u = \left(0, u_\lambda, u_\varphi\right)$. Then, by the definition of the retract operator $\oldReps$ and $\LQeps$-norm we have
\begin{align*}
\|\oldReps u\|^2_{\LQeps}& = \int_{Q_\eps}\left|\left[\oldReps u\right](\by)\right|^2 d\by = \int_{\bbS}\inteps \left[\frac{1}{|\by|^2}|u_\lambda(\bx)|^2 + \frac{1}{|\by|^2}|u_\varphi(\bx)|^2 \right] r^2\,dr\,d\sigma(\bx) \\
& = \int_{\bbS} \left(|u_\lambda(\bx)|^2 + |u_\varphi(\bx)|^2\right) d\sigma(\bx)\inteps dr = \eps \|u\|^2_{\LS}.
\end{align*}
\end{proof}
Using the definition of the map $\oldReps$ and Lemmas \ref{lem:grad_ret}, \ref{lem:lap_ret} we can deduce the following two lemmas (we provide the detailed proof of the latter in Appendix~\ref{sec:appLadyzhenskaya}):

\begin{lemma}
\label{lem:grad-ret-vec}
Let $u \in \bbW^{1,p}(\bbS)$ for $p \ge 2$. Then for $\eps \in (0,1)$ there exists a constant $C > 0$ independent of $\eps$ such that
\begin{equation}
    \label{eq:grad-ret-vec1}
    \|\nabla \oldReps u\|_{\lL^p(Q_\eps)} \le C(p)\eps^{1/p}\| u\|_{\bbW^{1,p}(\bbS)}.
\end{equation}
\end{lemma}

\begin{lemma}
\label{lem:lap-ret-vec}
Let $u \in \hH^2(\bbS) \cap \rH$ and $\eps \in (0,1)$. Then
\begin{equation}
    \label{eq:lap-ret-vec1}
    \|\Delta \oldReps u\|^2_{\LQeps} \le \eps\|\DDelta u\|^2_{\LS},
\end{equation}
where $\DDelta$ is defined in \eqref{eq:C.1.5}.
\end{lemma}


\section{Deterministic Navier--Stokes equations on thin spherical domains}
\label{sec:detNSE}

In this section we introduce the abstract form of the deterministic Navier--Stokes equations, i.e. the driving force $f^\eps$ is not random. The aim of this section is to obtain certain a priori estimates for $\alpha_\eps$ and  $\widetilde{\beta}_\eps$, as defined in \eqref{eq:notation}, using the a priori estimate for $u_\eps$, the solution of \eqref{eq:3.1}. The a priori estimates obtained in Lemma~\ref{lemma3.1} and Lemma~\ref{lemma3.2} will be used later along Theorem~\ref{thm4.1} to obtain certain strong convergence which will be used later to reprove Temam and Ziane's result \cite[Theorem~B]{[TZ97]}, Theorem~\ref{thm4.3}, regarding deterministic Navier--Stokes equations on thin spherical domains. We recall the result from \cite{[TZ97]} to show the reader that how natural it is to extend the proof from a deterministic setting to the stochastic framework.

The NSE on the thin spherical domain, using the operators from Section~\ref{sec:prelim}, can be written in the abstract form as
\begin{equation}
\label{eq:3.1}
\frac{d u_\eps}{dt} + \nu \rA_\eps u_\eps + B_\eps(u_\eps, u_\eps) = f^\eps, \qquad u_\eps(0) = u_0^\eps \qquad \mbox{on }\, Q_\eps,
\end{equation}
where $f^\eps \in L^2(0, T; V_\eps^\prime)$.

\begin{definition}
\label{defn_det_NSE_sol}
A function $u_\eps \in L^2(0,T; \rV_\eps) \cap L^\infty(0,T; \rH_\eps)$ is called a weak solution of \eqref{eq:3.1}, if $u_\eps(0) = u_0^\eps$ and for all $\v \in \rV_\eps$,
\begin{equation}
\label{eq:3.2}
\frac{d}{dt}(u_\eps, \v)_{\LQeps} + \nu \left\langle \rA_\eps u_\eps, \v \right\rangle_\eps + \left\langle B_\eps(u_\eps,u_\eps), \v \right\rangle_\eps = \left\langle f^\eps, \v \right\rangle_\eps
\end{equation}
holds.
\end{definition}

Moreover, the weak solution $u_\eps$ of \eqref{eq:3.1} satisfies the following energy inequality
\begin{equation}
\label{eq:3.3}
\|u_\eps(t)\|^2_{\LQeps} + \nu \int_0^t \|\curl u_\eps(s)\|^2_{\LQeps} \le \|u_0^\eps\|^2_{\LQeps} + \frac{1}{\nu} \int_0^t \|f^\eps(s)\|^2_{\rV_\eps^\prime}\,ds.
\end{equation}

Let $u_\eps = u_\eps(t),$ $t \ge 0$ be a weak solution of \eqref{eq:3.1}. We will use the following notations
\begin{equation}
\label{eq:notation}
\widetilde{\alpha}_\eps(t) := \newMepstilde [u_\eps(t)], \quad \widetilde{\beta}_\eps(t) := \widetilde{N}_\eps [u_\eps(t)], \quad \alpha_\eps(t) := \ocirc{M}_\eps [u_\eps(t)].
\end{equation}

\begin{lemma}
\label{lemma3.1}
Let $u_\eps$ be a weak solution of \eqref{eq:3.1} with $u_0^\eps \in \rH_\eps$ such that
\begin{equation}
\label{eq:3.4}
\|u_0^\eps\|_{\LQeps} \le \sqrt{\eps} C_1,
\end{equation}
for some $C_1 > 0$, $f^\eps \in L^2(0,T; \rV^\prime_\eps)$ and there exists a constant $C_2 > 0$ such that
\begin{equation}
\label{eq:3.5}
\int_0^T\|f^\eps(s)\|^2_{\rV^\prime_\eps}\,ds \le C_2 \eps.
\end{equation}
Then
\begin{equation}
\label{eq:3.6}
\|\ale(t)\|^2_{\LS} + \frac{4\nu}{9} \int_0^t \|\curlS \ale(s) \|^2_\LS \,ds \le C_1^2+ \frac{C_2}{\nu}, \quad t \ge 0 .
\end{equation}
\end{lemma}

\begin{proof}
Let $u_\eps$ be a weak solution of \eqref{eq:3.1}, satisfying the energy inequality \eqref{eq:3.3}. Let us define functions as in \eqref{eq:notation}. From Eq.~\eqref{eq:2.15}, we have
\begin{equation}
\label{eq:3.7}
\|\tale(t)\|^2_{\LQeps} \le \|u_\eps(t)\|^2_{\LQeps}, \qquad t \ge 0.
\end{equation}
Moreover, by Corollary~\ref{cor_grad_alpha}
\begin{equation}
\label{eq:3.8}
\frac49\|\curl \tale(t)\|^2_\LQeps \le \|\curl u(t)\|^2_\LQeps, \qquad t \ge 0.
\end{equation}
Therefore, using \eqref{eq:3.7} and \eqref{eq:3.8} in the energy inequality \eqref{eq:3.3}, we get
\[\|\tale(t)\|^2_{\LQeps} + \frac{4 \nu}{9} \int_0^t \|\curl \tale(s)\|^2_{\LQeps}\,ds \le \|u_0^\eps\|^2_{\LQeps} + \frac{1}{\nu} \int_0^t \|f^\eps(s)\|^2_{\rV_\eps^\prime}\,ds, \quad t \ge 0.\]
Hence from the scaling property in Lemma~\ref{lemma2.13}, we have
\begin{equation}
\label{eq:3.9}
\eps \|\ale(t)\|^2_{\LS} + \frac{4 \nu}{9} \eps \int_0^t \|\curlS \ale(s)\|^2_{\LS}\,ds \le \|u_0^\eps\|^2_{\LQeps} + \frac{1}{\nu} \int_0^t \|f^\eps(s)\|^2_{\rV_\eps^\prime}\,ds.
\end{equation}
Therefore, using assumptions \eqref{eq:3.4} and \eqref{eq:3.5} in \eqref{eq:3.9} and cancelling $\eps$ on both sides, we infer \eqref{eq:3.6}.
\end{proof}

From the results of Lemma~\ref{lemma3.1}, we deduce that
\begin{equation}
\label{eq:alpha_bounded}
\left\{ \ale \right\}_{\eps > 0}\, \mbox{ is bounded in }\, L^\infty(0,T; \rH) \cap L^2(0,T; \rV).
\end{equation}

Since $\rV$ can be embedded into $\mathbb{L}^6(\bbS)$, by using interpolation between $L^\infty(0,T; \rH)$ and $L^2(0,T; \mathbb{L}^6(\bbS))$ we obtain
\[\int_0^T \|\ale(s)\|^4_{\mathbb{L}^3(\bbS)}\,ds \le C.\]

\begin{lemma}
\label{lemma3.2}
Let $u_\eps$ be a weak solution of \eqref{eq:3.1} with $u_0^\eps \in \rH_\eps$ and $f^\eps \in L^2(0,T; \rV^\prime_\eps)$. Assume further that the conditions \eqref{eq:3.4} and \eqref{eq:3.5} hold. Then
\begin{equation}
\label{eq:3.13}
\|\tble(t)\|^2_{\LQeps} + \nu \int_0^t \|\curl \tble(s) \|^2_{\LQeps}\,ds \le C_1^2 \eps + \frac{C_2 \eps}{\nu} , \quad t \ge 0.
\end{equation}
\end{lemma}
\begin{proof}
Let $u_\eps$ be a weak solution of \eqref{eq:3.1}, then it satisfies the energy inequality \eqref{eq:3.3}. From \eqref{eq:2.15}, we have
\begin{equation}
\label{eq:3.14}
\|\tble(t)\|^2_{\LQeps} \le \|u_\eps(t)\|^2_{\LQeps}, \qquad t \ge 0.
\end{equation}
Moreover, by Corollary~\ref{cor_grad_alpha}, for every $t > 0$
\begin{equation}
\label{eq:3.15}
\frac49 \|\curl \tble(t) \|^2_\LQeps \le \|\curl u_\eps(t)\|^2_\LQeps.
\end{equation}
Therefore, using assumptions \eqref{eq:3.4}--\eqref{eq:3.5}, inequalities \eqref{eq:3.14} and \eqref{eq:3.15} in the energy inequality \eqref{eq:3.3}, we infer \eqref{eq:3.13}.
\end{proof}

Recall that if $u_\eps$ is a weak solution of \eqref{eq:3.1}, then it satisfies the Eq.~\eqref{eq:3.2}.
In particular, for $\phi \in \rD(\rA)$, using \eqref{eq:2.20} we have
\begin{equation}
\label{eq:3.16}
\frac{d}{dt}\left(u_\eps, \oldReps \phi\right)_{\LQeps} + \nu \left(\curl u_\eps, \curl \test\right)_{\LQeps} + \left\langle B_\eps(u_\eps,u_\eps), \test\right \rangle_\eps = \left\langle f^\eps, \test\right \rangle_\eps.
\end{equation}

In the next lemma, we will show that the function $\ale$, as introduced in \eqref{eq:notation}, taking values in $\rD(\rA^{-1})$ is equicontinuous, which along with the uniform bound \eqref{eq:alpha_bounded} gives us a strongly converging subsequence.

\begin{lemma}
\label{lemma3.3}
Let $u_\eps$ be a weak solution of \eqref{eq:3.1} and $\phi \in \rD(\rA)$, then there exists a constant $C(\eps)\,\colon$ $\lim_{\eps \to 0} C(\eps) = 0$, such that, for $\theta > 0$
\begin{equation}
\label{eq:3.17}
\left| \left(\tale(t + \theta) - \tale(t), \test \right)_\LQeps \right| \le \eps C(\eps)\theta^{1/2} \|\phi\|_{\rD(\rA)}.
\end{equation}
In particular,
\begin{equation}
\label{eq:3.18}
\left\| \ale (t+\theta) - \ale (t)\right\|_{\rD(\rA^{-1})} \le C(\eps) \theta^{1/2}.
\end{equation}
\end{lemma}

\begin{proof}
Let $u_\eps$ be a weak solution of \eqref{eq:3.1} and $\phi \in \rD(\rA)$. Then using the definition of the map $\oldReps$ and Lemma~\ref{lemma2.7b}, we have,
for all $t\in [0,T]$,
\begin{align}
\label{eq:3.19}
\left(\tNe u_\eps(t), \test \right)_{\LQeps} & = \int_\bbS \inteps r^2 \tNe u_\eps(t,r\bx) \phi(\bx)\frac{1}{r}\,dr\,d\sigma(\bx)  \\
& = \int_\bbS \phi(\bx) \left(\inteps r \tNe u_\eps(t,r \bx)\,dr\right)d\sigma(\bx) = 0. \nn
\end{align}
Similarly we have, also for all $t\in [0,T]$,
\begin{equation}
\label{eq:3.20}
\left \langle \tNe f^\eps(t), \test \right\rangle_\eps = 0.
\end{equation}
Thus, by Proposition~\ref{prop2.12}, equations \eqref{eq:3.16}, \eqref{eq:3.19} and \eqref{eq:3.20}, the weak solution $u_\eps$ satisfies the following equality
\begin{align}
\frac{d}{dt} \left(\tMe u_\eps, \test \right)_{\LQeps} = & - \nu \left(\curl \tMe u_\eps, \curl \test \right)_{\LQeps} - \nu \left(\curl \tNe u_\eps, \curl \test \right)_{\LQeps}
\nonumber
\\
\label{eq:3.21}
 & - \left \langle\left[\tMe u_\eps \cdot \nabla \right]\tMe u_\eps , \test \right \rangle_\eps - \left \langle \left[\tNe u_\eps \cdot \nabla \right]\tMe u_\eps , \test \right\rangle_\eps   \\
& - \left\langle\left[\tMe u_\eps \cdot \nabla \right]\tNe u_\eps , \test \right\rangle_\eps  - \left\langle\left[\tNe u_\eps \cdot \nabla \right]\tNe u_\eps, \test \right\rangle_\eps \nn \\
& + \left\langle \tMe f^\eps, \test \right\rangle_\eps. \nn
\end{align}
Now integrating both sides with respect to time in the interval $[t, t+\theta]$ and using the notations from \eqref{eq:notation}, we get
\begin{align}
\label{eq:3.23}
&\left(\tale(t+\theta) - \tale(t), \test \right)_{\LQeps} = -\nu \int_t^{t+\theta} \left(\curl \tale(s) , \curl \test \right)_{\LQeps}\,ds \\
& \quad - \nu \int_t^{t+\theta} \left(\curl \tble(s), \curl \test\right)_{\LQeps}\,ds - \int_t^{t+\theta} \left\langle\left[\tale(s) \cdot \nabla \right]\tale(s), \test \right\rangle_\eps\,ds  \nn \\
& \quad - \int_t^{t+\theta} \left\langle\left[\tble(s) \cdot \nabla \right]\tale(s), \test \right\rangle_\eps\,ds - \int_t^{t+\theta} \left\langle\left[\tale(s) \cdot \nabla \right]\tble(s), \test \right\rangle_\eps\,ds  \nn \\
& \quad - \int_t^{t+\theta} \left\langle\left[\tble(s) \cdot \nabla \right]\tble(s), \test \right\rangle_\eps\,ds + \int_t^{t+\theta} \left\langle \tMe f^\eps(s), \test \right\rangle_\eps\,ds. \nn
\end{align}
From Lemma~\ref{lem:lap-ret-vec}, we have for $\phi \in {\rD(\rA)}$,
\begin{equation}
\label{eq:3.24}
\left\|\Delta\left(\test\right)\right\|_\LQeps^2 \le C \eps \|\phi \|^2_{\rD(\rA)}.
\end{equation}
Now, in order to establish inequality \eqref{eq:3.17}, we will estimate each of the term in the righ hand side
of Eq.~\eqref{eq:3.23}. We will use the H\"older inequality, scaling property from Lemma~\ref{lemma2.13}, a priori estimates from Lemmas~\ref{lemma3.1}, \ref{lemma3.2} and the relation \eqref{eq:3.24}.\\
For the first term, we obtain
\begin{align}
\label{eq:3.25}
\left|\int_t^{t+\theta} \left(\curl \tale(s) , \curl \test \right)_{\LQeps}\,ds\right| & \le \int_t^{t+\theta} \left|\left(\tale(s), \Delta\left(\test\right)\right)_\LQeps\right|\,ds \nn \\
& \le \int_t^{t+\theta} \|\tale(s)\|_\LQeps \left\|\Delta\left(\test\right)\right\|_\LQeps\,ds\nn \\
& \le \eps C \int_t^{t+\theta} \|\ale(s)\|_\LS \|\phi\|_{\rD(\rA)}\,ds \nn \\
& \le  \eps C \left(\sup_{s \in [0,T]}\|\ale(s)\|^2_\LS\right)^{1/2} \|\phi	\|_{\rD(\rA)}\,\theta \nn\\
& \le \eps C \|\phi\|_{\rD(\rA)}\,\theta .
\end{align}
Similarly for the second term we have
\begin{align}
\label{eq:3.26}
\left|\int_t^{t+\theta} \left(\curl \tble(s) , \curl \test \right)_{\LQeps}\,ds\right| & \le \int_t^{t+\theta} \left|\left(\tble(s), \Delta\left(\test\right)\right)_\LQeps\right|\,ds \nn \\
& \le \int_t^{t+\theta} \|\tble(s)\|_\LQeps \left\|\Delta\left(\test\right)\right\|_\LQeps\,ds\nn \\
&\le \eps^{3/2} C \int_t^{t+\theta} \|\curl \tble(s)\|_{\LQeps}\|\phi\|_{\rD(\rA)}\,ds \nn\\
&\le \eps^{3/2} C \|\curl \tble\|_{L^2(0,T; \LQeps)}\|\phi\|_{\rD(\rA)}\,\theta^{1/2} \nn \\
& \le \eps^{3/2} C_\eps\|\phi\|_{\rD(\rA)}\, \theta^{1/2}.
\end{align}
Now we consider the non-linear terms. Since $\bbW^{2,2}(\bbS) \hookrightarrow \bbW^{1,6}(\bbS)$, using \eqref{eq:grad-ret-vec1} and the H\"older inequality repeatedly we have
\begin{align}
\label{eq:3.27}
\left|\int_t^{t+\theta} \left\langle\left[\tale(s) \cdot \nabla \right]\tale(s), \test \right\rangle_\eps\,ds\right| & = \left|\int_t^{t+\theta} \left(\tale(s), \left[\tale(s) \cdot \nabla \right]\test \right)_\LQeps\,ds\right| \nn \\
& \le  \int_t^{t+\theta} \|\tale(s)\|_\LQeps\|\tale(s)\|_{\mathbb{L}^3(Q_\eps)}\left\|\nabla\left(\test\right)\right\|_{\mathbb{L}^6(Q_\eps)} ds \nn \\
& \le \eps C \int_t^{t+\theta}\|\ale(s)\|_\LS\|\ale(s)\|_{\mathbb{L}^3(\bbS)}
\|\phi\|_{\bbW^{1,6}(\bbS)}\,ds \nn \\
& \le \eps C \|\ale\|_{L^\infty(0,T; \LS)} \|\ale\|_{L^4(0,T; \mathbb{L}^3(\bbS))}\|\phi\|_{\bbW^{2,2}(\bbS)}\,\theta^{3/4} \nn \\
& \le \eps C_\eps \|\phi\|_{\rD(\rA)}\,\theta^{3/4}.
\end{align}
Similarly for the next non-liner term, we have
\begin{align}
\label{eq:3.28}
\left|\int_t^{t+\theta} \left\langle\left[\tble(s) \cdot \nabla \right]\tale(s), \test \right\rangle_\eps\,ds\right| &= \left|\int_t^{t+\theta} \left(\tale(s), \left[\tble(s) \cdot \nabla \right]\test \right)_\LQeps\,ds\right| \nn \\
& \le  \int_t^{t+\theta} \|\tale(s)\|_\LQeps \|\tble(s)\|_{\mathbb{L}^3(Q_\eps)}\left\|\nabla\left(\test\right)\right\|_{\mathbb{L}^6(Q_\eps)}\,ds \nn \\
& \le \eps^{7/6} C \int_t^{t+\theta}\|\ale(s)\|_\LS\|\tble(s)\|_{\rV_\eps}\|\phi\|_{\bbW^{1,6}(\bbS)}\,ds \nn \\
& \le \eps^{7/6} C \|\ale\|_{L^\infty(0,T; \LS)}\|\tble\|_{L^2(0,T; \rV_\eps)}
\|\phi\|_{\bbW^{2,2}(\bbS)}\,\theta^{1/2} \nn \\
& \le \eps^{7/6} C_\eps \|\phi\|_{\rD(\rA)}\,\theta^{1/2}.
\end{align}
Now as in the previous case, for the second mixed non-linear term, we obtain
\begin{align}
\label{eq:3.29}
\left|\int_t^{t+\theta} \left\langle\left[\tale(s) \cdot \nabla \right]\tble(s), \test \right\rangle_\eps\,ds\right| & \le \eps^{7/6} C \|\ale\|_{L^\infty(0,T; \LS)} \|\tble\|_{L^2(0,T; \rV_\eps)}\|\phi\|_{\rD(\rA)}\,\theta^{1/2} \nn \\
& \le \eps^{7/6} C_\eps \|\phi\|_{\rD(\rA)}\,\theta^{1/2}.
\end{align}
Now for the last non-linear term, we get
\begin{align}
\label{eq:3.30}
\left|\int_t^{t+\theta} \left\langle\left[\tble(s) \cdot \nabla \right]\tble(s), \test \right\rangle_\eps\,ds\right| & = \left|\int_t^{t+\theta} \left(\tble(s), \left[\tble(s) \cdot \nabla \right]\test \right)_\LQeps\,ds\right| \nn \\
& \le  \int_t^{t+\theta} \|\tble(s)\|_\LQeps \|\tble(s)\|_{\mathbb{L}^3(Q_\eps)}\left\|\nabla\left(\test\right)\right\|_{\mathbb{L}^6(Q_\eps)}\,ds \nn \\
& \le \eps^{2/3} C \int_t^{t+\theta}\|\tble(s)\|_\LQeps \|\tble(s)\|_{\rV_\eps}\|\phi\|_{\bbW^{1,6}(\bbS)}\,ds \nn \\
& \le \eps^{2/3} C \|\tble\|_{L^\infty(0,T; \LQeps)}\|\tble\|_{L^2(0,T; \rV_\eps)}\|\phi\|_{\bbW^{2,2}(\bbS)}\,\theta^{1/2} \nn \\
& \le \eps^{5/3} C \|\phi\|_{\rD(\rA)}\,\theta^{1/2}.
\end{align}
We are now left to deal with the last term corresponding to the external forcing $f^\eps$. Using the definition of map $\ocirc{M}_\eps$, the H\"older inequality and assumption on $f^\eps$ (see \eqref{eq:3.5})
\begin{align}
\label{eq:3.31}
\left|\int_t^{t+\theta} \left\langle \tMe f^\eps(s), \test \right\rangle_\eps\,ds\right| &= \left|\int_t^{t+\theta} \left(\int_\bbS \inteps M_\eps f^\eps(s, r\bx)\phi(\bx)\,dr\,d\sigma(\bx)\right)\,ds\right| \nn \\
& = \eps  \left|\int_t^{t+\theta} \left(\int_\bbS \ocirc{M}_\eps f^\eps(s,\bx) \phi(\bx)\,d\sigma(\bx)\right)\,ds\right|\nn\\
& \le \eps \int_t^{t+\theta} \|\ocirc{M}_\eps f^\eps(s)\|_{\rV^\prime}\|\phi\|_\rV\,ds \nn \\
& \le \sqrt{\eps} C \|f^\eps\|_{L^2(0,T; \rV_\eps^\prime)}\|\phi\|_{\rD(\rA)}\,\theta^{1/2} \nn\\
& \le \eps C \|\phi\|_{\rD(\rA)}\,\theta^{1/2}.
\end{align}
By the Cauchy-Schwarz inequality and using \eqref{eq:3.25}--\eqref{eq:3.31} in \eqref{eq:3.23}, we infer the inequality \eqref{eq:3.17}.\\
Since
\[\left(\tale(t + \theta) - \tale(t), \test \right)_\LQeps = \eps \left(\ale(t + \theta) - \ale(t), \phi \right)_\LS,\]
we obtain
\begin{equation}
\label{eq:3.32}
\left|\left(\ale(t + \theta) - \ale(t), \phi \right)_\LS\right| \le C(\eps)\|\phi\|_{\rD(\rA)}\,\theta^{1/2}, \qquad \forall\, \phi\,\in {\rD(\rA)}.
\end{equation}
Hence, we can conclude \eqref{eq:3.18}.
\end{proof}


\section{Navier--Stokes equations on the sphere}
\label{sec:NSE_S}
In this section we introduce Navier--Stokes equations on the sphere and then using a result by Simon, see Theorem~\ref{thm4.1}, we reprove the result of Temam and Ziane, see \cite[Theorem~B]{[TZ97]}, also Theorem~\ref{thm4.3}.

The deterministic NSE on $\bbS$ is given by
\begin{equation}
\label{eq:4.6}
\begin{split}
\partial_t \widetilde{u} - \nu \DDelta \widetilde{u} + \left(\tu \cdot \nablaS\right)\tu + \nablaS \widetilde{p} = \widetilde{f} \qquad \mbox{in }\, \bbS\times (0,T), \\
\ddivS \tu = 0 \qquad \mbox{in }\, \bbS\times (0,T),
\end{split}
\end{equation}
with the initial data and boundary condition
\begin{equation}
\label{eq:4.7}
\tu (0, \bx) = \tu_0 \quad \mbox{in }\,\bbS,
\end{equation}
where $\tu = (0, \tu_\lambda, \tu_\varphi)$.

The space-time weak formulation of \eqref{eq:4.6} is the following$\colon$
\begin{equation}
\label{eq:4.8}
\langle \tu(t) , \phi\rangle = \langle \tu_0, \phi \rangle - \int_0^t \left[\nu \langle \curlS \tu (s), \curlS \phi \rangle + \langle (\tu(s) \cdot \nablaS)\tu(s), \phi \rangle + \langle f(s), \phi \rangle\right]\,ds,
\end{equation}
for all $\phi \in \rV$, where $\langle \cdot, \cdot \rangle$ has to be understood as the duality between $\rV^\prime$ and $\rV$.

\begin{definition}
\label{defn4.2} 
Let $\widetilde{u_0} \in \rH$ and $f \in \rV^\prime$. A function $\tu \in \ccal([0,T]; \rH) \cap L^2(0,T; \rV)$ is called a weak solution of \eqref{eq:4.6}, if for every $\phi \in \rV$, $\tu$ satisfies \eqref{eq:4.8}.
\end{definition}

The main result of this section is the following:

\begin{theorem}\label{thm4.3}
Let $T > 0$, and $u_\eps$ be the weak solution of \eqref{eq:3.1}. We assume that
\begin{equation}
\label{eq:4.9}
 \lim_{\eps \to 0} \oMe u_0^\eps = \widehat{u}_0 \quad  \mbox{ weakly in } \rH,
\end{equation}
and for every $t \in [0,T]$
\begin{equation*}
 \lim_{\eps \to 0} \oMe f^\eps(t) = \widehat{f}(t) \quad  \mbox{ weakly in } \rH.
\end{equation*}
Then there exists $\widehat{u}$ such that
\begin{equation}
\label{eq:4.10}
   \lim_{\eps \to 0} \ale = \widehat{u} \quad \mbox{ strongly in } {\ccal}([0,T];{\rV^\prime}) \cap L^2(0,T; \rH),
\end{equation}
and $\widehat{u}$ is the weak solution of \eqref{eq:4.6} with the initial condition $\widehat{u}(0) = \widehat{u}_0$ in $\bbS$.
\end{theorem}

\begin{proof}
Using Lemmas~ \ref{lemma3.1}, \ref{lemma3.3} and Theorem~\ref{thm4.1} for $p=\infty$ and $\rH \xhookrightarrow{\text{compact}} \rV^\prime \xhookrightarrow{} \rD(\rA^{-1})$, we conclude there exists a subsequence, still denoted by $\ale$, and a function $\widehat{u}$ such that
\[ \ale \rightarrow \widehat{u} \quad \mbox{ strongly in }\, \ccal([0,T]; \rV^\prime).\]
Applying Theorem~\ref{thm4.1} for $p=2$ and $\rV \xhookrightarrow{\text{compact}} \rH \xhookrightarrow{} {\rD(\rA^{-1})}$, using Lemmas~\ref{lemma3.1}, \ref{lemma3.3} and the embedding $L^\infty(0,T; {\rD(\rA^{-1})}) \subset L^2(0,T; {\rD(\rA^{-1})})$, we get another subsequence, still denoted by $\ale$, such that
\[\ale \rightarrow \widehat{u} \quad \text{ strongly in }\, L^2(0,T; \rH).\]
Now we will prove that $\widehat{u}$ is the weak solution of the deterministic
NSE on $\bbS$ \eqref{eq:4.6}, i.e. it satisfies \eqref{eq:4.8} for any $\phi \in \rV$.
From Lemma~\ref{lemma3.1} we have
\begin{equation}\label{eq:4.11}
\{\ale\}_{\eps > 0} \text{ is bounded in }\, L^\infty(0,T;\rH) \cap L^2(0,T;\rV).
\end{equation}
So, there exists a subsequence, still denoted by $\ale$, so that
\[ \ale \rightarrow \widehat{u} \quad \text{ weakly$^\ast$ in }\, L^\infty(0,T; \rH).\]
Therefore, by the assumption \eqref{eq:4.9}, we have
\begin{equation}
\label{eq:4.12}
\widehat{u}(0,\bx) = \widehat{u}_0, \quad \bx \in \bbS.
\end{equation}
Let us take $\phi = (0, \phi_\lambda, \phi_\varphi) \in \rU$ as the test function in \eqref{eq:3.21}, where
\[\rU := \left\{\v \in \ccal^\infty(\bbS) \colon \ddivS \v = 0\, \mbox{ in }\, \bbS\right\}\]
With the notations as in \eqref{eq:notation}, we have
\begin{align}
\label{eq:4.13}
&\left(\tale(t) , \test \right)_{\LQeps} = \left(\tMe u_0^\eps, \test \right)_{\LQeps} -\nu \int_0^{t} \left(\curl \tale(s) , \curl \test \right)_{\LQeps}\,ds \\
& \quad - \nu \int_0^{t} \left(\curl \tble(s), \curl \test\right)_{\LQeps}\,ds - \int_0^{t} \left\langle \left[\tale(s) \cdot \nabla \right]\tale(s), \test \right\rangle_\eps\,ds  \nn \\
& \quad - \int_0^{t} \left\langle\left[\tble(s) \cdot \nabla \right]\tale(s), \test \right\rangle_\eps\,ds - \int_0^{t} \left\langle\left[\tale(s) \cdot \nabla \right]\tble(s), \test \right\rangle_\eps\,ds  \nn \\
& \quad - \int_0^{t} \left\langle\left[\tble(s) \cdot \nabla \right]\tble(s), \test \right\rangle_\eps\,ds + \int_0^{t} \left\langle \tMe f^\eps(s), \test \right\rangle_\eps\,ds \nn.
\end{align}
Since
\begin{align*}
\left(\tale , \test \right)_{\LQeps} & = \eps \left(\ale, \phi \right)_{\LS} ,\\
\left(\tMe u_0^\eps, \test \right)_{\LQeps} & = \eps \left(\oMe u_0^\eps, \phi \right)_{\LS}, \\
\left(\curl \tale , \curl \test \right)_{\LQeps} & = \frac{\eps}{1 + \eps} \left(\curlS \ale, \curlS \phi \right)_{\LS},
\end{align*}
\begin{align*}
\left\langle \left[\tale(s) \cdot \nabla \right]\tale(s), \test \right\rangle_\eps & = \frac{\eps}{1+\eps}\left\langle\left[\ale \cdot \nablaS\right] \ale, \phi \right\rangle,\\
\left\langle\tMe f^\eps, \test \right\rangle_\eps & = \eps \left\langle\oMe f^\eps, \phi \right\rangle,
\end{align*}
we can rewrite Eq.~\eqref{eq:4.13} as follows
\begin{align}
\label{eq:4.14}
&\left(\alpha_\eps(t), \phi \right)_{\LS} =  \left(\oMe u_0^\eps, \phi \right)_{\LS} -\frac{\nu}{1+\eps} \int_0^t \left(\curlS \alpha_\eps(s), \curlS \phi \right)_{\LS}\,ds \\
&\quad  - \frac{1}{1+\eps} \int_0^{t} \left\langle\left[\alpha_\eps(s) \cdot \nablaS \right]\alpha_\eps(s), \phi \right\rangle\,ds + \int_0^{t} \left\langle\oMe f^\eps(s), \phi \right\rangle\,ds \nn \\
&\quad - \frac{\nu}{\eps} \int_0^{t} \left(\curl \tble(s), \curl \test\right)_{\LQeps}\,ds - \frac{1}{\eps} \int_0^{t} \left\langle\left[\tble(s) \cdot \nabla \right]\tale(s), \test \right\rangle_\eps\,ds  \nn \\
&\quad - \frac{1}{\eps} \int_0^{t} \left\langle\left[\tale(s) \cdot \nabla \right]\tble(s), \test \right\rangle_\eps\,ds  - \frac{1}{\eps} \int_0^{t} \left\langle\left[\tble(s) \cdot \nabla \right]\tble(s), \test \right\rangle_\eps\,ds \nn.
\end{align}
We will show that the last four terms of Eq.~\eqref{eq:4.14} converge to $0$ as $\eps \to 0$. Using \eqref{eq:3.26} and Lemma~\ref{lemma3.2} we have
\begin{align}
\label{eq:4.15}
\left|\frac{\nu}{\eps} \int_0^{t} \left(\curl \tble(s), \curl \test\right)_{\LQeps}\,ds \right| \le c \nu \eps^{1/2}\|\tble\|_{L^2(0,T; \rV_\eps)}\|\DDelta \phi\|_\LS \to 0.
\end{align}
Using the H\"older inequality, Corollary~\ref{cor_L3ineq} and the scaling property (Lemma~\ref{lemma2.13}), we have
\begin{align}
\label{eq:4.16}
\Big|\frac{1}{\eps} &\int_0^{t} \left\langle\left[\tble(s) \cdot \nabla \right]\tale(s), \test \right\rangle_\eps\,ds \Big| \le \frac{1}{\eps} \|\nabla \oldReps \phi\|_{L^6(Q_\eps)} \int_0^t \|\tale(s)\|_\LQeps \|\tble(s)\|_{\lL^3(Q_\eps)} ds \\
 &\le c \eps^{1/6} \|\ale\|_{L^\infty(0,T; \rH)}\|\tble\|_{L^2(0,T; \rV_\eps)}\|\phi\|_{\bbW^{1,6}(\bbS)} \to 0, \nn
\end{align}
where the convergence holds because of uniform bounds (in $\eps$) obtained in Lemmas~\ref{lemma3.1}, \ref{lemma3.2}.\\
Similarly, we get by Lemmas~\ref{lemma3.1} and \ref{lemma3.2}
\begin{align}
\label{eq:4.17}
\Big|\frac{1}{\eps}& \int_0^{t} \left\langle\left[\tale(s) \cdot \nabla \right]\tble(s), \test \right\rangle_\eps\,ds\Big| \\
 &\le c \eps^{1/6} \|\ale\|_{L^\infty(0,T; \rH)}\|\tble\|_{L^2(0,T; \rV_\eps)}\|\phi\|_{\bbW^{1,6}(\bbS)} \to 0.\nn
\end{align}
Now for the final term, using the H\"older inequality , the Poincar\'e inequality \eqref{eq:2.24} and Lemma~\ref{lemma3.2}, we have (see \eqref{eq:3.30} for explicit calculations)
\begin{align}
\label{eq:4.18}
\Big|\frac{1}{\eps} \int_0^{t} \left\langle\left[\tble(s) \cdot \nabla \right]\tble(s), \test \right\rangle_\eps\,ds \Big| \le c \eps^{2/3} \|\tble\|^2_{L^2(0,T; \rV_\eps)}\| \phi\|_{\bbW^{1,6}(\bbS)} \to 0.
\end{align}
For the remaining three terms, we have
\begin{align}
\label{eq:4.19}
& \frac{\nu}{1+\eps} \int_0^t \left(\curlS \alpha_\eps(s), \curlS \phi \right)_{\LS}\,ds  = - \frac{\nu}{1+\eps} \int_0^t \left(\alpha_\eps(s), \DDelta \phi \right)_\LS\,ds \\
& \rightarrow \qquad \quad - \nu \int_0^t \left(\widehat{u}, \DDelta \phi \right)_\LS\,ds = \nu \int_0^t \left(\curlS \widehat{u}, \curlS \phi \right)_\LS\,ds, \nn
\end{align}
as $\eps \to 0$, because of \eqref{eq:2.20} and \eqref{eq:4.10}. For the remaining non-linear term, we obtain
\begin{align}
\label{eq:4.20}
&\left| \int_0^{t} \left\langle\left[\alpha_\eps(s) \cdot \nablaS \right]\alpha_\eps(s), \phi \right\rangle\,ds - \int_0^t \left\langle\left[\widehat{u}(s) \cdot \nablaS \right]\widehat{u}(s), \phi \right\rangle\,ds\right| \\
&\le \left|\int_0^t \left\langle\left[\left(\alpha_\eps(s) - \widehat{u}(s) \right)\cdot \nablaS \right]\hu(s) , \phi \right\rangle\,ds\right|
+ \left|\int_0^t \left\langle\left[\ale(s)\cdot \nablaS \right]\left(\alpha_\eps(s) - \widehat{u}(s)\right) , \phi \right\rangle\,ds\right| \nn \\
& \le \|\alpha_\eps - \widehat{u}\|_{L^2(0,T;\rH)} \|\hu\|_{L^2(0,T;\rV)}\|\phi\|_{\lL^\infty(\bbS)} + \left|\int_0^t \left\langle\nablaS\left(\alpha_\eps(s) - \widehat{u}(s)\right) , \ale(s)\phi \right\rangle\,ds\right| \nn  \to 0
\end{align}
as $\eps \to 0$, due to strong convergence of $\ale \to \widehat{u}$ in $L^2(0,T; \rH)$, uniform boundedness of $\ale$ in $L^2(0,T; \rV)$ and weak convergence of $\ale \to \widehat{u}$ in $L^2(0,T; \rV)$. Now by assumption of the theorem, we have
\begin{equation}
\label{eq:4.21}
\int_0^{t} \left\langle\oMe f^\eps(s), \phi \right\rangle\,ds \to \int_0^{t} (\widehat{f}(s), \phi )_\LS\,ds.
\end{equation}
Hence, from \eqref{eq:4.12}, and \eqref{eq:4.15}--\eqref{eq:4.21}, the right hand side in \eqref{eq:4.14} converges to
\begin{align*}
\left(\widehat{u}_0, \phi\right)_{\LS} - \int_0^t\left[\nu \left(\curlS \widehat{u}(s), \curlS \phi \right)_{\LS} - \left\langle\left[\widehat{u}(s) \cdot \nablaS \right]\widehat{u}(s), \phi \right\rangle + (\widehat{f}(s), \phi )_\LS\,ds \right]\,ds,
\end{align*}
as $\eps \to 0$ ($t$ is fixed). By \eqref{eq:4.11} $\forall\,\phi \in \rU$ and $t \in [0,T]$
\[\lim_{\eps \to 0} \left(\alpha_\eps (t), \phi \right)_{\LS} = \left(\widehat{u}(t), \phi \right)_{\LS}.\]
Therefore, $\forall\, \phi \in \rU$ and $t \in [0,T]$
\begin{align*}
\left( \widehat{u}(t) , \phi \right)_{\LS} = & \left( \widehat{u}_0, \phi \right)_{\LS} - \int_0^t \left[\nu \left( \curlS \widehat{u} (s), \curlS \phi \right)_{\LS} - \left\langle (\widehat{u}(s) \cdot \nablaS)\widehat{u}(s), \phi \right\rangle\right]\,ds \\
&\;\; + \int_0^t (\widehat{f}(s), \phi )_\LS\,ds\,ds . \nn
\end{align*}
We conclude the proof using the density of $\rU$ in $\rV$ under $\hH^1(\bbS)$-norm.
\end{proof}

\section{Stochastic NSE on thin spherical domains}
\label{sec:SNSE_shell}

This section deals with the proof of our main result, Theorem~\ref{thm:main_thm}. First we introduce our two systems; stochastic NSE in thin spherical domain and stochastic NSE on the sphere, then we present the definition of martingale solutions for both systems. We also state the assumptions under which we prove our result. In Section~\ref{sec:estimates_shell_sphere}, we obtain a priori estimates (formally) which we further use to establish some tightness criterion (see Section~\ref{sec:tightness_shell}) which along with Jakubowski's generalisation of Skorokhod Theorem gives us a converging (in $\varepsilon$) subsequence. At the end of this section we show that the limiting object of the previously obtained converging subsequence is a martingale solution of stochastic NSE on the sphere (see Section~\ref{sec:proof-theorem}).

In thin spherical domain $Q_\eps$, which was introduced in \eqref{def:Qeps}, we consider the following stochastic Navier--Stokes equations (SNSE)
\begin{align}
 d \tue - [ \nu \Delta \tue - (\tue \cdot \nabla) \tue - \nabla \widetilde{p}_\eps]dt = \tfe dt +  \widetilde{G}_\eps\, d\tWe(t) &\quad \text{ in } Q_\eps \times (0,T), \label{eq:5.1}\\
\ddiv \tue = 0 &\quad \text{ in } Q_\eps \times (0,T),\label{eq:5.2} \\
\tue \cdot \vec{n} = 0, \quad \curl \tue \times \vec{n} = 0 &\quad\text{ on } \partial Q_\eps \times (0,T), \label{eq:5.3} \\
\tue(0, \cdot) = \tu_0^\eps &\quad\text{ in } Q_\eps. \label{eq:5.4}
\end{align}
Recall that, $\tue=(\tue^r, \tue^\lambda, \tue^\varphi)$ is the fluid velocity field, $p$ is the pressure, $\nu>0$ is a (fixed) kinematic viscosity, $\tu_0^\eps$ is a divergence free vector field on $Q_\eps$ and $\vecn$ is the unit normal outer vector to the boundary $\partial Q_\eps$. 
We assume that\footnote{We could have considered the case  $N = \infty$. This case will be considered in the companion paper \cite{[BDL19]}.} $N \in \N$. 
We consider a family of maps 
\[\widetilde{G}_\eps : \Rp   \to \mathcal{T}_2(\R^N; \rH_\eps)\]
such that
\begin{equation}
\label{eq:5.5}
\widetilde{G}_\eps(t) k  := \sum_{j=1}^N \widetilde{g}^j_\eps(t) k_j,\;\; k=(k_j)_{j=1}^N \in \mathbb{R}^N,
\end{equation}
for some  $\widetilde{g}_\eps^j:  \Rp  \to \rH_\eps$, $j=1, \cdots,N$. The Hilbert--Schmidt norm of $\widetilde{G}_\eps$ is given by
\begin{equation}
\label{eq:HSnorm}
\|\widetilde{G}_\eps(s)\|^2_{\mathcal{T}_2(\R^N; \rH_\eps)} = \sum_{j=1}^N \|\widetilde{g}^j_\eps(s)\|^2_{\LQeps}.
\end{equation}

Finally  we assume that  $\tWe(t)$, $t \ge 0$ is an $\R^N$-valued Wiener process defined  on the probability space $\left(\Omega, \mathcal{F}, \mathbb{F}, \mathbb{P}\right)$. We assume that 
 $\left(\beta_j\right)_{j=1}^N$ are i.i.d real valued Brownian motions such that  $W(t)=\bigl(\beta_j(t)\bigr)_{j=1}^N$, $t\geq 0$.

In this section, we shall establish convergence of the radial averages of the martingale solution of the 3D stochastic equations \eqref{eq:5.1}--\eqref{eq:5.4}, as the thickness of the shell $\eps \rightarrow 0$, to a martingale solution $u$ of the following stochastic Navier--Stokes equations on the sphere $\bbS$:
\begin{align}
du - \left[ \nu \DDelta u - (u \cdot \nablaS) u  - \nablaS p \right]dt = f dt + G\,dW(t) & \quad\text{ in } \bbS \times (0,T), \label{eq:5.6}\\
\ddivS u = 0 & \quad \text{ in } \bbS \times (0,T),\label{eq:5.7} \\
u(0, \cdot ) = u_0 & \quad\text{ in } \bbS,\label{eq:5.9}
\end{align}
where $u=(0, u_\lambda, u_\varphi)$ and $\DDelta$, $\nablaS$ are as defined in \eqref{eq:C.1.1}--\eqref{eq:C.1.5}. Assumptions on initial data and external forcing will be specified later (see Assumptions \ref{ass5.1},~\ref{ass_sphere}). Here, $G : \Rp \to \mathcal{T}_2(\R^N; \rH)$ and $W(t)$, $t \ge 0$ is an $\R^N$-valued Wiener process such that
\begin{equation}
\label{eq:wiener}
G(t)dW(t) := \sum_{j=1}^N g^j(t) d\beta_j(t),
\end{equation}
where $N \in \N$, $\left(\beta_j\right)_{j=1}^N$ are i.i.d real valued Brownian motions as before and $\left\{g^j\right\}_{j=1}^N$ are elements of $\rH$, with certain relation to $\tge^j$, which is specified later in Assumption~\ref{ass_sphere}.

\begin{remark}
We are aware of other formulations of the Laplacian in \eqref{eq:4.6} such as the one with an additional
Ricci tensor term \cite{[Serrin59],[Taylor92]}. However, as it was written in \cite[p. 144]{[Serrin59]}, 
``Deriving appropriate equations of motion involves dynamical considerations which do not seem adapted to 
Riemannian space; in particular it is not evident how to formulate the principle of conservation of momentum.''
Therefore, in this paper, we follow the approach presented in \cite{[TZ97]}, that the Navier--Stokes 
equations on the sphere is the thin shell limit of the 3-dimensional Navier--Stokes equations defined on 
a thin spherical shell.
\end{remark}

Now, we specify assumptions on the initial data $\widetilde{u}_0^\eps$ and external forcing $\widetilde{f}_\eps$, $\widetilde{g}_\eps^j$.
\begin{assumption}\label{ass5.1}
Let $\left(\Omega, \mathcal{F}, \mathbb{F}, \mathbb{P}\right)$ be the given filtered probability space. Let us assume that $p \ge 2$ and that  $\tu_0^\eps \in \rH_\eps$, for $\eps \in (0,1]$,   such that for some $C_1 > 0$
\begin{equation}
\label{eq:5.10}
\|\tu_0^\eps\|_{\LQeps} = C_1\eps^{1/2}, \qquad \eps \in (0,1].
\end{equation}
We also assume that  $\tfe \in L^p([0,T]; \rV_\eps^\prime)$, for $\eps \in (0,1]$,  such that for some $C_2 > 0$,
\begin{equation}
\label{eq:5.11}
\int_0^T\|\tfe(s)\|^p_{\rV_\eps^\prime}\,ds \le C_2 \eps^{p/2}, \qquad \eps \in (0,1].
\end{equation}
Let $\wtd{W}^\eps$ be an $\R^N$-valued Wiener process as before and assume that 
\[
\widetilde{G}_\eps \in L^p(0,T; \mathcal{T}_2(\R^N; \rH_\eps)), \;\;\mbox{for $\eps \in (0,1]$,}
\]
such that, using convention \eqref{eq:5.5}, for each $j=1,\ldots,N$,
\begin{equation}\label{eq:5.12}
\int_0^T\|\tge^j(t)\|^p_{\LQeps}dt = O(\eps^{p/2}),\;\; \eps \in (0,1].
\end{equation}
\end{assumption}

Projecting the stochastic NSE (on thin spherical shell) \eqref{eq:5.1}--\eqref{eq:5.4} onto $\rH_\eps$ using the Leray-Helmholtz projection operator and using the definitions of operators from Section~\ref{sec:functional_setting_shell}, we obtain the following abstract It\^o equation in $\rH_\eps$, $t \in [0,T]$
\begin{equation}
\label{eq:5.13}
d \tue (t) + \left[\nu \rA_\eps \tue(t) + B_\eps(\tue(t), \tue(t))\right]dt = \tfe(t)\,dt + \widetilde{G}_\eps(t)\,d \tWe(t), \quad \tue(0) = \tu_0^\eps.
\end{equation}

\begin{definition}
\label{defn5.2}
Let $\eps \in (0,1]$. A martingale solution to \eqref{eq:5.13} is a system
\[\left(\Omega, \calF, \bbF, \bbP, \wtd{W}_\eps, \wtd{u}_\eps\right)\]
where $\left(\Omega, \calF, \bbP \right)$ is a probability space and $\bbF = \left(\calF_t\right)_{t \ge 0}$ is a filtration on it, such that
\begin{itemize}
\item $\wtd{W}_\eps$ is a $\R^N$-valued Wiener process on $\left(\Omega, \calF, \bbF, \bbP\right)$,
\item $\wtd{u}_\eps$ is $\rV_\eps$-valued progressively measurable process, $\rH_\eps$-valued weakly continuous $\bbF$-adapted process such that\footnote{The space $X^w$ denotes a topological space $X$ with weak topology. In particular, $\ccal([0,T]; X^w)$ is the space of weakly continuous functions $\v:[0,T] \to X$.} $\bbP$-a.s.
\[
    \tue(\cdot, \omega) \in \ccal([0,T],{\rm H}_{\eps}^w) \cap L^2(0,T; \rV_\eps),
\]
\[\bbE \left[\frac{1}{2} \sup_{0\le s \le T} \|\tue(s) \|^2_{\LQeps}
  + \nu \int_0^T \| \curl \tue (s) \|^2_{\LQeps}\,ds    \right] < \infty\]
and, for all $t \in [0,T]$ and $\v \in \rV_\eps$, $\bbP$-a.s.,

\begin{align}
\label{eq:5.14}
& (\tue(t),\v)_{\LQeps} +  \nu \int_0^t (\curl \tue(s), \curl \v)_{\LQeps} \,ds + \int_0^t \left\langle B_\eps( \tue(s), \tue(s)), \v\right\rangle_\eps\,ds \\
&\quad = (\tu_0^\eps, \v)_{\LQeps} + \int_0^t \left\langle\tfe(s),\v\right\rangle_{\eps}\,ds + \left( \int_0^t \widetilde{G}_\eps(s)\,d\wtd{W}_\eps(s),\v \right)_\LQeps . \nn
\end{align}
\end{itemize}
\end{definition}

In the following remark we show that a martingale solution $\tue$ of \eqref{eq:5.13}, as defined above, satisfies an equivalent equation in the weak form.
\begin{remark}
\label{rem_mod_equation}
Recall the definition of the processes $\tale$, $\ale$ and $\tble$ from \eqref{eq:notation}, where
\[\tue(t) = \tale(t) + \tble(t), \qquad t \in [0,T].\]
Then, for $\phi \in \rD(\rA)$, we have
\begin{align*}
\left(\tale , \test \right)_{\LQeps} & = \eps \left(\ale, \phi \right)_{\LS} ,\\
\left(\tMe u_0^\eps, \test \right)_{\LQeps} & = \eps \left(\oMe u_0^\eps, \phi \right)_{\LS}, \\
\left(\curl \tale , \curl \test \right)_{\LQeps} & = \frac{\eps}{1 + \eps} \left(\curlS \ale, \curlS \phi \right)_{\LS}, \\
\left\langle\left[\tale \cdot \nabla \right]\tale, \test \right\rangle_\eps & = \frac{\eps}{1+\eps}\left\langle\left[\ale \cdot \nablaS\right]\ale, \phi \right\rangle,\\
\left\langle\tMe f^\eps, \test \right\rangle_\eps & = \eps \left\langle\oMe f^\eps, \phi \right\rangle,\\
\left(\sum_{j=1}^N \tMe \left( \tge^j d\beta_j(\cdot)\right), \test \right)_\LQeps & = \eps \left(\sum_{j=1}^N \oMe \left( \tge^j d\beta_j(\cdot)\right), \phi \right)_\LS,
\end{align*}
and using Lemma~\ref{lemma2.7}, Proposition~\ref{prop2.12} and Lemma~\ref{lemma2.14}, we can rewrite the weak formulation identity \eqref{eq:5.14} as follows.
\begin{align}
\label{eq:snse_mod}
&\left(\alpha_\eps(t), \phi \right)_{\LS} =  \left(\oMe \tu_0^\eps, \phi \right)_{\LS} -\frac{\nu}{1+\eps} \int_0^t \left(\curlS \alpha_\eps(s), \curlS \phi \right)_{\LS}\,ds \\
&\quad  - \frac{1}{1+\eps} \int_0^{t} \left\langle\left[\alpha_\eps(s) \cdot \nablaS \right]\alpha_\eps(s), \phi \right\rangle\,ds + \int_0^{t} \left\langle\oMe \tfe(s), \phi \right\rangle\,ds \nn \\
&\quad + \int_0^t \left(\sum_{j=1}^N \ocirc{M}_\eps \left(\tge^j(s)d \beta_j(s)\right), \phi \right)_\LS - \frac{\nu}{\eps} \int_0^{t} \left(\curl \tble(s), \curl \test\right)_{\LQeps}\,ds \nn \\ & \quad - \frac{1}{\eps} \int_0^{t} \left\langle\left[\tble(s) \cdot \nabla \right]\tale(s), \test \right\rangle_\eps\,ds  - \frac{1}{\eps} \int_0^{t} \left\langle\left[\tale(s) \cdot \nabla \right]\tble(s), \test \right\rangle_\eps\,ds \nn \\
&\quad  - \frac{1}{\eps} \int_0^{t} \left\langle\left[\tble(s) \cdot \nabla \right]\tble(s), \test \right\rangle_\eps\,ds, \nn
\end{align}
where $\langle \cdot, \cdot \rangle$ denotes the duality between $\rV^\prime$ and $\rV$.
\end{remark}

Next, we present the definition of a martingale solution for stochastic NSE on $\bbS$.

\begin{definition}
\label{defn_mart_SNSE_sphere}
A {martingale solution} to equation \eqref{eq:5.6}--\eqref{eq:5.9} is a system
\[\left(\widehat{\Omega}, \widehat{\calF}, \widehat{\bbF}, \widehat{\bbP}, \widehat{W}, \widehat{u}\right)\]
where $\left(\widehat{\Omega}, \widehat{\calF},\widehat{\bbP} \right)$ is a probability space and $\widehat{\bbF} = \left(\widehat{\calF}_t\right)_{t \ge 0}$ is a filtration on it, such that
\begin{itemize}
\item $\what{W}$ is an $\R^N$-valued Wiener process on $\left(\what{\Omega}, \what{\calF}, \what{\bbF}, \what{\bbP}\right)$,
\item $\what{u}$ is $\rV$-valued progressively measurable process, $\rH$-valued continuous $\what{\bbF}$-adapted process such that
\[
    \widehat{u}(\cdot, \omega) \in \ccal([0,T],{\rm H}) \cap L^2(0,T; \rV),
\]
\[\what{\bbE} \left[ \sup_{0\le s \le T} \|\hu(s) \|^2_{\LS}
  + \nu \int_0^T \| \curlS \hu (s) \|^2_{\LS}\,ds    \right] < \infty\]
and 
\begin{align}
\label{eq:mart_sol_sphere}
\left(\hu(t), \phi\right)_\LS + \nu \int_0^t \left(\curlS \hu(s), \curlS \phi \right)_\LS\,ds + \int_0^{t} \left\langle\left[\hu(s) \cdot \nablaS \right]\hu(s), \phi \right\rangle\,ds  \\
 = \left(u_0, \phi \right)_\LS + \int_0^{t} \left\langle f(s), \phi \right\rangle\,ds + \left(\int_0^t G(s) d \widehat{W}(s), \phi \right)_\LS, \nn
\end{align}
for all $t \in [0,T]$ and $\phi \in \rV$.
\end{itemize}
\end{definition}

\begin{assumption}\label{ass_sphere}
Let $p \ge 2$. Let $\left(\widehat{\Omega}, \widehat{\mathcal{F}}, \widehat{\mathbb{F}}, \widehat{\mathbb{P}}\right)$ be the given probability space, $\u_0 \in \rH$ such that
\begin{equation}
\label{eq:initial_data_convg}
\lim_{\eps\rightarrow 0} \oMe\widetilde{u}_0^\eps = u_0 \quad \text{weakly in } \rH.
\end{equation}
Let $f \in L^p([0,T]; \rV^\prime)$, such that for every
$s \in [0,T]$,
\begin{equation}
\label{eq:convg_ext_force}
\lim_{\eps\rightarrow 0} \langle \oMe \tfe(s), \v \rangle
  = \langle f(s), \v \rangle \text{ for all } \v \in \rV.
\end{equation}
And finally, we assume that $G \in L^p(0,T; \mathcal{T}_2(\R^N; \rH))$, such that for each $j=1,\ldots,N$ and $s \in [0,T]$, $\oMe \tge^j(s)$ converges weakly to $g^j(s)$ in $\mathbb{L}^2(\bbS)$ as $\eps \to 0$ and
\begin{equation}\label{eq:bdd_coeff}
\int_0^T\|g^j(t)\|^2_{\LS}dt \le M,
\end{equation}
for some $M > 0$.
\end{assumption}

\begin{remark}[{\bf Existence of martingale solutions}]
\label{rem_mart_sol_SNSE_shell}
In a companion  paper \cite{[BDL19]} we will address an easier  question about  the existence of a martingale solution for \eqref{eq:0.1}--\eqref{eq:0.4} in a more general setting with multiplicative noise. The key idea of the proof is taken from \cite{[BM13]}, where authors prove existence of a martingale solution for stochastic NSE in unbounded 3D domains.

The existence of a pathwise unique strong solution (hence a martingale solution) for the stochastic NSE on a sphere $\bbS$ is already established by two of the authors and Goldys in \cite{[BGL15]}. Through this article we give another proof of the existence of a martingale solution for such a system.
\end{remark}

We end this subsection by stating the main theorem of this article.

\begin{theorem}
\label{thm:main_thm}
Let the given data $\widetilde{u}_0^\eps$, $u_0$, $\tfe$, $f$, $\tge^j$, $g^j$, $j \in \{1, \cdots, N\}$ satisfy Assumptions~\ref{ass5.1} and \ref{ass_sphere}. Let $\left( \Omega, \mathcal{F}, \mathbb{F}, \mathbb{P}, \wtd{W}_\eps, \tue\right)$
be a martingale solution of \eqref{eq:5.1}--\eqref{eq:5.4} as defined in Definition~\ref{defn5.2}. Then, the averages in the radial direction of this martingale solution i.e. $\hae := \oMe[\tue]$ converge to a martingale solution, $\left(\widehat{\Omega}, \widehat{\mathcal{F}}, \widehat{\mathbb{F}}, \widehat{\mathbb{P}}, \widehat{W}, \what{u}\right)$, of \eqref{eq:5.6}--\eqref{eq:5.9} in $L^2(\widehat{\Omega}\times[0,T]\times \bbS)$.
\end{theorem}

\begin{remark}
\label{rem.one_prob_space}
According to Remark~\ref{rem_mart_sol_SNSE_shell}, for every $\eps \in [0,1]$ there exists a martingale solution of \eqref{eq:5.1}--\eqref{eq:5.4} as defined in Definition~\ref{defn5.2}, i.e. we will obtain a tuple $\left(\Omega_\eps, \mathcal{F}_\eps, \mathbb{F}_\eps, \mathbb{P}_\eps, \wtd{W}_\eps, \tue\right)$ as a martingale solution. It was shown in \cite{[Jakubowski98]} that is enough to consider only one probability space, namely,
\[\left(\Omega_\eps, \mathcal{F}_\eps, \mathbb{P}_\eps\right) = \left([0,1], \mathcal{B}([0,1]), \mathcal{L}\right) \qquad \forall\, \eps \in (0,1],\]
where $\mathcal{L}$ denotes the Lebesgue measure on $[0,1]$. Thus, it is justified to consider the probability space $\left( \Omega, \mathcal{F}, \mathbb{P}\right)$ independent of $\eps$ in Theorem~\ref{thm:main_thm}.
\end{remark}

\subsection{Estimates}
\label{sec:estimates_shell_sphere}
From this point onward we will assume that for every $\eps \in (0,1]$ there exists a martingale solution $\left(\Omega, \calF, \bbF, \bbP, \wtd{W}_\eps, \wtd{u}_\eps\right)$ of \eqref{eq:5.13}.
Please note that we do not claim neither we use the uniqueness of this solution. \\
The main aim of this subsection is to obtain estimates for $\alpha_\eps$ and $\widetilde{\beta}_\eps$ uniform in $\eps$ using the estimates for the process $\tue$.

The energy inequality \eqref{eq:5.15} and the higher-order estimates \eqref{eq:hoe1}--\eqref{eq:hoe2}, satisfied by the process $\tue$, as obtained in Lemma~\ref{lemma5.3} and Lemma~\ref{lemma_higher_estimates_u} is actually a consequence (essential by-product) of the existence proof. In principle, one obtains these estimates (uniform in the approximation parameter $N$) for the finite-dimensional process $\tue^{(N)}$ (using Galerkin approximation) with the help of the It\^o lemma. Then, using the lower semi-continuity of norms, convergence result ($\tue^{(N)} \to \tue$ in some sense),  one can establish the estimates for the limiting process. Such a methodology was employed in a proof of Theorem 4.8 in the   recent paper \cite{[BMO17]} by the first named author, Motyl and Ondrej\'at.

In Lemma~\ref{lemma5.3} and Lemma~\ref{lemma_higher_estimates_u} we present a formal proof where we assume that one can apply (ignoring the existence of Lebesgue and stochastic integrals) the It\^o lemma to the infinite dimensional process $\tue$. The idea is to showcase (though standard) the techniques involved in establishing such estimates.

\begin{lemma}
\label{lemma5.3}
Let $\tu_0^\eps \in \rH_\eps$, $\tfe \in L^2([0,T]; \rV^\prime_\eps)$ and $\widetilde{G}_\eps \in L^2([0,T]; \mathcal{T}_2(\R^N; \rH_\eps))$. Then, the martingale solution $\tue$ of \eqref{eq:5.13} satisfies the following energy inequality
\begin{equation}\label{eq:5.15}
\begin{split}
&\bbE \left[\frac{1}{2} \sup_{0\le s \le T} \|\tue(s)\|^2_{\LQeps}
  + \nu \int_0^T \| \curl \tue (s) \|^2_{\LQeps}\,ds    \right] \\
&\quad \le \|\tu^\eps_0\|^2_{\LQeps} + \frac{1}{\nu}  \int_0^T \|\tfe(s)\|^2_{\rV^\prime_\eps}\,ds + K \int_0^T \|\widetilde{G}_\eps(s)\|_{\mathcal{T}_2(\R^N; \rH_\eps)}^2\,ds,
\end{split}
\end{equation}
where $K$ is some positive constant independent of $\eps$.
\end{lemma}

\begin{proof}
Using the It\^{o} formula for the function $\|\xi\|^2_{\LQeps}$ with the process $\tue$, for a fixed $t \in [0,T]$ we have
\begin{equation}\label{eq:5.16}
\begin{aligned}
& \|\tue(t)\|^2_{\LQeps} + 2\nu \int_0^t \|\curl \tue(s)\|^2_{\LQeps}\,ds  \le  \|\tu_0^\eps\|^2_{\LQeps} + 2\int_0^t \left\langle \tfe(s), \tue(s)\right\rangle_\eps\,ds \\
&\qquad \qquad \qquad  + 2 \int_0^t \left(\widetilde{G}_\eps(s)\,d\wtd{W}_\eps(s), \tue(s)\right)_{\LQeps} + \int_0^t \|\widetilde{G}_\eps(s)\|^2_{\mathcal{T}_2(\R^N; \rH_\eps)}\, ds.
\end{aligned}
\end{equation}
Using the Cauchy-Schwarz inequality and the Young inequality, we get the following estimate
\[
\left|\left\langle \tfe, \tue\right\rangle_\eps\right| \le \|\tue\|_{\rV_\eps} \|\tfe\|_{\rV^\prime_\eps} \le \frac{\nu}{2}  \|\tue\|^2_{\rV_\eps}  + \frac{1}{2\nu} \|\tfe\|^2_{\rV^\prime_\eps},
\]
which we use in \eqref{eq:5.16}, to obtain
\begin{equation}\label{eq:5.17}
\begin{split}
\|\tue(t)\|^2_{\LQeps} & + \nu \int_0^t \|\curl \tue(s)\|^2_{\LQeps}\,ds  \le \|\tu_0^\eps\|^2_{\LQeps} + \frac{1}{\nu}\int_0^t \|\tfe(s)\|^2_{\rV^\prime_\eps}\,ds \\
&\, + 2 \int_0^t \left(\widetilde{G}_\eps(s)\,d\wtd{W}_\eps , \tue(s)\right)_{\LQeps}  + \int_0^t \|\widetilde{G}_\eps(s)\|^2_{\mathcal{T}_2(\R^N; \rH_\eps)}\,ds.
\end{split}
\end{equation}
Using the Burkholder-Davis-Gundy inequality (see \cite[Prop. 2.12]{[Kru14]}), we have
\begin{align}\label{eq:5.18}
\bbE \sup_{0\le t \le T}& \left| \int_0^t (\widetilde{G}_\eps(s)\,d\wtd{W}_\eps(s) , \tue(s)  )_{\LQeps} \right| \\
& \le C \bbE \left( \int_0^T \|\tue(s)\|^2_{\LQeps} \|\widetilde{G}_\eps(s)\|_{\mathcal{T}_2(\R^N; \rH_\eps)}^2\,ds\right)^{1/2} \nn  \\
&\le C \bbE \left[ \left(\sup_{0\le t \le T} \|\tue(t)\|^2_{\LQeps} \right)^{1/2}
   \left( \int_0^T \|\widetilde{G}_\eps(s)\|^2_{\mathcal{T}_2(\R^N; \rH_\eps)}\,ds\right)^{1/2} \right] \nn \\
&\le \frac{1}{4} \bbE\left( \sup_{0 \le t \le T} \| \tue(t)\|^2_{\LQeps} \right)
+ C\int_0^T \|\widetilde{G}_\eps(s)\|^2_{\mathcal{T}_2(\R^N; \rH_\eps)}\,ds. \nn
\end{align}
Taking the supremum of \eqref{eq:5.17} over the interval $[0,T]$, then taking
expectation and using inequality \eqref{eq:5.18} we infer the energy inequality \eqref{eq:5.15}.
\end{proof}

Let us recall the following notations, which we introduced earlier, for $t \in [0,T]$
\begin{equation}
\label{eq:5.19}
\widetilde{\alpha}_\eps(t) := \newMepstilde [\tue(t)], \quad \widetilde{\beta}_\eps(t) := \widetilde{N}_\eps [\tue(t)], \quad \alpha_\eps(t) := \ocirc{M}_\eps [\tue(t)].
\end{equation}

\begin{lemma}\label{lemma5.4}
Let $\tue$ be a martingale solution of \eqref{eq:5.13} and Assumption~\ref{ass5.1} hold, in particular, for $p =2$. Then
\begin{equation}\label{eq:5.20}
\begin{split}
&\bbE \left[ \frac{1}{2} \sup_{t \in [0,T]} \|\alpha_\eps(t)\|^2_{\LS} +
\nu \int_0^T \| \curlS \alpha_\eps(s)\|^2_{\LS}\,ds \right] \le C_1^2 + \frac{C_2}{\nu} + C_3,
\end{split}
\end{equation}
where $C_1,C_2$ are positive constants from \eqref{eq:5.10} and \eqref{eq:5.11} and $C_3 > 0$ (determined within the proof) is another constant independent of $\eps$.
\end{lemma}

\begin{proof}
Let $\tue$ be a martingale solution of \eqref{eq:5.13}, then it satisfies the energy inequality \eqref{eq:5.15}. From Eq.~\eqref{eq:2.15}, we have
\begin{equation}
\label{eq:5.21}
\|\tale(t)\|^2_{\LQeps} \le \|\tue(t)\|^2_{\LQeps}, \qquad t \in [0,T].
\end{equation}
Moreover, by Corollary~\ref{cor_grad_alpha}
\begin{equation}
\label{eq:5.22}
\frac49 \|\curl \tale(t)\|^2_\LQeps \le \|\curl \tue(t)\|^2_\LQeps, \qquad t \in [0,T].
\end{equation}
Therefore, using \eqref{eq:5.21} and \eqref{eq:5.22} in the energy inequality \eqref{eq:5.15}, we get
\begin{align*}
\bbE & \left[ \frac12 \sup_{t \in [0,T]} \|\tale(t)\|^2_{\LQeps} + \frac{4 \nu}{9} \int_0^T \|\curl \tale(s)\|^2_\LQeps\,ds\right]  \\
& \le \|\tu_0^\eps\|^2_{\LQeps} + \frac{1}{\nu} \int_0^T \|f_\eps(s)\|^2_{\rV_\eps^\prime}\,ds + K \int_0^T\|\widetilde{G}_\eps(s)\|^2_{\mathcal{T}_2}\,ds,
\end{align*}
and hence from the scaling property, Lemma~\ref{lemma2.13}, we have
\begin{align}
\label{eq:5.23}
\bbE & \left[\frac12 \eps \sup_{t \in [0,T]} \|\ale(t)\|^2_{\LS} + \frac{4 \nu}{9} \eps \int_0^T \|\curlS \ale(s)\|^2_{\LS}\,ds\right]  \\
& \le \|\tu_0^\eps\|^2_{\LQeps} + \frac{1}{\nu} \int_0^T \|f_\eps(s)\|^2_{\rV_\eps^\prime}\,ds + K \int_0^T\|\widetilde{G}_\eps(s)\|^2_{\mathcal{T}_2}\,ds.\nn
\end{align}
By the assumptions on $\tge^j$ \eqref{eq:5.12}, there exists a positive constant $c$ 
such that for every $j \in \{1, \cdots, N\}$
\begin{equation}
\label{eq:5.26}
\int_0^T\|\tge^j(t)\|_\LQeps^2\,dt \le c \eps.
\end{equation}
Therefore, using Assumption~\ref{ass5.1} and \eqref{eq:5.26} in \eqref{eq:5.23}, cancelling $\eps$ on both sides and defining $C_3 = NKc$, we infer inequality \eqref{eq:5.20}.
\end{proof}

From the results of Lemma~\ref{lemma5.4}, we deduce that
\begin{equation}
\label{eq:5.27}
\left\{ \ale \right\}_{\eps > 0}\, \mbox{ is bounded in }\,
L^2(\Omega;L^\infty(0,T; \rH) \cap L^2(0,T; \rV)).
\end{equation}
Since $\rV$ can be embedded into $\mathbb{L}^6(\bbS)$, by using interpolation between $L^\infty(0,T; \rH)$ and $L^2(0,T; \mathbb{L}^6(\bbS))$ we obtain
\begin{equation}
\label{eq:5.28}
\bbE \int_0^T \|\ale(s)\|^2_{\mathbb{L}^3(\bbS)}\,ds \le C.
\end{equation}

\begin{lemma}
\label{lemma5.5}
Let $\tue$ be a martingale solution of \eqref{eq:5.13} and Assumption~\ref{ass5.1} hold, in particular, for $p =2$. Then
\begin{equation}
\label{eq:5.29}
\bbE \left[ \frac12\sup_{t \in [0,T]}\|\tble(t)\|^2_{\LQeps} + \nu \int_0^T \|\curl \tble(s)\|^2_{\LQeps}\,ds\right] \le  C_1^2 \eps + \frac{C_2 \eps}{\nu} + C_3\eps.
\end{equation}
\end{lemma}
\begin{proof}
Let $\tue$ be a martingale solution of \eqref{eq:5.13}, then it satisfies the energy inequality \eqref{eq:5.15}. From \eqref{eq:2.15}, we have
\begin{equation}
\label{eq:5.30}
\|\tble(t)\|^2_{\LQeps} \le \|\tue(t)\|^2_{\LQeps}, \qquad t \in[0,T].
\end{equation}
Thus, by Corollary~\ref{cor_grad_alpha}
\begin{equation}
\label{eq:5.31}
\frac49 \|\curl \tble(t)\|^2_\LQeps \le \|\curl \tue(t)\|^2_\LQeps, \quad t \in [0,T].
\end{equation}
Therefore, using Assumption~\ref{ass5.1}, \eqref{eq:5.26}, inequalities \eqref{eq:5.30}--\eqref{eq:5.31}, in the energy inequality \eqref{eq:5.15}, we infer \eqref{eq:5.29}.
\end{proof}

In the following lemma we obtain some higher order estimates (on a formal level) for the martingale solution $\tue$, which will be used to obtain the higher order estimates for the processes $\ale$ and $\tble$.
\begin{lemma}
\label{lemma_higher_estimates_u}
Let Assumption~\ref{ass5.1} hold true and $\tue$ be a martingale solution of \eqref{eq:5.13}. Then, for $p > 2$ we have following estimates
\begin{equation}
\label{eq:hoe1}
\bbE\sup_{0\le s\le T} \|\tue(s)\|^p_{\LQeps}
 \le C_2\left(p, \widetilde{u}_0^\eps, \tfe,\widetilde{G}_\eps\right)\exp\left(K_p\, T\right)
\end{equation}
and
\begin{equation}
\label{eq:hoe2}
\E \int_0^T \|\tue(t)\|_{\LQeps}^{p-2}\|\tue(t)\|_{\rV_\eps}^2dt \le C_2\left(p, \widetilde{u}_0^\eps, \tfe,\widetilde{G}_\eps\right) \left[1 + K_{p}\,T\exp\left(K_{p}\,T\right) \right],
\end{equation}
where
\[
C_2\left(p, \widetilde{u}_0^\eps, \tfe, \widetilde{G}_\eps \right):=
\|\widetilde{u}_0^\eps\|^p_{\LQeps} + \nu^{-p/2}\|\tfe\|^p_{L^p(0,T; \rV'_\eps)}
   + \left(\frac{1}{4}p^2 (p-1) + \frac{K_1^2}{p}\right) \|\widetilde{G}_\eps\|^p_{L^p(0,T;\calT_2)},
\]
\[K_{p} := \left(\frac{K_1^2}{p}+p\right)\frac{(p-2)}{2},\]
and $K_1$ is a constant from the Burkholder-Davis-Gundy inequality.
\end{lemma}

\begin{proof}
Let $F(x) = \|x\|^p_{\LQeps}$ then
\[
  \frac{\partial F}{\partial x} = \nabla F
   = p \|x\|^{p-2}_{\LQeps} x,
\]
and
\begin{equation}\label{eqn:2}
 \left| \frac{\partial^2 F}{\partial x^2} \right|
  \le p(p-1) \|x\|^{p-2}_{\LQeps}.
\end{equation}
Applying the It\^{o} lemma with $F(x)$ and process $\widetilde{u}_\eps$ for
$t \in [0,T]$, we have
\begin{align*}
& \|\tue(t)\|^p_{\LQeps} = \|\tue(0)\|^p_{\LQeps}
  + p \int_0^t \|\tue(s)\|_{\LQeps}^{p-2}
\langle -\nu \rA_\eps \tue(s) - B_\eps(\tue(s),\tue(s))
+ \widetilde{f}_\eps(s), \tue(s) \rangle_\eps\, ds \\
&\quad + p \int_0^t \|\tue(s)\|^{p-2}_{\LQeps} \left( \tue(s), \widetilde{G}_\eps(s) d\widetilde{W}_\eps(s) \right)_{\LQeps}
+ \frac{1}{2}\int_0^t \mathrm{Tr}
\left( \frac{\partial^2 F}{\partial x^2} (\widetilde{G}_\eps(s),\widetilde{G}_\eps(s))\right)\,ds .
\end{align*}
Using the fact that
$\langle B_\eps(\tue,\tue), \tue \rangle_\eps = 0$ and
$ \langle A_\eps \tue,\tue \rangle_\eps = \|\tue\|_{\rV_\eps}^2 $ we arrive at
\begin{align*}
&\|\tue(t)\|^p_{\LQeps} = \|\tue(0)\|^p_{\LQeps}
 - p\nu \int_0^t \|\tue(s)\|_{\LQeps}^{p-2} \| \tue(s) \|_{\rV_\eps}^2
 + p \int_0^t \|\tue(s)\|_{\LQeps}^{p-2} \langle \widetilde{f}_\eps(s), \tue(s) \rangle_\eps\, ds \\
&\quad + p \int_0^t \|\tue(s)\|^{p-2}_{H_\eps}
 \left( \tue(s), \widetilde{G}_\eps(s) d\widetilde{W}_\eps(s)
\right)_{\LQeps}
+ \frac{1}{2}
\int_0^t \mathrm{Tr}
\left( \frac{\partial^2 F}{\partial x^2} (\widetilde{G}_\eps(s),\widetilde{G}_\eps(s))\right)ds.
\end{align*}
Using \eqref{eqn:2} and the Cauchy-Schwarz inequality, we get
\begin{align*}
&\|\tue(t)\|^p_{\LQeps}  + p \nu \int_0^t \|\tue(s)\|_{\LQeps}^{p-2} \|\tue(s)\|_{\rV_\eps}^2 ds\\
&\;\; \le \|\tue(0)\|^p_{\LQeps}
 + p \int_0^t \|\tue(s)\|_{\LQeps}^{p-2} \|\widetilde{f}_\eps(s)\|_{\rV'_\eps} \|\tue(s)\|_{\rV_\eps}\,ds \\
&\quad + p \int_0^t \|\tue(s)\|_{\LQeps}^{p-2} \left( \tue(s), \widetilde{G}_\eps(s) d\widetilde{W}_\eps(s) \right)_\LQeps + \frac{p(p-1)}{2}
\int_0^t \|\tue(s)\|_{\LQeps}^{p-2} \|\widetilde{G}_\eps(s)\|^2_{\calT_2(\R^N; \rH_\eps)}\, ds,
\end{align*}
where we recall
\[
\|\widetilde{G}_\eps(s)\|^2_{\calT_2(\R^N; \rH_\eps)} = \sum_{j=1}^N \|\tge^j(s)\|^2_{\LQeps}.
\]
Using the generalised Young inequality $abc \le a^q/q + b^r/r + c^s/s$ (where
$1/q+1/r+1/s=1$) with $a = \sqrt{\nu} \|\tue\|_{\LQeps}^{p/2-1} \|\tue\|_{\rV_\eps}$,
$b=\|\tue\|_{\LQeps}^{p/2-1}$, $c = \frac{1}{\sqrt{\nu}}\|f_\eps\|_{\rV'_\eps}$ and
exponents $q=2, r=p, s=2p/(p-2)$ we get
\begin{equation}\label{equ:3}
 \nu\|\tue\|_{\LQeps}^{p-2}
 \|f_\eps\|_{\rV'_\eps}
\|\tue\|_{\rV_\eps}
\le \frac{\nu}{2} \|\tue\|_{\LQeps}^{p-2} \|\tue\|_{\rV_\eps}^2
 + \frac{1}{p\nu^{p/2}} \|\widetilde{f}_\eps\|^p_{\rV'_\eps} + \frac{p-2}{2p} \|\tue\|_{\LQeps}^p.
\end{equation}
Again using the Young inequality with exponents $p/(p-2)$, $p/2$ we get
\begin{equation}\label{equ:4}
\|\tue\|_{\LQeps}^{p-2}\|\widetilde{G}
_\eps\|^2_{\calT_2(\R^N; \rH_\eps)}
\le \frac{p-2}{p} \|\tue\|_{\LQeps}^p + \frac{p}{2} \|\widetilde{G}_\eps\|^p_{\calT_2(\R^N; \rH_\eps)}.
\end{equation}
Using \eqref{equ:3} and \eqref{equ:4} we obtain
\begin{equation}\label{equ:5}
\begin{aligned}
&\|\tue(t)\|^p_{\LQeps}
 + \frac{p \nu}{2}
\int_0^t \|\tue(s)\|_{\LQeps}^{p-2} \|\tue(s)\|_{\rV_\eps}^2 ds
 \\
& \le \|\tue(0)\|^p_{\LQeps}
+ \frac{p(p-2)}{2}
\int_0^t \|\tue(s)\|_{\LQeps}^p\,ds
 + \nu^{-p/2} \int_0^t \|\widetilde{f}_\eps(s)\|^p_{\rV'_\eps}\, ds
\\
&\; + \frac{1}{4}p^2(p-1) \int_0^t \|\widetilde{G}_\eps(s)\|^p_{\calT_2(\R^N; \rH_\eps)}\, ds  + p \int_0^t \|\tue(s)\|_{\LQeps}^{p-2}
 \left( \tue(s),\widetilde{G}_\eps(s) d\widetilde{W}_\eps(s) \right)_\LQeps.
\end{aligned}
\end{equation}
Since $\tue$ is a martingale solution of \eqref{eq:5.13} it satisfies the energy inequality \eqref{eq:5.15}, hence the real-valued random variable
\[
\mu_\eps(t) = \int_0^t \|\tue(s)\|_{\LQeps}^{p-2} \left( \tue(s), \widetilde{G}_\eps(s) d\widetilde{W}_\eps(s) \right)_\LQeps
\]
is a $\mathcal{F}_t$-martingale. Taking expectation both sides of \eqref{equ:5}
we obtain
\begin{equation}
\begin{aligned}
& \bbE \|\tue(t)\|^p_{\LQeps}
+ \frac{p \nu}{2}
\bbE \int_0^t \|\tue(s)\|_{\LQeps}^{p-2} \|\tue(s)\|_{\rV_\eps}^2 ds \\
& \qquad \le
\|\tue(0)\|^p_{\LQeps}
+ \frac{p(p-2)}{2} \E \int_0^t \|\tue(s)\|_{\LQeps}^p ds
+ \nu^{-p/2} \int_0^t \|\widetilde{f}_\eps(s)\|^p_{\rV'_\eps}\,ds \\
& \qquad \quad + \frac{1}{4}p^2 (p-1) \int_0^t \|\widetilde{G}_\eps(s)\|^p_{\calT_2(\R^N; \rH_\eps)}\,ds.
\end{aligned}
\end{equation}
Therefore, by the Gronwall lemma we obtain
\[
\bbE\|\tue(t)\|^p_{\LQeps} \le C\left(\widetilde{u}_0^\eps, \tfe,\widetilde{G}_\eps\right)\exp\left( p\frac{(p-2)t}{2}\right),
\]
where
\[
C\left(\widetilde{u}_0^\eps, \tfe,\widetilde{G}_\eps\right) := \|\widetilde{u}_0^\eps\|_{\LQeps}^p
+ \nu^{-p/2} \|f\|^p_{L^p(0,T; \rV'_\eps)}
  + \frac{1}{4} p^2(p-1)
\| \widetilde{G}_\eps\|^p_{L^p(0,T;\calT_2(\R^N; \rH_\eps))} .
\]
By Burkholder-Davis-Gundy inequality, we have
\begin{align}\label{equ:9}
\bbE & \left(\sup_{0\le s \le t}\left| \int_0^s \|\tue(\sigma)\|^{p-2}_{\LQeps}
  \left( \tue(\sigma), \widetilde{G}_\eps(\sigma) d\widetilde{W}_\eps(\sigma) \right)_\LQeps \right|\right) \nonumber \\
& \le K_1 \bbE \left( \int_0^t \|\tue(s)\|_{\LQeps}^{2p-4}
 \|\tue(s)\|^2_{\LQeps} \|\widetilde{G}_\eps(s)\|^2_{\calT_2(\R^N; \rH_\eps)}\,ds \right)^{1/2} \nn \\
&\le K_1 \bbE \left[ \sup_{0\le s\le t} \|\tue(s)\|_{\LQeps}^{p/2}
 \left(\int_0^t \|\tue(s)\|_{\LQeps}^{p-2} \|\widetilde{G}_\eps(s)\|^2_{\calT_2(\R^N; \rH_\eps)} ds  \right)^{1/2}
 \right] \nn\\
&\le
 \frac{1}{2} \bbE \sup_{0\le s \le t} \|\tue(s)\|_{\LQeps}^p
+ \frac{K^2_1}{2} \bbE \int_0^t \|\tue(s)\|_{\LQeps}^{p-2} \|\widetilde{G}_\eps(s)\|^2_{\calT_2(\R^N; \rH_\eps)}\, ds \nn\\
&\le
 \frac{1}{2} \bbE \sup_{0\le s \le t} \|\tue(s)\|^p_{\LQeps}
+ \frac{K_1^2}{2} \frac{(p-2)}{p} \bbE \int_0^t \|\tue(s)\|_{\LQeps}^p ds +
\frac{K_1^2}{p} \int_0^t \|\widetilde{G}_\eps(s)\|^p_{\calT_2(\R^N; \rH_\eps)}\,ds
\end{align}
where in the last step we have used the Young inequality with exponents $p/(p-2)$
and $p/2$.

\noindent Taking supremum over $0\le s\le t$ in \eqref{equ:5}
and using \eqref{equ:9} we get
\begin{align}
\label{equ:10}
\frac{1}{2} \bbE & \sup_{0 \le s\le t} \|\tue(s)\|^p_{\LQeps}
+ \frac{\nu}{p} \bbE \int_0^t \|\tue(s)\|_{\LQeps}^{p-2} \|\tue(s)\|^2_{\rV_\eps} ds \\
&\le
\|\tue(0)\|^p_{\LQeps}
+ \left( \frac{K_1^2}{p} + p \right) \frac{(p-2)}{2}
\int_0^t \bbE \sup_{0 \le s \le \sigma} \|\tue(s)\|_{\LQeps}^p d\sigma  \nn \\
&\qquad + \nu^{-p/2} \int_0^t \|\widetilde{f}_\eps(s)\|^p_{\rV'_\eps} ds + \left(\frac{1}{4}p^2(p-1)
+ \frac{K_1^2}{p}\right) \int_0^t \|\widetilde{G}_\eps(s)\|^p_{\calT_2(\R^N; \rH_\eps)}\,ds. \nn
\end{align}
Thus using the Gronwall lemma, we obtain
\[
\bbE\sup_{0\le s\le t} \|\tue(s)\|^p_{\LQeps}
 \le C_2\left(p, \widetilde{u}_0^\eps, \tfe,\widetilde{G}_\eps\right)\exp\left(K_p\,t\right),
\]
where $C_2\left(p, \widetilde{u}_0^\eps, \tfe,\widetilde{G}_\eps\right)$ and $K_p$ are the constants as defined in the statement of lemma. We deduce \eqref{eq:hoe2} from \eqref{equ:10} and \eqref{eq:hoe1}.
\end{proof}

In the following lemma we will use the estimates from previous lemma to obtain higher order estimates for $\alpha_\eps$ and $\tble$.

\begin{lemma}
Let $p > 2$. Let $\tue$ be a martingale solution of \eqref{eq:5.13} and Assumption~\ref{ass5.1} hold with the chosen $p$. Then, the processes $\ale$ and $\tble$ (as defined in \eqref{eq:5.19}) satisfy the following estimates
\begin{equation}
\label{eq:hoe_alpha}
\E\sup_{t\in[0,T]}\|\ale(t)\|^p_{\LS} \le K(\nu, p)\exp\left(K_p\,T\right),
\end{equation}
and
\begin{equation}
\label{eq:hoe_beta}
\E\sup_{t\in[0,T]}\|\tble(t)\|^p_{\LQeps} \le \eps^{p/2}K(\nu, p)\exp\left(K_p\,T\right),
\end{equation}
where $K(\nu,p)$ is a positive constant independent of $\eps$ and $K_p$ is defined in Lemma~\ref{lemma_higher_estimates_u}.
\end{lemma}
\begin{proof}
The lemma can be proved following the steps of Lemma~\ref{lemma5.4} and Lemma~\ref{lemma5.5} with the use of Proposition~\ref{prop2.12}, scaling property from Lemma~\ref{lemma2.13}, the Cauchy-Schwarz inequality, Assumptions~\ref{ass5.1}, \ref{ass_sphere} and the estimates obtained in Lemma~\ref{lemma_higher_estimates_u}.
\end{proof}

\subsection{Tightness}
\label{sec:tightness_shell}
In this subsection we will prove that the family of laws induced by the processes $\alpha_\eps$ is tight on an appropriately chosen topological space $\mathcal{Z}_T$. In order to do so we will consider the following functional spaces for fixed $T > 0$:

\noindent $\ccal([0,T]; {\rD(\rA^{-1})}) :=$ the space of continuous functions $u: [0,T] \to {\rD(\rA^{-1})}$ with the topology $\mathbf{T}_1$ induced by the norm $\|u\|_{\ccal([0,T]; {\rD(\rA^{-1})})} := \sup_{t \in [0,T]}\|u(t)\|_{{\rD(\rA^{-1})}}$,

\noindent $L^2_{\mathrm{w}}(0,T; \rV) :=$ the space $L^2(0,T; \rV)$ with the weak topology $\mathbf{T}_2$,

\noindent $L^2(0,T; \rH) :=$ the space of measurable functions $u : [0,T] \to \mathrm{H}$ such that
\[\|u\|_{L^2(0,T; \mathrm{H})} = \left(\int_0^T \| u(t)\|_{\LS}^2\,dt \right)^{\frac12} < \infty,\]
with the topology $\mathbf{T}_3$ induced by the norm $\|u\|_{L^2(0,T; \mathrm{H})}$.

\noindent Let $\rH_\mathrm{w}$ denote the Hilbert space $\rH$ endowed with the weak topology.

\noindent $\ccal([0,T]; \mathrm{H_w}) :=$ the space of weakly continuous functions $u: [0,T] \to \mathrm{H}$ endowed with the weakest topology $\mathbf{T}_4$ such that for all $h \in \mathrm{H}$ the mappings
\[\ccal([0,T]; \mathrm{H_w}) \ni u \to \left( u(\cdot), h \right)_{\LS} \in \ccal([0,T]; \R)\]
are continuous. In particular $u_n \to u$ in $\ccal([0,T]; \rH_\mathrm{w})$ iff for all $h \in \rH \colon$
\[\lim_{n \to \infty} \sup_{t \in [0,T]} \left|\left( u_n(t) - u(t), h \right)_{\LS} \right| = 0.\]

Let
\begin{equation}
\label{eq:5.32}
\mathcal{Z}_T = \ccal([0,T]; \rD(\rA^{-1})) \cap L^2_{\mathrm{w}}(0,T; \rV) \cap L^2(0,T; \rH) \cap \ccal([0,T]; \mathrm{H_w}),
\end{equation}
and let $\mathcal{T}$ be the supremum\footnote{$\mathcal{T}$ is the supremum of topologies $\mathbf{T}_1$, $\mathbf{T}_2$, $\mathbf{T}_3$ and $\mathbf{T}_4$, i.e. it is the coarsest topology on $\mathcal{Z}_T$ that is finer than each of $\mathbf{T}_1$, $\mathbf{T}_2$, $\mathbf{T}_3$ and $\mathbf{T}_4$} of the corresponding topologies.

\begin{lemma}
\label{lemma5.6}
The set of measures $\left\{\mathcal{L}(\alpha_\eps),\, \eps \in (0,1]\right\}$ is tight on $\left(\mathcal{Z}_T, \mathcal{T}\right)$.
\end{lemma}
\begin{proof}
Let $\tue$, for some fixed $\eps>0$,  be a martingale solution of problem \eqref{eq:5.13}. Let us choose and fix   $\phi \in \rD(\rA)$.  Then, recalling the definition \eqref{eq:2.8} of the operator $\tNe$,  by  Lemma~\ref{lemma2.7b}, we infer that for $t\in [0,T]$ we have
\begin{align}
\label{eq:5.33}
\left(\tNe \tue(t), \test \right)_\LQeps & = \int_\bbS \inteps r^2 \tNe \tue(t,r\bx) \phi(\bx)\frac{1}{r}\,dr\,d\sigma(\bx)  \\
& = \int_\bbS \phi(\bx) \left(\inteps r \tNe \tue(t,r \bx)\,dr\right)d\sigma(\bx) = 0. \nn
\end{align}
Similarly we have, also for $t\in [0,T]$,
\begin{equation}
\label{eq:5.34}
\left(\tNe f^\eps(t), \test \right)_\LQeps = 0, \quad \mbox{ and } \quad \left(\tNe \left[\widetilde{G}_\eps(t) d\tWe(t) \right], \test \right)_\LQeps = 0.
\end{equation}
Thus, by Proposition~\ref{prop2.12}, identity \eqref{eq:5.14}, equalities \eqref{eq:5.33}, \eqref{eq:5.34}, and the notations from \eqref{eq:5.19}, we infer that  martingale solution $\tue$ satisfies the following equality, for $t \in [0,T]$, $\bbP$-a.s.
\begin{align}
\label{eq:5.35}
&\left(\tale(t) - \tale(0), \test \right)_\LQeps = -\nu \int_0^{t} \left(\curl \tale(s) , \curl \test \right)_\LQeps\,ds \\
& \quad - \nu \int_0^t \left(\curl \tble(s), \curl \test\right)_\LQeps\,ds - \int_0^t \left(\left[\tale(s) \cdot \nabla \right]\tale(s), \test \right)_\LQeps\,ds  \nn \\
& \quad - \int_0^t \left(\left[\tble(s) \cdot \nabla \right]\tale(s), \test \right)_\LQeps\,ds - \int_0^t \left(\left[\tale(s) \cdot \nabla \right]\tble(s), \test \right)_\LQeps\,ds  \nn \\
& \quad - \int_0^t \left(\left[\tble(s) \cdot \nabla \right]\tble(s), \test \right)_\LQeps\,ds + \int_0^t \left(\tMe f^\eps(s), \test \right)_\LQeps\,ds \nn \\
& \quad +  \int_0^t \left( \tMe\left(\widetilde{G}_\eps(s) d \tWe(s)\right), \test \right)_\LQeps \nn \\
&=: \sum_{i=1}^8 J_i^\eps(t).\nn
\end{align}
The proof of lemma turns out to be  a direct application of Corollary~\ref{corA.2.2}. Indeed, by  Lemma~\ref{lemma5.4}, assumptions  $(a)$ and $(b)$ of Corollary~\ref{corA.2.2} are satisfied and therefore,
it is sufficient to show that the sequence $\left(\alpha_\eps\right)_{\eps > 0}$ satisfies the Aldous condition $[\textbf{A}]$, see Definition~\ref{defnA.1.8},  in space $\rD(\rA^{-1})$.

Let $\theta \in (0,T)$ and $\left(\tau_\eps\right)_{\eps > 0}$ be a sequence of stopping times such that $0 \le \tau_\eps \le \tau_\eps +\theta\le T$. We start by estimating each term in the R.H.S. of \eqref{eq:5.35}. We will use the H\"older inequality, the scaling property from Lemma~\ref{lemma2.13}, the Poincar\'{e} type inequality \eqref{eq:2.24}, the Ladyzhenskaya inequality \eqref{eq:2.25}, inequality \eqref{eq:2.26}, the a priori estimates from Lemmas~\ref{lemma5.4}, \ref{lemma5.5}, the result from Lemma~\ref{lem:grad-ret-vec} and the relation \eqref{eq:3.24}.

In what follows, we will prove that each of the eight process from equality \eqref{eq:5.35} satisfies the Aldous condition $[\textbf{A}]$.
In order to help the reader, we will divide the following part of the proof into eight  parts.
\begin{itemize}
\item
For the first term, we obtain
\begin{align}
\label{eq:5.36}
\bbE \left|J_1^\eps(\tau_\eps + \theta) - J_1^\eps(\tau_\eps)\right| & = \bbE \left|\int_{\tau_\eps}^{\tau_\eps+\theta} \left(\curl \tale(s) , \curl \test \right)_\LQeps\,ds\right|\nn \\
& \le \bbE \int_{\tau_\eps}^{{\tau_\eps}+\theta} \left|\left(\tale(s), \Delta\left(\test\right)\right)_\LQeps\right|\,ds \nn\\
& \le \bbE \int_{\tau_\eps}^{{\tau_\eps}+\theta} \|\tale(s)\|_\LQeps \left\|\Delta\left(\test\right)\right\|_\LQeps\,ds\nn \\
& \le \eps C \E \int_{\tau_\eps}^{{\tau_\eps}+\theta} \|\ale(s)\|_\LS ds
\|\phi\|_{\rD(\rA)} \nn \\
& \le  \eps C \left[\E\left(\sup_{s \in [0,T]}\|\ale(s)\|^2_\LS\right)\right]^{1/2} \|\phi\|_{\rD(\rA)}\,\theta \nn\\
& \le \eps C_{\nu} \|\phi\|_{\rD(\rA)}\,\theta := c_1 \eps \cdot \theta \|\phi\|_{\rD(\rA)} .
\end{align}

\item  Similarly for the second term we have
\begin{align}
\label{eq:5.37}
\bbE \left|J_2^\eps(\tau_\eps + \theta) - J_2^\eps(\tau_\eps)\right| & = \bbE \left|\int_{\tau_\eps}^{{\tau_\eps}+\theta} \left(\curl \tble(s) , \curl \test \right)_\LQeps\,ds\right| \nn \\
&  \le \bbE \int_{\tau_\eps}^{{\tau_\eps}+\theta} \left|\left(\tble(s), \Delta\left(\test\right)\right)_\LQeps\right|\,ds \nn \\
& \le \bbE \int_{\tau_\eps}^{{\tau_\eps}+\theta} \|\tble(s)\|_\LQeps \left\|\Delta\left(\test\right)\right\|_\LQeps\,ds\nn \\
&\le \eps^{3/2} C \|\phi\|_{\rD(\rA)} \bbE \left( \int_{\tau_\eps}^{{\tau_\eps}+\theta} \|\tble(s)\|_{\rV_\eps}^2\,ds\right)^{1/2} \theta^{1/2} \nn\\
&\le \eps^{3/2} C \left(\bbE \|\tble\|^2_{L^2(0,T; \rV_\eps)} \right)^{1/2}\|\phi\|_{\rD(\rA)}\,\theta^{1/2} \nn \\
&\le \eps^{2} C_{\nu}\|\phi\|_{\rD(\rA)}\, \theta^{1/2} := c_2 \eps^2 \cdot \theta^{1/2} \|\phi\|_{\rD(\rA)}.
\end{align}

\item
Now we consider the first non-linear term.
\begin{align}
\label{eq:5.38}
\bbE & \left|J_3^\eps(\tau_\eps + \theta) - J_3^\eps(\tau_\eps)\right| = \bbE \left|\int_{\tau_\eps}^{{\tau_\eps}+\theta} \left(\tale(s), \left[\tale(s) \cdot \nabla \right]\test \right)_\LQeps\,ds\right| \nn\\
& \le  \bbE \int_{\tau_\eps}^{{\tau_\eps}+\theta} \|\tale(s)\|_\LQeps \|\tale(s)\|_{\mathbb{L}^3(Q_\eps)}\left\|\nabla\left(\test\right)\right\|_{\mathbb{L}^6(Q_\eps)}\,ds \nn \\
&\le \eps C \bbE \int_{\tau_\eps}^{{\tau_\eps}+\theta}\|\ale(s)\|_\LS\|\ale(s)\|_{\mathbb{L}^3(\bbS)}\|\phi\|_{\bbW^{1,6}(\bbS)}\,ds \nn \\
& \le \eps \left[\bbE \left(\sup_{s \in [0,T]} \|\ale(s)\|^2_\LS \right) \right]^{1/2} \left[\bbE \|\ale\|^2_{L^2(0,T; \mathbb{L}^3(\bbS))} \right]^{1/2}\|\phi\|_{\bbW^{2,2}(\bbS)} \theta^{1/2} \nn \\
&\le \eps C_{\nu} \|\phi\|_{\rD(\rA)}\,\theta^{1/2} := c_3 \eps \cdot \theta^{1/2} \|\phi\|_{\rD(\rA)} .
\end{align}

\item

Similarly for the second  non-linear term, we have
\begin{align}
\label{eq:5.39}
\bbE & \left|J_4^\eps(\tau_\eps + \theta) - J_4^\eps(\tau_\eps)\right| = \bbE  \left|\int_{\tau_\eps}^{{\tau_\eps}+\theta} \left(\left[\tble(s) \cdot \nabla \right]\tale(s), \test \right)_\LQeps\,ds\right| \nn\\
&\qquad \qquad \qquad \qquad \quad = \bbE \left|\int_{\tau_\eps}^{{\tau_\eps}+\theta} \left(\tale(s), \left[\tble(s) \cdot \nabla \right]\test \right)_\LQeps\,ds\right| \nn \\
&\le  \bbE \int_{\tau_\eps}^{{\tau_\eps}+\theta} \|\tale(s)\|_\LQeps\|\tble(s)\|_{\mathbb{L}^3(Q_\eps)}\left\|\nabla\left(\test\right)\right\|_{\mathbb{L}^6(Q_\eps)}\,ds \nn \\
&\le \eps^{7/6} C \bbE\int_{\tau_\eps}^{{\tau_\eps}+\theta}\|\ale(s)\|_\LS\|\tble(s)\|_{\rV_\eps}\|\phi\|_{\bbW^{1,6}(\bbS)}\,ds \nn \\
&\le \eps^{7/6} C \left[\bbE \left(\sup_{s\in[0,T]}\|\ale(s)\|^2_\LS\right)\right]^{1/2} \left[\bbE\|\tble\|^2_{L^2(0,T; \rV_\eps)}\right]^{1/2} \|\phi\|_{\bbW^{2,2}(\bbS)}\,\theta^{1/2} \nn \\
&\le \eps^{5/3} C_{\nu} \|\phi\|_{\rD(\rA)}\,\theta^{1/2} := c_4 \eps^{5/3} \cdot \theta^{1/2} \|\phi\|_{\rD(\rA)} .
\end{align}

\item

Now as in the previous case, for the next  mixed non-linear term, we obtain
\begin{align}
\label{eq:5.40}
\bbE & \left|J_5^\eps(\tau_\eps + \theta) - J_5^\eps(\tau_\eps)\right| = \bbE \left|\int_{\tau_\eps}^{{\tau_\eps}+\theta} \left(\left[\tale(s) \cdot \nabla \right]\tble(s), \test \right)_\LQeps\,ds\right| \nn \\
& \le \eps^{7/6} C \left[\bbE \left(\sup_{s\in[0,T]}\|\ale(s)\|^2_\LS\right)\right]^{1/2} \left[\bbE\|\tble\|^2_{L^2(0,T; \rV_\eps)}\right]^{1/2} \|\phi\|_{\bbW^{2,2}(\bbS)}\,\theta^{1/2} \nn \\
& \le \eps^{5/3} C_\nu \|\phi\|_{\rD(\rA)}\,\theta^{1/2} := c_5 \eps^{5/3} \cdot \theta^{1/2} \|\phi\|_{\rD(\rA)} .
\end{align}

\item  Finally,  for the last non-linear term, we get
\begin{align}
\label{eq:5.41}
\bbE & \left|J_6^\eps(\tau_\eps + \theta) - J_6^\eps(\tau_\eps)\right| = \bbE \left|\int_{\tau_\eps}^{{\tau_\eps}+\theta} \left(\left[\tble(s) \cdot \nabla \right]\tble(s), \test \right)_\LQeps\,ds\right| \nn \\
&= \bbE \left|\int_{\tau_\eps}^{{\tau_\eps}+\theta} \left(\tble(s), \left[\tble(s) \cdot \nabla \right]\test \right)_\LQeps\,ds\right| \nn \\
&\le \bbE \int_{\tau_\eps}^{{\tau_\eps}+\theta} \|\tble(s)\|_\LQeps \|\tble(s)\|_{\mathbb{L}^3(Q_\eps)} \left\|\nabla\left(\test\right)\right\|_{\mathbb{L}^6(Q_\eps)}\,ds \nn \\
&\le \eps^{2/3} C \bbE \int_{\tau_\eps}^{{\tau_\eps}+\theta}\|\tble(s)\|_\LQeps\|\tble(s)\|_{\rV_\eps}\|\phi\|_{\bbW^{1,6}(\bbS)}\,ds \nn \\
&\le \eps^{2/3} C \left[\bbE \left(\sup_{s\in[0,T]}\|\tble(s)\|^2_\LQeps\right)\right]^{1/2} \left[\bbE\|\tble\|^2_{L^2(0,T; \rV_\eps)}\right]^{1/2} \|\phi\|_{\bbW^{2,2}(\bbS)}\,\theta^{1/2} \nn \\
&\le \eps^{5/3} C_\nu \|\phi\|_{\rD(\rA)}\,\theta^{1/2} =: c_6 \eps^{5/3} \cdot \theta^{1/2} \|\phi\|_{\rD(\rA)}.
\end{align}

\item
Now for the term corresponding to the external forcing $\tfe$, we have using the radial invariance of $M_\eps \wtd{f}_\eps$ and assumption \eqref{eq:5.11}
\begin{align}
\label{eq:5.42}
\bbE  \left|J_7^\eps(\tau_\eps + \theta) - J_7^\eps(\tau_\eps)\right|& = \bbE \left|\int_{\tau_\eps}^{{\tau_\eps}+\theta} \left(\tMe \tfe(s), \test \right)_\LQeps\,ds\right| \nn \\
&= \bbE \left|\int_{\tau_\eps}^{{\tau_\eps}+\theta} \left(\int_\bbS \inteps  M_\eps \tfe(s, r\bx)\phi(\bx)\,dr\,d\sigma(\bx)\right)\,ds\right| \nn \\
& = \eps \bbE\left|\int_{\tau_\eps}^{{\tau_\eps}+\theta} \left(\int_\bbS \ocirc{M}_\eps \tfe(s,\bx) \phi(\bx)\,d\sigma(\bx)\right)\,ds\right|\nn \\
& \le \eps \bbE \int_{\tau_\eps}^{{\tau_\eps}+\theta} \|\ocirc{M}_\eps \tfe(s)\|_{\rV^\prime}\|\phi\|_\rV\,ds \nn \\
& \le \sqrt{\eps} C \left[\bbE\|\tfe\|^2_{L^2(0,T; \rV_\eps^\prime)} \right]^{1/2}\|\phi\|_{\rD(\rA)}\,\theta^{1/2} \nn\\
& \le \eps C \|\phi\|_{\rD(\rA)}\,\theta^{1/2}  =: c_7 \eps \cdot \theta^{1/2} \|\phi\|_{\rD(\rA)}.
\end{align}

\item
At the very end  we are left to deal with the last term corresponding to the stochastic forcing. Using the radial invariance of $M_\eps\tge^j$, It\^o isometry, scaling (see Lemma~\ref{lemma2.13}) and assumption \eqref{eq:5.12}, we get
\begin{align}
& \bbE  \left|J_8^\eps(\tau_\eps + \theta) - J_8^\eps(\tau_\eps)\right|^2  = \bbE \left| \left\langle \int_{\tau_\eps}^{{\tau_\eps}+\theta} \tMe \left( \widetilde{G}_\eps(s)\,d \tWe(s)\right), \test \right\rangle\right|^2 \\
& \qquad = \bbE \left| \int_{\tau_\eps}^{{\tau_\eps}+\theta} \left( \int_\bbS \inteps \sum_{j=1}^N M_\eps \left(\tge^j(s,r\bx)\,d\beta_j(s)\right)\phi(\bx)\,dr\,d\sigma(\bx)\right)\right|^2 \nn\\
&\qquad= \eps \bbE \left|\int_{\tau_\eps}^{{\tau_\eps}+\theta} \sum_{j=1}^N \int_\bbS \ocirc{M}_\eps \left( \tge^j(s, \bx)d\beta_j(s)\right) \phi(\bx)\,d\sigma(\bx)\right|^2 \nn \\
&\qquad\le \eps \bbE  \left\|\sum_{j=1}^N \int_{\tau_\eps}^{{\tau_\eps}+\theta} \ocirc{M}_\eps \left(\tge^j(s)d\beta_j(s)\right)\right\|_\LS^2 \|\phi\|_\LS^2 \nn \\
&\qquad=  \E \left( \int_{\tau_\eps}^{{\tau_\eps}+\theta} \sum_{j=1}^N \eps\|\ocirc{M}_\eps \tge^j(s)\|_\LS^2\,ds \right) \|\phi\|^2_\LS \nn \\
&\qquad=  \E \left(\int_{\tau_\eps}^{{\tau_\eps}+\theta} \sum_{j=1}^N \|\tMe\tge^j(s)\|^2_\LQeps ds \right)\|\phi\|_\LS^2 \nn \\
\label{eq:5.43}
&\qquad \le  \E  \left(\int_{\tau_\eps}^{{\tau_\eps}+\theta} \sum_{j=1}^N \|\tge^j(s)\|^2_\LQeps\,ds\right)\|\phi\|_\LS^2 \nn \\
&\qquad \le \eps Nc \|\phi\|_{\rD(\rA)}^2 \theta := c_8 \eps \cdot \theta \|\phi\|_{\rD(\rA)}^2 .
\end{align}
\end{itemize}
After having proved what we had promised, we are ready to conclude the proof of Lemma \ref{lemma5.6}.  Since for every $t > 0$
\[ \left(\tale (t), \test \right)_\LQeps = \eps \left(\ale(t), \phi \right)_\LS,\]
one has for $\phi \in \rD(\rA)$,
\begin{align}
\label{eq:aldous_2}
\left\|\ale(\tau_\eps + \theta) - \ale(\tau_\eps)\right\|_{\rD(\rA^{-1})} & = \sup_{\|\phi\|_{\rD(\rA)} = 1} \left(\ale(\tau_\eps + \theta) - \ale(\tau_\eps), \phi \right)_\LS \\
& = \frac{1}{\eps} \sup_{\|\phi\|_{\rD(\rA)} = 1} \left(\tale(\tau_\eps + \theta) - \tale(\tau_\eps), \test \right)_\LQeps. \nn
\end{align}
Let us fix $\kappa > 0$ and $\gamma > 0$. By equality \eqref{eq:5.35}, the sigma additivity property of probability measure and \eqref{eq:aldous_2}, we have
\begin{align*}
\mathbb{P} \left(\left\{\|\ale(\tau_\eps + \theta) - \ale(\tau_\eps)\|_{{\rD(\rA^{-1})}} \ge \kappa \right\}\right) \le \frac{1}{\eps} \sum_{i=1}^8 \mathbb{P} \left(\left\{\sup_{\|\phi\|_{\rD(\rA)} = 1} \left|J_i^\eps(\tau_\eps + \theta) - J_i^\eps(\tau_\eps)\right| \ge \kappa\right\}\right).
\end{align*}
Using the Chebyshev's inequality, we get
\begin{align}
\label{eq:aldous_3}
\mathbb{P} \left(\left\{\|\ale(\tau_\eps + \theta) - \ale(\tau_\eps)\|_{{\rD(\rA^{-1})}} \ge \kappa \right\}\right) & \le \frac{1}{\kappa \eps} \sum_{i=1}^7 \bbE \left(\sup_{\|\phi\|_{\rD(\rA)} = 1} \left|J_i^\eps(\tau_\eps + \theta) - J_i^\eps(\tau_\eps)\right| \right) \\
& \;\; + \frac{1}{\kappa^2 \eps} \bbE \left(\sup_{\|\phi\|_{\rD(\rA)} = 1} \left|J_8^\eps(\tau_\eps + \theta) - J_8^\eps(\tau_\eps)\right|^2\right) \nn.
\end{align}
Thus, using estimates \eqref{eq:5.36}--\eqref{eq:5.43} in \eqref{eq:aldous_3}, we get
\begin{align}
\label{eq:aldous_4}
& \mathbb{P}  \left(\left\{\|\ale(\tau_\eps + \theta) - \ale(\tau_\eps)\|_{{\rD(\rA^{-1})}} \ge \kappa \right\}\right) \\
&\; \le \frac{1}{\kappa \eps} \eps \theta^{1/2} \left[c_1 \theta^{1/2} + c_2 \eps + c_3 + c_4 \eps^{2/3} + c_5 \eps^{2/3} + c_6 \eps^{2/3} + c_7 \right] + \frac{1}{\kappa^2 \eps} c_8 \eps \theta. \nn
\end{align}
Let $\delta_i = \left(\dfrac{\kappa}{8 c_i} \gamma\right)^2$, for $i = 1, \cdots, 7$ and $\delta_8 = \dfrac{\kappa^2}{8 c_8} \gamma$. Choose $ \delta = \max_{i \in \{1, \cdots, 8\}}\delta_i$. Hence,
\[ \sup_{\eps > 0} \sup_{0 \leq \theta \leq \delta} \mathbb{P}  \left(\left\{\|\ale(\tau_\eps + \theta) - \ale(\tau_\eps)\|_{{\rD(\rA^{-1})}} \ge \kappa \right\}\right) \leq \gamma.\]
Since $\ale$ satisfies the Aldous condition $[\textbf{A}]$ in ${\rD(\rA^{-1})}$, we conclude   the proof of Lemma  \ref{lemma5.6} by invoking Corollary~\ref{corA.2.2}.
\end{proof}

\subsection{Proof of Theorem~\ref{thm:main_thm}}
\label{sec:proof-theorem}
For every $\eps > 0$, let us define the following intersection of spaces
\begin{equation}
\label{eq:y_eps space}
\mathcal{Y}_T^\eps = L^2_w(0,T; \rV_\eps) \cap \ccal ([0,T]; \rH_\eps^w).
\end{equation}
Now, choose a countable subsequence $\left\{\eps_k\right\}_{k \in \N}$ converging to $0$. For this subsequence define a product space $\mathcal{Y}_T$ given by
\[\mathcal{Y}_T = \Pi_{k \in \N} \mathcal{Y}_T^{\eps_k},\]
and $\eta \colon \Omega \to \mathcal{Y}_T$ by
\[\eta(\omega) = \left(\widetilde{\beta}_{\eps_1}(\omega), \widetilde{\beta}_{\eps_2}(\omega), \cdots, \right) \in \mathcal{Y}_T.\]
Now with this $\mathcal{Y}_T$-valued function we define a constant $\mathcal{Y}_T$-sequence
\[\eta_k \equiv \eta, \quad k \in \N.\]

\noindent Then by Lemma~\ref{lemma5.6} and the definition of sequence $\eta_k$, the set of measures $\left\{\mathcal{L}\left(\alpha_{\eps_k}, \eta_k\right), k \in \N\right\}$ is tight on $\mathcal{Z}_T \times \mathcal{Y}_T$.

Thus, by the Jakubowski--Skorokhod theorem\footnote{The space $\mathcal{Z}_T \times \mathcal{Y}_T \times \ccal([0,T]; \R^N)$ satisfies the assumption of Theorem~\ref{thmA.1.4}. Indeed, since $Z_T$ and $\mathcal{Y}_T^\eps$, $\eps > 0$ satisfies the assumptions (see \cite[Lemma~4.10]{[BD18]}) and $\ccal([0,T]; \R^N)$ is a Polish space and thus automatically satisfying the required assumptions.} there exists a subsequence $\left(k_n\right)_{n \in \N}$, a probability space $(\widehat{\Omega}, \widehat{\mathcal{F}}, \hp)$ and, on this probability space, $\mathcal{Z}_T \times \mathcal{Y}_T \times \ccal([0,T]; \R^N)$-valued random variables 
 $(\widehat{u}, \widehat{\eta}, \widehat{W})$, $\left(\widehat{\alpha}_{\eps_{k_n}}, \widehat{\eta}_{k_n}, \widehat{W}_{\eps_{k_n}}\right), n \in \N$ such that
\begin{equation}
\label{eq:5.44}
\begin{split}
& \left(\widehat{\alpha}_{\eps_{k_n}}, \widehat{\eta}_{k_n}, \widehat{W}_{\eps_{k_n}}\right) \,\mbox{ has the same law as }\, \left({\alpha}_{\eps_{k_n}}, \eta_{k_n}, \widetilde{W}\right)\\
&\mbox{ on } \calB\left(\mathcal{Z}_T \times \mathcal{Y}_T \times \ccal([0,T]; \R^N)\right)
\end{split}
\end{equation}
and
\begin{equation}
\label{eq:5.44a}
\left(\widehat{\alpha}_{\eps_{k_n}}, \widehat{\eta}_{k_n}, \widehat{W}_{\eps_{k_n}}\right) \to \left(\widehat{u}, \widehat{\eta}, \widehat{W} \right) \,\mbox{ in }\,\mathcal{Z}_T \times \mathcal{Y}_T \times \ccal([0,T]; \R^N),\quad \hp\mbox{-a.s.}
\end{equation}
In particular, using marginal laws, and definition of the process $\eta_k$, we have
\begin{equation}
    \label{eq:law_alp_bet}
    \mathcal{L}\left(\what{\alpha}_{\eps_{k_n}}, \what{\beta}_{\eps_{k_n}}\right) = \mathcal{L}\left(\alpha_{\eps_{k_n}}, \widetilde{\beta}_{\eps_{k_n}}\right) \; \mbox{ on } \calB\left(\mathcal{Z}_T \times \mathcal{Y}_T^{{\eps_{k_n}}}\right)
\end{equation}
where $\what{\beta}_{\eps_{k_n}}$ is the $k_n$th component of $\mathcal{Y}_T$-valued random variable $\what{\eta}_{k_n}$. We are not interested in the limiting process $\widehat{\eta}$ and hence will not discuss it further.

Using the equivalence of law of $\widehat{W}_{\eps_{k_n}}$ and $\widetilde{W}$ on $\ccal([0,T];\R^N)$ for $n \in \N$ one can show that $\widehat{W}$ and $\widehat{W}_{\eps_{k_n}}$ are $\R^N$-valued Wiener processes (see \cite[Lemma~5.2 and Proof]{[BGJ13]} for details).

$\widehat{\alpha}_{\eps_{k_n}} \to \widehat{u}$ in $\mathcal{Z}_T $, $\hp\mbox{-a.s.}$ precisely means that
\begin{align*}
\widehat{\alpha}_{\eps_{k_n}} \to \hu \; \mbox{ in }\,\ccal([0,T]; \rD(\rA^{-1})), \qquad \qquad
&\widehat{\alpha}_{\eps_{k_n}} \rightharpoonup \hu \; \mbox{ in }\,L^2(0,T; \rV),\\
\widehat{\alpha}_{\eps_{k_n}} \to \hu \; \mbox{ in }\,L^2(0,T; \rH), \qquad \qquad
&\widehat{\alpha}_{\eps_{k_n}} \to \hu \; \mbox{ in }\,\ccal([0,T]; \rH_w),
\end{align*}
and
\[\widehat{W}_{\eps_{k_n}} \to \widehat{W}\, \mbox{in }\,\ccal([0,T]; \R^N).\]
Let us denote the subsequence $(\widehat{\alpha}_{\eps_{k_n}}, \what{\beta}_{\eps_{k_n}}, \widehat{W}_{\eps_{k_n}})$ again by $(\hae, \what{\beta}_\eps, \widehat{W}_\eps)_{\eps \in (0,1]}$.\\
Note that since $\calB\left(\mathcal{Z}_T \times \mathcal{Y}_T \times \ccal([0,T]; \R^N) \right) \subset \calB(\mathcal{Z}_T) \times \mathcal{B}(\mathcal{Y}_T)\times \calB\left(\ccal([0,T]; \R^N)\right)$, the functions $\hu$, $\widehat{\eta}$ are $\mathcal{Z}_T$, $\mathcal{Y}_T$ Borel random variables respectively.

Using the retract operator $\oldReps : \LS \to \LQeps$ as defined in \eqref{eq:2.36}-\eqref{eq:2.35}, we define new processes $\underline{\hae}$ corresponding to old processes $\widetilde{\alpha}_\eps$ on the new probability space as follows
\begin{equation}
\label{eq:new_proc}
\underline{\hae} := \oldReps \hae.
\end{equation}
Moreover, by Lemma~\ref{lemma2.22} we have the following scaling property for these new processes, i.e.
\begin{equation}
\label{eq:scaling_new_pro}
\|\underline{\hae}\|_{\LQeps} = \sqrt{\eps}\|\hae\|_{\LS}.
\end{equation}

The following auxiliary result which is needed in the proof of Theorem~\ref{thm:main_thm}, cannot be deduced directly from the Kuratowski Theorem (see Theorem~\ref{thmB.1.1}).

\begin{lemma}
\label{lemma5.7}
Let $T > 0$ and $\mathcal{Z}_T$ be as defined in \eqref{eq:5.32}. Then the following sets $\ccal([0,T];\rH) \cap \mathcal{Z}_T$, $L^2(0,T; \rV) \cap \mathcal{Z}_T$ are Borel subsets of $\mathcal{Z}_T$.
\end{lemma}
\begin{proof}
See Appendix~\ref{s:AppB.2}.
\end{proof}

By Lemma~\ref{lemma5.7}, $\ccal([0,T]; \rH)$ is a Borel subset of $\mathcal{Z}_T$. Since $\ale \in \ccal([0,T]; \rH)$, $\hp$-a.s. and $\hae$, $\ale$ have the same laws on $\mathcal{Z}_T$, thus
\begin{equation}
\label{eq:5.45}
\mathcal{L}(\hae)\left(\ccal([0,T]; \rH)\right) = 1, \quad \eps > 0,
\end{equation}
and from estimates \eqref{eq:5.20} and \eqref{eq:hoe_alpha}, for $p \in[2, \infty)$
\begin{equation}
\label{eq:5.46}
\sup_{\eps > 0} \widehat{\E} \left(\sup_{0 \le s \le T} \|\hae(s)\|^{p}_{\LS} \right) \le K_1(p).
\end{equation}
Since $L^2(0,T; \mathrm{V}) \cap \mathcal{Z}_T$ is a Borel subset of $\calZ_T$ (Lemma~\ref{lemma5.7}), $\ale$ and $\hae$ have same laws on $\calZ_T$; from \eqref{eq:5.20}, we have
\begin{equation}
\label{eq:5.47}
\sup_{\eps > 0} \widehat{\E} \left[ \int_0^T \|\curlS \hae(s)\|^2_{\LS}\,ds \right] \le K_2.
\end{equation}
Since the laws of $\eta_{k_n}$ and $\widehat{\eta}_{k_n}$ are equal on $\mathcal{Y}_T$,  we infer that  the corresponding marginal laws are also equal.  In other words, the laws on $\calB\left(\mathcal{Y}_T^{\eps_{k_n}}\right)$ of  $\mathcal{L}(\what{\beta}_{\eps_{k_n}})$ and $\mathcal{L}(\wtd{\beta}_{\eps_{k_n}})$ are equal for every ${k_n}$.

Therefore, from the estimates \eqref{eq:5.29} and \eqref{eq:hoe_beta} we infer for $p \in [2, \infty)$
\begin{equation}
\label{eq:beta_estimate_1}
\widehat{\E} \left(\sup_{0 \le s \le T} \|\hbe(s)\|^{p}_{\LQeps} \right) \le K_3(p) \eps^{p/2}, \qquad \eps \in (0,1]
\end{equation}
and
\begin{equation}
\label{eq:beta_estimate_2}
\widehat{\E} \left[ \int_0^T \|\curl \hbe(s)\|^2_{\LQeps}\,ds \right] \le K_4 \eps, \qquad \eps \in (0,1].
\end{equation}

By inequality \eqref{eq:5.47} we infer that the sequence $(\hae)_{\eps > 0}$ contains a subsequence, still denoted by $(\hae)_{\eps > 0}$ convergent weakly (along the sequence $\eps_{k_n}$) in the space $L^2([0,T] \times \homega; \rV)$. Since $\hae \to \hu$ in $\mathcal{Z}_T$ $\hp$-a.s., we conclude that $\hu \in L^2([0,T] \times \homega; \rV)$, i.e.
\begin{equation}
\label{eq:5.48}
\widehat{\E} \left[\int_0^T \|\curlS \hu(s)\|^2_{\LS}\,ds \right] < \infty.
\end{equation}
Similarly by inequality \eqref{eq:5.46}, for every $p \in [2, \infty)$ we can choose a subsequence of $(\hae)_{\eps>0}$ convergent weak star (along the sequence $\eps_{k_n}$) in the space $L^p(\homega; L^\infty(0,T; \rH))$ and, using \eqref{eq:5.44a}, we infer that
\begin{equation}
\label{eq:5.49}
\widehat{\E} \left(\sup_{0 \le s \le T} \|\hu(s)\|^{p}_{\LS} \right) < \infty.
\end{equation}

Using the convergence from \eqref{eq:5.44a} and estimates \eqref{eq:5.46}--\eqref{eq:5.49} we will prove certain term-by-term convergences which will be used later to prove Theorem~\ref{thm:main_thm}.
In order to simplify the notation, in the result below  we write  $\lim_{\eps \to 0}$ but we mean
$\lim_{k_n \to \infty }$.

Before stating the next lemma, we introduce a new functional space $\mathbb{U}$ as the space of compactly supported, smooth divergence free vector fields on $\bbS$:
\begin{equation}
\label{eq:spaceU}
\mathbb{U}:= \{ \v := (0,\v_\lambda, \v_\varphi) \in C^\infty_c(\bbS;\R^3) : \ddivS \v = 0 \mbox{ in } \bbS\}.
\end{equation}

\begin{lemma}
\label{lemma5.8}
For all $t \in [0,T]$, and {$\phi \in \mathbb{U},$} we have (along the sequence $\eps_{k_n}$)
\begin{itemize}
\item[(a)] $\lim_{\eps \to 0} \hE \left[\int_0^T \left|\left(\hae(t) - \hu(t),\phi \right)_\LS\right|\,dt\right] = 0$, 
\item[(b)] $\lim_{\eps \to 0} \;\left(\hae(0) - u_0,\phi \right)_\LS  = 0$, 
\item[(c)] $\lim_{\eps \to 0} \hE\left[\int_0^T \left| \int_0^t \left(\dfrac{\nu}{1+\eps} \curlS \hae(s) - \nu \curlS \hu(s), \curl \phi \right)_\LS ds\right|\,dt\right] = 0$,
\item[(d)] $\lim_{\eps \to 0} \hE \left[ \int_0^T \left| \int_0^t \left\langle\dfrac{1}{1+\eps} \left[\hae(s) \cdot \nablaS \right]\hae(s) - \left[\hu(s) \cdot \nablaS\right]\hu(s), \phi \right\rangle ds\right|\,dt\right] = 0$,
\item[(e)] $\lim_{\eps \to 0} \int_0^T \left|\int_0^t \left\langle\ocirc{M}_\eps \tfe (s) - f(s), \phi \right\rangle ds\right|\,dt = 0,$
\item[(f)] $\lim_{\eps \to 0} \hE \left[\int_0^T\left|  \left( \int_0^t \left[\ocirc{M}_\eps \left(\widetilde{G}_\eps(s)d\widehat{W}_\eps(s)\right) - G(s)d\widehat{W}(s)\right], \phi\right)_\LS\,\right|^2 dt\right] = 0.$
\end{itemize}
\end{lemma}

\begin{proof} Let us fix $\phi \in \mathbb{U}$.

\noindent \textbf{(a)} We know that $\hae \to \hu$ in $\calZ_T$. In particular,
\[
 \hae \rightarrow \widehat{u} \quad \text{ in }\, \ccal([0,T]; \rH_w)\quad \hp\text{-a.s.}
\]
Hence, for $t \in [0,T]$
\begin{equation}
\label{eq:5.51}
\lim_{\eps \to 0} \left(\hae(t),\phi \right)_\LS = \left(\hu(t), \phi\right)_\LS, \quad \hp\text{-a.s.}
\end{equation}
Since by \eqref{eq:5.46}, for every $\eps > 0$, $\sup_{t\in[0,T]}\|\hae(s)\|^2_\LS \le K_1(2)$, $\hp$-a.s., using the dominated convergence theorem we infer that
\begin{equation}
\label{eq:convg1}
\lim_{\eps \to 0} \int_0^T \left|\left(\hae(t) - \hu(t), \phi\right)_\LS\right| dt = 0,\quad \hp\text{-a.s.}
\end{equation}
By the H\"older inequality, \eqref{eq:5.46} and \eqref{eq:5.49} for every $\eps > 0$ and every $r > 1$
\begin{equation}
\label{eq:convg2}
\begin{split}
& \hE \left[\left|\int_0^T \|\hae(t) - \hu(t)\|_\LS dt\right|^r\right] \\
& \qquad \le c \hE \left[\int_0^T\left(\|\hae(t)\|_\LS^{r} + \|\hu(t)\|^{r}_\LS\right)dt \right] \le cTK_1(r),
\end{split}
\end{equation}
where $c$ is some positive constant. To conclude the proof of assertion $(a)$ it is sufficient to use \eqref{eq:convg1}, \eqref{eq:convg2} and the Vitali's convergence theorem.

\noindent \textbf{(b)} Since $\hae \to \hu$ in $\ccal([0,T]; \rH_w)$ $\hp$-a.s.  we infer that
\begin{equation}\label{test-01}
\left(\hae(0),\phi \right)_\LS \to \left(\hu(0),\phi \right)_\LS, \quad \hp\text{-a.s.}
\end{equation}
Also, note that by condition \eqref{eq:initial_data_convg} in Assumption \ref{ass_sphere},
${\alpha}_\eps(0) = \oMe \tue(0)=\oMe[\widetilde{u}_0^\eps]$ converges weakly to $u_0$ in $\mathbb{L}^2(\bbS)$.  \\
On the other hand, by \eqref{eq:5.44} we infer that the laws of $\hae(0)$ and $\alpha_\eps(0)$ on $\rH$ are equal.  Since  $\alpha_\eps(0)$ is a constant  random variable  on the old probability space, we infer that
$\hae(0)$ is also a constant random variable (on the new probability space) and hence, by \eqref{eq:5.4} and \eqref{eq:5.19}, we infer that
$\hae(0)=\ocirc{M}_\eps [ \tu_0^\eps ]$ almost surely (on the new probability space).
Therefore we infer that
\[\left(\widehat{u}(0), \phi \right)_{\LS} = \left(u_0, \phi \right)_{\LS},\]
concluding the proof of assertion $(b)$.

\noindent \textbf{(c)} Since $\hae \to \hu$ in $\ccal([0,T]; \rH_w)$ $\hp$-a.s.,
\begin{align}
\label{eq:5.52}
\lim_{\eps \to 0}\frac{\nu}{1+\eps} \int_0^t \left(\curlS \hae(s), \curlS \phi \right)_\LS\,ds  & = - \lim_{\eps \to 0} \frac{\nu}{1+\eps} \int_0^t \left(\hae (s), \DDelta \phi \right)_\LS\,ds \\
= - \nu \int_0^t \left(\widehat{u}, \DDelta \phi \right)_\LS\,ds & = \nu \int_0^t \left(\curlS \widehat{u}, \curlS \phi \right)_\LS\,ds. \nn
\end{align}
The Cauchy-Schwarz inequality and estimate \eqref{eq:5.47} infer that for all $t \in [0,T]$ and $\eps \in (0,1]$
\begin{align}
\label{eq:convg3}
\hE & \left[\left|\int_0^t \left(\dfrac{\nu}{1+\eps}\curlS \hae(s), \curlS \phi\right)_\LS\,ds\right|^2\right]\\
& \quad \le \nu^2 \|\curlS \phi\|_\LS^2 \hE \left[\int_0^t \|\curlS \hae(s)\|^2_\LS\,ds\right] \le c K_2\, \nn
\end{align}
for some constant $c > 0$. By \eqref{eq:5.52}, \eqref{eq:convg3} and the Vitali's convergence theorem we conclude that for all $t \in [0,T]$
\[\lim_{\eps \to 0} \hE\left[ \left| \int_0^t \nu\left(\dfrac{1}{1+\eps} \curlS \hae(s) - \curlS \hu(s), \curlS \phi \right)_\LS\,ds\right|\right] = 0 .\]
Assertion $(c)$ follows now from \eqref{eq:5.47}, \eqref{eq:5.48} and the dominated convergence theorem.

\noindent \textbf{(d)} For the non-linear term, using the Sobolev embedding $\hH^2(\bbS) \hookrightarrow \lL^\infty(\bbS)$, we have
\begin{align}
\label{eq:5.53}
&\left| \int_0^{t} \left\langle\left[\hae(s)  \cdot \nablaS \right]\hae(s), \phi \right)_\LS\,ds - \int_0^t \left\langle\left[\widehat{u}(s) \cdot \nablaS \right]\widehat{u}(s), \phi \right\rangle\,ds\right| \\
&\le \left|\int_0^t \int_\bbS \left[\left(\hae(s,x) - \widehat{u}(s,x) \right) \cdot \nablaS \widehat{u}(s,x) \right]\cdot \phi(x) \,dx\,ds\right| \nn \\
&\quad + \left|\int_0^t \int_\bbS \left[\hae(s,x) \cdot \nablaS \left(\hae(s,x) - \widehat{u}(s,x)\right)\right] \cdot \phi(x) \,dx\,ds\right| \nn \\
& \le \|\hae - \widehat{u}\|_{L^2(0,T; \rH)}\|\widehat{u}\|_{L^2(0,T; \rV)}\|\phi\|_{\hH^2(\bbS)} \nn \\
& \quad + \left|\int_0^t \left(\nablaS \left(\hae(s,x) - \widehat{u}(s,x)\right), \hae(s)\phi\right)_\LS ds\right|. \nn
\end{align}
The first term converges to zero as $\eps \to 0$, since $\hae \to \hu$ strongly in $L^2(0,T; \rH)$ $\hp$-a.s., $\widehat{u} \in L^2(0,T; \rV)$ and the second term converges to zero too as $\eps \to 0$ because $\hae \to \widehat{u}$ weakly in $L^2(0,T; \rV)$. Using the H\"older inequality, estimates \eqref{eq:5.46} and the embedding $\hH^2(\bbS) \hookrightarrow \mathbb{L}^\infty(\bbS)$ we infer that for all $t \in [0,T]$, $\eps \in (0,1]$, the following inequalities hold
\begin{align}
\label{eq:convg4}
\hE & \left[\left|\int_0^t \tfrac{1}{1+ \eps}\left\langle\left[\hae(s)\cdot \nablaS\right]\hae(s),\phi\right\rangle ds\right|^2\right] \nn \\
&\le \hE \left[\left(\int_0^t\|\hae(s)\|_\LS^2 \| \nablaS \phi\|_{\mathbb{L}^\infty(\bbS)}\,ds\right)^2\right]
\le c \|\nablaS \phi\|_{\hH^2(\bbS)}\,t\, \hE \left[\int_0^t \|\hae(s)\|^4_\LS \,ds\right] \nn\\
& \le c \|\phi\|_{\hH^3(\bbS)}\, t\,  \hE \left[\sup_{s \in [0,t]}\|\hae(s)\|^4_\LS \right] \le \widetilde{c} K_1(4).
\end{align}
By \eqref{eq:5.53}, \eqref{eq:convg4} and the Vitali's convergence theorem we obtain for all $t \in [0,T]$,
\begin{equation}
\label{eq:convg5}
\lim_{\eps \to 0} \hE \left[\left|\int_0^t \left\langle\frac{1}{1+\eps} \left[\hae(s)\cdot\nablaS \right]\hae(s) - \left[\hu(s)\cdot \nablaS \right]\hu(s), \phi \right\rangle\,ds\right|\right] = 0.
\end{equation}
Using the H\"older inequality and estimates \eqref{eq:5.46}, \eqref{eq:5.49}, we obtain for all $t \in [0,T]$, $\eps \in (0,1]$
\begin{align*}\hE \left[\left|\int_0^t \frac{1}{1+\eps}\left\langle\left[\hae(s)\cdot \nablaS\right] \hae(s) - \left[\hu(s) \cdot \nablaS \right]\hu(s), \phi\right\rangle\,ds\right|\right] \nn \\
 \le c \|\phi\|_{\hH^3(\bbS)}\hE \left[\sup_{t \in [0,T]}\|\hae(s)\|^2_\LS + \sup_{t \in[0,T]}\|\hu(s)\|^2_\LS\right] \le 2\widetilde{c} K_1(2)
\end{align*}
where $c, \widetilde{c} > 0$ are constants. Hence by \eqref{eq:convg5} and the dominated convergence theorem, we infer assertion $(d)$.

\noindent \textbf{(e)} Assertion $(e)$ follows because by Assumption~\ref{ass_sphere} the sequence $\left(\oMe \tfe\right)$ converges weakly in $L^2(0,T; \LS)$ to $f$.

\noindent \textbf{(f)} By the definition of maps $\wtd{G}_\eps$ and $G$, we have
\begin{align*}
\int_0^t\left\|\left(\oMe\left(\widetilde{G}_\eps(s)\right) - G(s), \phi \right)_\LS\right\|_{\mathcal{T}_2(\R^N; \R)}ds = \int_0^t \sum_{j=1}^N\left|\left(\oMe[\tge^j(s)] - g^j(s), \phi\right)_\LS\right|^2\,ds.
\end{align*}
Since, by Assumption \ref{ass_sphere},  for every $j \in \{1, \cdots, N\}$, and $s \in [0,t]$, $\oMe[\tge^j(s)]$ converges weakly to $g^j(s)$ in $\LS$ as $\eps \to 0$, we get
\begin{equation}
\label{eq:5.54}
\lim_{\eps \to 0} \int_0^t \left\|\left(\oMe\left(\widetilde{G}_\eps(s)\right) - G(s), \phi \right)_\LS\right\|^2_{\mathcal{T}_2(\R^N; \R)}ds = 0.
\end{equation}
By assumptions on $\tge^j$, we obtain the following inequalities for every $t \in [0,T]$ and $\eps \in (0,1]$
\begin{align}
\label{eq:5.55}
\hE & \left[\left|\int_0^t \left\|\left(\oMe(\widetilde{G}_\eps(s)) , \phi \right)_\LS\right\|^2_{\mathcal{T}_2(\R^N;\R)}\,ds\right|^2\right] \nn\\
& \le c \|\phi\|_\LS^4 \left[\int_0^t\|\oMe(\widetilde{G}_\eps(s))\|^4_{\mathcal{T}_2(\R^N;\LS)}\,ds\right] \nn \\
& = c \|\phi\|^4_\LS \left[\int_0^t \dfrac{1}{\eps^2}\left(\sum_{j=1}^N \eps \|\oMe(\tge^j(s))\|^2_\LS\right)^2 ds\right] \nn \\
&  \le \dfrac{\widetilde{c}}{\eps^2}\|\phi\|^4_\LS \left[\sum_{j=1}^N \int_0^t \|\tge^j(s)\|^4_\LQeps ds\right] \le K,
\end{align}
where $c, \widetilde{c} > 0$ are some constants. Using the Vitali's convergence theorem, by \eqref{eq:5.54} and \eqref{eq:5.55} we infer
\begin{equation}
\label{eq:5.56}
\lim_{\eps \to 0} \hE \left[ \int_0^t \left\|\left(\oMe\left(\widetilde{G}_\eps(s)\right) - G(s), \phi \right)_\LS\right\|_{\mathcal{T}_2(\R^N; \R)}^2ds\right] = 0.
\end{equation}
Hence, by the properties of the It\^o integral we deduce that for all $t \in [0,T]$,
\begin{equation}
\label{eq:5.57}
\lim_{\eps \to 0} \hE \left[\left| \left( \int_0^t \left[\oMe \left(\widetilde{G}_\eps(s)\right) - G(s)\right] d \widehat{W}(s), \phi \right)_\LS \right|^2 \right] = 0.
\end{equation}
By the It\^o isometry and assumptions on $\tge^j$ and $g^j$ we have for all $t \in [0,T]$ and $\eps \in (0,1]$
\begin{align}
\label{eq:5.58}
\hE & \left[\left|\left(\int_0^t \left[\oMe\left(\widetilde{G}_\eps(s)\right) - G(s)\right]d\widehat{W}(s) , \phi \right)_\LS\right|^2\right] \nn \\
& = \hE \left[\int_0^t \left\|\left(\oMe \left(\widetilde{G}_\eps(s)\right) - G(s), \phi \right)_\LS\right\|^2_{\mathcal{T}_2(\R^N; \R)}\,ds\right] \nn \\
& \le c \|\phi\|^2_\LS \left[\sum_{j=1}^N \int_0^t  \left(\|\oMe(\tge^j(s))\|^2_\LS + \|g^j(s)\|^2_\LS\right)\,ds \right] \nn \\
& = c \|\phi\|^2_\LS \left[\sum_{j=1}^N \int_0^t  \left(\frac{1}{\eps}\|\tMe \tge^j(s)\|^2_\LQeps + \|g^j(s)\|^2_\LS\right)\,ds \right] \nn \\
& \le c \|\phi\|^2_\LS \left[\sum_{j=1}^N \int_0^t  \left(\frac{1}{\eps}\|\tge^j(s)\|^2_\LQeps + \|g^j(s)\|^2_\LS\right)\,ds \right] \le \widetilde{K},
\end{align}
where $c > 0$ is a constant. Thus, by \eqref{eq:5.57}, \eqref{eq:5.58} and the dominated convergence theorem assertion $(f)$ holds.
\end{proof}

\begin{lemma}
\label{lemma_convg_beta}
For all $t \in [0,T]$ and $\phi \in \mathbb{U}$, we have (along the sequence $\eps_{k_n}$)
\begin{itemize}
\item[(a)] $\lim_{\eps \to 0} \hE \left[ \int_0^T \left| \int_0^t \left( \frac{\nu}{\eps}\curl \hbe(s), \curl \test\right)_\LQeps\,ds \right|\,dt\right] = 0,$
\item[(b)] $\lim_{\eps \to 0} \hE \left[ \int_0^T \left|\frac{1}{\eps} \int_0^t \left\langle \left[\hbe(s)\cdot \nabla \right]\underline{\hae}(s), \test \right\rangle_\eps\,ds\right|\,dt \right] = 0,$
\item[(c)] $\lim_{\eps \to 0} \hE \left[ \int_0^T \left|\frac{1}{\eps}\int_0^t \left\langle\left[\underline{\hae}(s)\cdot \nabla \right]\tble(s), \test \right\rangle_\eps\,ds\right|\,dt\right] = 0,$
\item[(d)] $\lim_{\eps \to 0} \hE \left[ \int_0^T \left| \frac{1}{\eps} \int_0^t \left\langle\left[\tble(s) \cdot \nabla \right]\tble(s), \test \right\rangle_\eps\,ds \right|dt\right] = 0,$
\end{itemize}
where the process $\underline{\hae}$ is defined in \eqref{eq:new_proc}.
\end{lemma}
\begin{proof}
Let us fix $\phi  \in \mathbb{U}$.

\noindent \textbf{(a)} Let $t \in [0,T]$, then by the H\"older inequality, \eqref{eq:3.24}, Poincar\'e inequality and estimate \eqref{eq:beta_estimate_2}, we have the following inequalities
\begin{align*}
\widehat{\E} & \left|\int_0^t \left(\frac{\nu}{\eps} \curl \widehat{\beta}_\eps (s), \curl \test \right)_\LQeps\,ds \right| \le \widehat{\E} \int_0^t \left|\left(\frac{\nu}{\eps} \hbe (s),  \Delta \left(\test\right)\right)_\LQeps \right|\,ds\\
& \le \frac{\nu}{\eps} T^{1/2} \widehat{\E} \left[\int_0^t \|\hbe(s)\|^2_\LQeps\,ds\right]^{1/2} \left\|\Delta\left(\test\right)\right\|_\LQeps \\
& \le \frac{\nu}{\eps} T^{1/2}C \eps \left(\widehat{\E}\int_0^T \|\curl \hbe(t)\|^2_\LQeps\right)^{1/2}\sqrt{\eps} \|\phi\|_{{\rD(\rA)}} \\
& \le \nu \sqrt{\eps} C \sqrt{K_4 \eps} T^{1/2}\|\phi\|_{\rD(\rA)}.
\end{align*}
Thus
\begin{equation}
\label{eq:convg_beta_1}
\lim_{\eps \to 0} \widehat{\E} \left|\int_0^t \left(\frac{\nu}{\eps} \curl \widehat{\beta}_\eps (s), \curl \test \right)_\LQeps\,ds \right| = 0.
\end{equation}
We infer assertion $(a)$ by dominated convergence theorem, estimate \eqref{eq:beta_estimate_2} and convergence \eqref{eq:convg_beta_1}.

\noindent \textbf{(b)} Using the H\"older inequality, scaling property (Lemma~\ref{lemma2.13}), Corollary~\ref{cor_L3ineq}, relations \eqref{eq:grad-ret-vec1}, \eqref{eq:3.24}, and estimates \eqref{eq:5.46}, \eqref{eq:beta_estimate_2}, we get for $t \in [0,T]$
\begin{align*}
\frac{1}{\eps}\widehat{\E} & \left| \int_0^t \left\langle \left[\hbe(s) \cdot \nabla \right]\underline{\hae}(s), \test \right\rangle_\eps\,ds \right| = \frac{1}{\eps} \widehat{\E} \left| \int_0^t \left( \underline{\hae}(s), \left[\hbe(s) \cdot \nabla \right] \test \right)_\LQeps\,ds \right| \\
& \le \frac{1}{\eps} \widehat{\E} \int_0^t \|\underline{\hae}(s)\|_\LQeps \|\hbe(s)\|_{\mathbb{L}^3(Q_\eps)}\,ds\left\|\nabla \left(\test\right)\right\|_{\mathbb{L}^6(Q_\eps)}\\
& \le \frac{1}{\eps} \widehat{\E} \left[\int_0^t  \sqrt{\eps}\|\hae(s)\|_\LS \sqrt{\eps} \|\hbe(s)\|_{\rV_\eps}\,ds \right]\eps^{1/6}\|\phi\|_{\bbW^{1,6}(\bbS)} \\
& \le \eps^{1/6} T^{1/2} \left(\widehat{\E} \sup_{s \in [0,T]}\|\hae(s)\|_\LS^2\right)^{1/2} \left(\widehat{\E} \int_0^T \|\hbe(s)\|^2_{\rV_\eps}\,ds\right)^{1/2}\|\phi\|_{\bbW^{2,2}(\bbS)} \\
&\le \eps^{2/3} T^{1/2} K_1(2)^{1/2}K_4^{1/2} \|\phi\|_{\rD(\rA)}.
\end{align*}
Thus
\begin{equation}
\label{eq:convg_beta_2}
\lim_{\eps \to 0}\frac{1}{\eps}\widehat{\E}  \left| \int_0^t \left\langle \left[\hbe(s) \cdot \nabla \right]\underline{\hae}(s), \test \right\rangle_\eps\,ds \right| = 0.
\end{equation}
We infer assertion $(b)$ by dominated convergence theorem, estimates \eqref{eq:5.46}, \eqref{eq:beta_estimate_2} and convergence \eqref{eq:convg_beta_2}. Assertion $(c)$ can be proved similarly.

\noindent \textbf{(d)} Now for the last one, using the H\"older inequality, Corollary~\ref{cor_L3ineq}, relations \eqref{eq:grad-ret-vec1}, \eqref{eq:3.24}, and estimates \eqref{eq:beta_estimate_1}, \eqref{eq:beta_estimate_2}, we get for $t \in [0,T]$
\begin{align*}
\frac{1}{\eps}& \widehat{\E}  \left| \int_0^t \left\langle \left[\hbe(s) \cdot \nabla \right]\hbe(s), \test \right\rangle_\eps\,ds \right| = \frac{1}{\eps} \widehat{\E} \left| \int_0^t \left( \hbe(s), \left[\hbe(s) \cdot \nabla \right] \test \right)_\LQeps\,ds \right| \\
& \le \frac{1}{\eps} \widehat{\E} \int_0^t \|\hbe(s)\|_\LQeps \|\hbe(s)\|_{\mathbb{L}^3(Q_\eps)}\,ds\left\|\nabla \left(\test\right)\right\|_{\mathbb{L}^6(Q_\eps)}\\
& \le \frac{1}{\eps} \widehat{\E} \left[\int_0^t \|\hbe(s)\|_\LQeps \sqrt{\eps} \|\hbe(s)\|_{\rV_\eps}\,ds \right]\eps^{1/6}\|\phi\|_{\bbW^{1,6}(\bbS)} \\
& \le \eps^{-1/3}T^{1/2} \left(\widehat{\E} \sup_{s \in [0,T]}\|\hbe(s)\|_\LQeps^2\right)^{1/2} \left(\widehat{\E} \int_0^T \|\hbe(s)\|^2_{\rV_\eps}\,ds\right)^{1/2}\|\phi\|_{\bbW^{2,2}(\bbS)} \\
&\le \eps^{2/3} T^{1/2} K_3(2)^{1/2}K_4^{1/2} \|\phi\|_{\rD(\rA)}.
\end{align*}
Thus
\begin{equation}
\label{eq:convg_beta_3}
\lim_{\eps \to 0}\frac{1}{\eps}\widehat{\E}  \left| \int_0^t \left\langle \left[\hbe(s) \cdot \nabla \right]\hbe(s), \test \right\rangle_\eps\,ds \right| = 0.
\end{equation}
We infer assertion $(d)$ by dominated convergence theorem, estimates \eqref{eq:beta_estimate_1}, \eqref{eq:beta_estimate_2} and convergence \eqref{eq:convg_beta_3}.
\end{proof}

Finally, to finish the proof of Theorem~\ref{thm:main_thm}, we will follow the methodology as in \cite{[Motyl14]} and introduce some auxiliary notations (along sequence $\eps_{k_n}$)
\begin{align}
\label{eq:5.59}
&\Lambda_\eps(\hae,\hbe,\widehat{W}_\eps, \phi) :=  \left(\hae(0), \phi \right)_\LS -\frac{\nu}{1+\eps} \int_0^t \left(\curlS \hae(s), \curlS \phi \right)_\LS\,ds \\
&\quad  - \frac{1}{1+\eps} \int_0^{t} \left\langle\left[\hae(s) \cdot \nablaS \right]\hae(s), \phi \right\rangle\,ds + \int_0^{t} \left\langle\oMe \tfe(s), \phi \right\rangle\,ds \nn \\
&\quad + \left(\int_0^t\ocirc{M}_\eps \left[\widetilde{G}_\eps(s)d\widehat{W}_\eps(s) \right], \phi \right)_\LS - \frac{\nu}{\eps} \int_0^{t} \left(\curl \hbe(s), \curl \test\right)_\LQeps\,ds \nn \\ & \quad - \frac{1}{\eps} \int_0^{t} \left\langle\left[\hbe(s) \cdot \nabla \right]\underline{\hae}(s), \test \right\rangle_\eps\,ds  - \frac{1}{\eps} \int_0^{t} \left\langle\left[\underline{\hae}(s) \cdot \nabla \right]\hbe(s), \test \right\rangle_\eps\,ds \nn \\
&\quad  - \frac{1}{\eps} \int_0^{t} \left\langle\left[\hbe(s) \cdot \nabla \right]\hbe(s), \test \right\rangle_\eps\,ds \nn,
\end{align}

\begin{align}
\label{eq:5.60}
&\Lambda(\hu,\widehat{W}, \phi) :=  \left(u_0, \phi \right)_\LS -\nu \int_0^t \left(\curlS \hu(s), \curlS \phi \right)_\LS\,ds \\
&\quad  - \int_0^{t} \left(\left[\hu(s) \cdot \nablaS \right]\hu(s), \phi \right)_\LS\,ds + \int_0^{t} \left\langle f(s), \phi \right\rangle\,ds \nn \\
&\quad + \left(\int_0^t G(s) d \widehat{W}(s), \phi \right)_\LS \nn.
\end{align}

\begin{corollary}
\label{cor5.12}
Let $\phi \in \mathbb{U}$. Then (along the sequence $\eps_{k_n}$)
\begin{equation}
\label{eq:5.61}
\lim_{\eps \to 0} \left\|\left(\hae(\cdot), \phi \right)_\LS - \left(\hu(\cdot), \phi\right)_\LS\right\|_{L^1(\widehat{\Omega} \times [0,T])} = 0
\end{equation}
and
\begin{equation}
\label{eq:5.62}
\lim_{\eps \to 0} \left\|\Lambda_\eps(\hae, \hbe, \widehat{W}_\eps, \phi) - \Lambda(\hu, \widehat{W}, \phi)\right\|_{L^1(\widehat{\Omega} \times [0,T])} = 0.
\end{equation}
\end{corollary}
\begin{proof}
Assertion \eqref{eq:5.61} follows from the equality
\[\left\|\left(\hae(\cdot), \phi \right)_\LS - \left(\hu(\cdot), \phi\right)_\LS\right\|_{L^1(\widehat{\Omega} \times [0,T])} = \hE \left[ \int_0^T \left|\left(\hae(t) - \hu(t), \phi \right)_\LS\right|\,dt\right]\]
and Lemma~\ref{lemma5.8} $(a)$. To prove assertion \eqref{eq:5.62}, note that by the Fubini Theorem, we have
\[\left\|\Lambda_\eps(\hae, \hbe, \widehat{W}_\eps, \phi) - \Lambda(\hu, \widehat{W}, \phi)\right\|_{L^1(\widehat{\Omega} \times [0,T])} =  \int_0^T \hE \left|\Lambda_\eps(\hae, \hbe, \widehat{W}_\eps, \phi)(t) - \Lambda(\hu, \widehat{W}, \phi)(t)\right|dt.\]
To conclude the proof of the corollary, it is sufficient to note that by Lemma~\ref{lemma5.8} $(b)-(f)$ and Lemma~\ref{lemma_convg_beta}, each term on the right hand side of \eqref{eq:5.59} tends at least in $L^1(\widehat{\Omega} \times [0,T])$ to the corresponding term (to zero in certain cases) in \eqref{eq:5.60}.
\end{proof}

\begin{proof}[Conclusion of proof of Theorem~\ref{thm:main_thm}]
Let us fix $\phi \in \mathbb{U}$. Since $\ale$ is a solution of \eqref{eq:snse_mod}, for all $t \in [0,T]$,
\[\left(\ale(t), \phi \right)_\LS = \Lambda_\eps(\ale, \tble, \widetilde{W}_\eps, \phi)(t), \qquad \mathbb{P}\text{-a.s.}\]
In particular,
\[\int_0^T \E \left[\left|\left(\ale(t),\phi\right)_\LS - \Lambda_\eps(\ale, \tble, \widetilde{W}_\eps, \phi)(t)\right|\right]\,dt = 0.\]
Since $\mathcal{L}(\ale, \eta_{k_n}, \widetilde{W}_\eps) = \mathcal{L}(\hae, \what{\eta}_{k_n}, \widehat{W}_\eps)$, on $\calB\left(\mathcal{Z}_T \times \mathcal{Y}_T \times \mathcal{C}([0,T]; \R^N)\right)$ (along the sequence $\eps_{k_n}$),
\[\int_0^T \hE \left[\left|\left(\hae(t),\phi\right)_\LS - \Lambda_\eps(\hae, \hbe, \widehat{W}_\eps, \phi)(t)\right|\right]\,dt = 0.\]
Therefore by Corollary~\ref{cor5.12} and the definition of $\Lambda$, for almost all $t \in [0,T]$ and $\widehat{\mathbb{P}}$-almost all $\omega \in \widehat{\Omega}$
\[\left(\hu(t), \phi \right)_\LS - \Lambda(\hu, \widehat{W}, \phi)(t) = 0,\]
i.e. for almost all $t \in [0,T]$ and $\widehat{\mathbb{P}}$-almost all $\omega \in \widehat{\Omega}$
\begin{align}
\label{eq:5.63}
\left(\hu(t), \phi\right)_\LS + \nu \int_0^t \left(\curlS \hu(s), \curlS \phi \right)_\LS\,ds + \int_0^{t} \left\langle\left[\hu(s) \cdot \nablaS \right]\hu(s), \phi \right\rangle\,ds  \\
 = \left(u_0, \phi \right)_\LS + \int_0^{t} \left\langle f(s), \phi \right\rangle\,ds + \left(\int_0^t G(s) d \widehat{W}(s), \phi \right)_\LS \nn.
\end{align}
Hence \eqref{eq:5.63} holds for every $\phi \in  \mathbb{U}$. Since $\hu$ is a.s. $\rH$-valued continuous process, by a standard density argument, we infer that \eqref{eq:5.63} holds for every $\phi \in \rV$ ($\mathbb{U}$ is dense in $\rV$).

Putting $\widehat{\mathcal{U}} := \left(\widehat{\Omega}, \widehat{\mathcal{F}}, \widehat{\bbF}, \hp\right)$, we infer that the system $\left(\widehat{\mathcal{U}}, \widehat{W}, \hu\right)$ is a martingale solution to \eqref{eq:5.6}--\eqref{eq:5.9}.
\end{proof}

\appendix

\section{Vector analysis in spherical coordinates}
\label{sec:vec-ana}

In this appendix we collect some basic results from vector algebra and formulas for Laplace and gradient of scalar function and vector fields in spherical coordinates.

The following identities are very well known \cite[Appendix]{[TZ97]} in vector algebra$\colon$
Let $\u, \v$ and $\w$ be $\R^3$-valued smooth vector fields then
\begin{align}
\label{eq:2.17}
\curl \curl \u &= - \Delta \u + \nabla \ddiv \u ,\\
\label{eq:2.18}
\u \cdot \left(\v \times \w \right) &= \left(\u \times \v \right)\cdot \w ,\\
\label{eq:2.19}
\int_{Q_\eps} \v \cdot \curl \u \, d\by &= \int_{Q_\eps} \u \cdot \curl \v \, d\by + \int_{\partial Q_\eps} \left(\u \times \v \right)\cdot \vecn\,d\sigma.
\end{align}

The Laplace--Beltrami operator of a scalar function $\psi$, in spherical coordinates $(r, \lambda, \varphi)$ is given by
\begin{equation}
    \label{eq:lap-bel}
    \nabla^2 \psi = \frac{\pa^2 \psi}{\pa r^2} + \frac{2}{r} \frac{\pa \psi}{\pa r} + \frac{1}{r^2 \sin \lambda}\frac{\pa}{\pa \lambda}\left(\sin \lambda \frac{\pa \psi}{\pa \lambda}\right) + \frac{1}{r^2 \sin^2 \lambda}\frac{\pa^2\psi}{\pa \varphi^2},
\end{equation}
and its gradient is given by
\begin{equation}
    \label{eq:grad}
    \nabla \psi = \frac{\pa \psi}{\pa r}\widehat{e_r} + \frac{1}{r}\frac{\pa \psi}{\pa \lambda}\widehat{e_\lambda} + \frac{1}{r \sin \lambda}\frac{\pa \psi}{\pa \varphi}\widehat{e_\varphi}.
\end{equation}
For a vector field $u$ written in the spherical coordinates, $u = u_r \widehat{e_r} + u_\lambda \widehat{e_\lambda} + u_\varphi \widehat{e_\varphi}$, the curl and the divergence are given as follows
\begin{align}
\label{eq:curl}
    \curl u & = \frac{1}{r \sin \lambda} \left[\frac{\pa}{\pa \lambda}(\sin \lambda u_\varphi) - \frac{\pa u_\lambda}{\pa \varphi} \right]\widehat{e_r} + \left[\frac{1}{r \sin \lambda}\frac{\pa u_r}{\pa \varphi} - \frac{1}{r}\frac{\pa}{\pa r}(r u_\varphi)\right] \widehat{e_\lambda} \\
    &\qquad + \left[\frac{1}{r}\frac{\pa}{\pa r}(r u_\lambda) - \frac{1}{r}\frac{\pa u_r}{\pa \lambda}\right]\widehat{e_\varphi}, \nn\\
    \label{eq:diver}
    \ddiv u &= \frac{1}{r^2}\frac{\pa}{\pa r}(r^2 u_r) + \frac{1}{r \sin \lambda} \frac{\pa}{\pa \lambda}(u_\lambda \sin \lambda) + \frac{1}{r \sin \lambda}\frac{\pa u_\varphi}{\pa \varphi}.
\end{align}
The Laplacian of a vector field in spherical coordinates is
\begin{equation}
    \begin{split}
    \left(\Delta u\right)_r & = \nabla^2 u_r -\frac{2u_r}{r^2} - \frac{2}{r^2}\frac{\pa u_\lambda}{\pa \lambda} - 2 \frac{\cot \lambda}{r^2} u_\lambda - \frac{2}{r^2 \sin \lambda} \frac{\pa u_\varphi}{\pa \varphi},\\
    \left(\Delta u\right)_\lambda & = \nabla^2 u_\lambda + \frac{2}{r^2}\frac{\pa u_r}{\pa \lambda} - \frac{u_\lambda}{r^2 \sin^2 \lambda} - \frac{2 \cos \lambda}{r^2 \sin^2 \lambda} \frac{\pa u_\varphi}{\pa \varphi}, \\
    \left(\Delta u\right)_\varphi & = \nabla^2 u_\varphi - \frac{u_\varphi}{r^2 \sin^2 \lambda} + \frac{1}{r^2 \sin \lambda}\frac{\pa u_r}{\pa \varphi} + \frac{2 \cos \lambda}{r^2 \sin^2 \lambda} \frac{\pa u_\lambda}{\pa \varphi},
    \end{split}
\end{equation}
where $\nabla^2 u_r$, $\nabla^2 u_\lambda$ and $\nabla^2 u_\varphi$ are as in \eqref{eq:lap-bel}.

We recall some standard differential operators on the unit sphere $\bbS$.
For $\psi$ a scalar function defined on $\bbS$, the tangential gradient is given by
\begin{equation}
\label{eq:C.1.1}
\nabla^\prime \psi = \frac{\partial \psi}{\partial \lambda} \widehat{e_\lambda} + \frac{1}{\sin \lambda} \frac{\partial \psi}{\partial \varphi}\widehat{e_\varphi}.
\end{equation}

The Laplace--Beltrami of a scalar function $\psi$ is
\begin{equation}
\label{eq:C.1.2}
\Delta^\prime \psi = \frac{1}{\sin{\lambda}} \left[ \frac{\partial}{\partial \lambda}\left(\sin{\lambda} \frac{\partial \psi}{\partial \lambda}\right) + \frac{1}{\sin \lambda} \frac{\partial^2\psi}{\partial \varphi^2}\right].
\end{equation}
For a tangential vector field $\v$ defined on $\bbS$, $\v = \v_\lambda \hat{e}_\lambda + \v_\varphi\hat{e}_\varphi$, the tangential divergence is expressed by
\begin{equation}
\label{eq:C.1.3}
{\rm{div}}^\prime\v = \frac{1}{\sin{\lambda}}\frac{\partial}{\partial \lambda}\left(\v_\lambda \sin{\lambda}\right) + \frac{1}{\sin \lambda} \frac{\partial \v_\varphi}{\partial \varphi},
\end{equation}
and the ${\rm{curl}}^\prime \v$ is the scalar function defined by
\begin{equation}
\label{eq:C.1.4}
{\rm{curl}}^\prime\v = \frac{1}{\sin{\lambda}}\frac{\partial}{\partial \lambda}\left(\v_\varphi \sin{\lambda}\right) - \frac{1}{\sin \lambda} \frac{\partial \v_\lambda}{\partial \varphi}.
\end{equation}
The Laplace--de Rham operator applied to a vector field $\v$ is given by
\begin{equation}
\label{eq:C.1.5}
\DDelta\v = \left[\deltaS\v_\lambda - \frac{2 \cos{\lambda}}{\sin^2{\lambda}} \frac{\partial \v_\varphi}{\partial \varphi} - \frac{\v_\lambda}{\sin^2{\lambda}}\right]\widehat{e_\lambda} + \left[\deltaS\v_\varphi + \frac{2 \cos{\lambda}}{\sin^2{\lambda}} \frac{\partial \v_\lambda}{\partial \varphi} - \frac{\v_\varphi}{\sin^2{\lambda}}\right]\widehat{e_\varphi},
\end{equation}
where $\deltaS\v_\varphi$ and $\deltaS\v_\lambda$ are as in \eqref{eq:C.1.2}.


\section{The \textit{curl} and the Stokes operator}
\label{sec:curl-stokes}

In this section we present a integration by parts formula corresponding to $\mathrm{curl}$ operator and later we use it to give a relation between the Stokes operator $\rA_\eps$ and $\curl$.

Let $\mathcal{O} \subset \R^3$ be a bounded domain with a regular boundary $\partial \dom$. Define
\begin{equation}
    \label{eq:hcurl}
    \rH(\mathrm{curl}) := \{ u \in \lL^2(\mathcal{O}) : \curl u \in \lL^2(\mathcal{O}) \}.
\end{equation}
$\rH(\mathrm{curl})$ with the graph norm
\[\|\v\|_{\rH(\mathrm{curl})}^2 := \|\v\|^2_{\lL^2(\mathcal{O})} + \|\curl \v\|^2_{\lL^2(\mathcal{O})}, \quad \v \in \rH(\mathrm{curl}), \]
is a Hilbert space.

The following theorem is a reformulation of Lemma~4.2 from \cite[Pg~341]{[DL76]}.

\begin{theorem}
\label{thm:ibp}
Assume that $\mathcal{O} \subset \R^3$ be a bounded domain with a regular boundary and $\vec{n}$ be unit normal vector field on $\Gamma = \partial \mathcal{O}$ (directed towards exterior of $\mathcal{O}$). Then there exists a unique bounded linear map
\begin{equation}
    \label{eq:linmap}
    n \times \cdot \big|_\Gamma \colon \rH(\mathrm{curl}) \to \hH^{-1/2}(\Gamma),
\end{equation}
such that
\begin{equation}
    \label{eq:linmap2}
    \left(n \times \cdot \right)(u) = n \times u \big|_\Gamma,
\end{equation}
if $u \in C^1_0(\overline{\mathcal{O}})$ and
\begin{equation}
    \label{eq:ibp1}
    \int_\mathcal{O} \curl u \cdot \v \,dy - \int_\mathcal{O} u \cdot \curl \v \,dy = {}_{\hH^{-1/2}(\Gamma)}\langle (n \times \cdot )(u) \big|_\Gamma , \v\big|_\Gamma \rangle_{\hH^{1/2}(\Gamma)}
\end{equation}
for every $u \in \rH(\mathrm{curl})$ and $\v \in \hH^1(\mathcal{O})$.
\end{theorem}

\begin{remark}
We will call formula \eqref{eq:ibp1} the generalised Stokes formula. From now on we will write $n \times u \big|_\Gamma$ instead of $(n \times \cdot )(u)\big|_\Gamma$ for $u \in \rH(\mathrm{curl})$.
\end{remark}

Recall that
\begin{equation}
    \label{eq:hspace1}
    \rH = \{ u \in \lL^2(\dom) \colon \ddiv u = 0\; \mbox{in } \dom \mbox{ and } u \cdot \vec{n} = 0 \mbox{ on } \Gamma\},
\end{equation}
and the Stokes operator is given by
\begin{equation}
    \label{eq:stokes1}
    D(A) = \{ u \in \hH^2(\dom) \colon \ddiv u = 0 \mbox{ in } \dom,\; u \cdot \vec{n} = 0 \mbox{ and } \vec{n} \times \curl u = 0 \mbox{ on } \Gamma\},
\end{equation}
\begin{equation}
    \label{eq:stokes2}
    A u = \pi(- \Delta u), \quad u \in D(A).
\end{equation}

\begin{remark}
To define $n \times \curl u$ as an element of $\hH^{-1/2}(\Gamma)$, we need to know that $\curl(\curl u) \in \lL^2(\dom)$. But if $u \in \hH^2(\dom)$, then obviously this condition is satisfied.
\end{remark}

\begin{theorem}
\label{thm:stokes-adj}
$A$ is self-adjoint and non-negative on $\rH$.
\end{theorem}

\begin{proof}
Here we will only show that $A$ is symmetric and non-negative. Let $u, \v \in D(A)$, then
\begin{align}
\label{eq:adj1}
    \left( Au, \v \right)_{\lL^2(\dom)} = \left( \pi(-\Delta u), \v \right)_{\lL^2(\dom)} = - \left( \Delta u, \pi \v \right)_{\lL^2(\dom)} = \left( - \Delta u, \v \right)_{\lL^2(\dom)}.
\end{align}
Recall that (from \eqref{eq:2.17}) for smooth $\R^3$-valued vector fields,
\[\curl(\curl u) = - \Delta u + \nabla (\ddiv u).\]
Using the above identity in \eqref{eq:adj1} along with the fact that $u \in D(A)$, in particular, $\ddiv u = 0$ and generalised Stokes formula \eqref{eq:ibp1}, we have
\begin{align*}
    \left( Au, \v \right)_{\lL^2(\dom)} & = \int_\dom \curl(\curl u) \cdot \v\,dy \\
    & = \int_\dom \curl u \cdot \curl \v\,dy + {}_{\hH^{-1/2}(\Gamma)}\langle \underbrace{\vec{n} \times (\curl u) \big|_\Gamma}_{=\,0, \mbox{ since } u \in D(A)} , \v\big|_\Gamma \rangle_{\hH^{1/2}(\Gamma)} \\
    & = \int_\dom \curl u \cdot \curl \v \,dy.
\end{align*}
Similarly
\[\left( u, A \v \right)_{\lL^2(\dom)} = \left( A\v, u \right)_{\lL^2(\dom)} = \int_\dom \curl \v \cdot \curl u\,dy = \int_\dom \curl u \cdot \curl \v\,dy.\]
This establishes that $A$ is symmetric on $\rH$. The non-negativity follows from the above identity by taking $\v = u \in D(A)$.
\end{proof}

Using the definition of $D(A)$, we can characterise $D(A^{1/2})$ as
\[D(A^{1/2}) = \{u \in \lL^2(\dom) \colon \ddiv u = 0,\, \curl u \in \lL^2(\dom) \mbox{ and } u \cdot \vec{n} = 0 \mbox{ on } \Gamma\}.\]
By Theorem~6.1 \cite[Pg~358]{[DL76]} we have
\[D(A^{1/2}) \subset \{ u \in \hH^1(\dom) : u \cdot \vec{n} = 0 \mbox{ on } \Gamma\}.\]

We use the following relation repeatedly in our calculations.

\begin{lemma}
\label{lem:curl-stokes}
Let $u \in D(A^{1/2})$ and $\v \in D(A)$. Then
\[\left(\curl u, \curl \v \right)_{\lL^2(\dom)} = \left(u, A \v \right)_{\lL^2(\dom)}.\]
\end{lemma}

\begin{proof}
Note that for $u \in D(A^{1/2})$ and $\v \in D(A)$, the LHS makes sense. Using the generalised Stokes formula \eqref{eq:ibp1}, we get
\begin{align*}
    \int_\dom \curl u \cdot \curl \v\,dy = \int_\dom u \cdot \curl(\curl \v)\,dy + {}_{\hH^{-1/2}(\Gamma)}\langle \vec{n} \times  u \big|_\Gamma , \curl \v\big|_\Gamma \rangle_{\hH^{1/2}(\Gamma)}.
\end{align*}
To finish the proof we need the following lemma:
\begin{lemma}
\label{lem:curl-stokes2}
Let $u \in \hH^1(\dom)$ such that $u \cdot \vec{n} = 0$ on $\Gamma$ and $\Phi \in \hH^1(\dom)$ with $\vec{n}\times \Phi = 0$ on $\Gamma$. Then
\begin{equation}
    \label{eq:boundary1}
    {}_{\hH^{-1/2}(\Gamma)}\langle \vec{n} \times  u \big|_\Gamma , \Phi\big|_\Gamma \rangle_{\hH^{1/2}(\Gamma)} = 0.
\end{equation}
\end{lemma}
\begin{proof}
It is sufficient to prove \eqref{eq:boundary1} for $u \in C^1_0(\overline{\dom})$ with $u \cdot \vec{n} =0$ on $\Gamma$. In this case for all $x \in \Gamma$
\begin{align*}
    \left(\vec{n} \times u\right)\cdot \Phi = -\left(u \times \vec{n}\right)\cdot \Phi = - u \cdot \left(n \times \Phi\right) = 0.
\end{align*}
\end{proof}
The proof of Lemma~\ref{lem:curl-stokes} is finished by observing that
\[\left(u, A \v \right)_{\lL^2(\dom)} = \int_\dom u \cdot \curl(\curl \v)\,dy,\]
from the proof of Theorem~\ref{thm:stokes-adj}.
\end{proof}

Let us consider an abstract framework. Let $H$ be a Hilbert space and $A$ be a non-negative self-adjoint operator on $H$.

\begin{lemma}
\label{lem:adj2}
Let $u \in D(A^{1/2})$ and $\v \in D(A)$. Then
\[\left(A^{1/2} u, A^{1/2} \v\right)_H = (u, A \v)_H.\]
\end{lemma}

\begin{proof}
Take $u \in D(A^{1/2})$, $\v \in D(A)$. Then $A^{1/2} \v \in D(A^{1/2})$. So by self-adjointness of $A^{1/2}$ and $A^{1/2}A^{1/2} = A$,
\[\left(A^{1/2} u , A^{1/2}\v\right)_H = \left(u, A^{1/2}(A^{1/2} \v)\right)_H = \left(u, A\v\right)_H.\]
\end{proof}

\begin{lemma}
\label{lem:stokes3}
Let $u, \v \in D(A^{1/2})$, then by Lemma~\ref{lem:adj2} and Lemma~\ref{lem:curl-stokes},
\[\left(\curl u , \curl \v \right)_{\lL^2(\dom)} = \left(A^{1/2}u, A^{1/2} \v\right)_{\lL^2(\dom)}.\]
\end{lemma}

\begin{proof}
Let $u, \v \in D(A)$, then
\[\left(A^{1/2}u, A^{1/2}\v\right)_{\lL^2(\dom)} = \left(u, A \v \right)_{\lL^2(\dom)} = \left(\curl u, \curl \v \right)_{\lL^2(\dom)}.\]
By density argument, this is true for all $u, \v \in D(A^{1/2})$.
\end{proof}

As a consequence of the above lemma we have
\[\|\nabla u \|^2_{\lL^2(\dom)} = \|\curl u\|^2_{\lL^2(\dom)}, \quad u \in D(A^{1/2}).\]


\section{Proof of the Poincar\'e and the Ladyzhenskaya inequalities}
\label{sec:appLadyzhenskaya}

\textbf{Proof of Lemma~\ref{lemma2.15}:} We will establish the Poincar\'e inequality \eqref{eq:2.24} following the footsteps of Lemma~2.1 \cite{[TZ97]} with all the details. By density argument, it is enough to prove \eqref{eq:2.24} for smooth functions. Let $\psi \in \mathcal{C}(\overline{Q_\eps})$ be a real continuous function. We write for any $\xi, \eta \in [1, 1 + \eps]$:
\begin{equation}
\label{eq:poinc1}
\begin{split}
\xi^2 \psi^2(\xi, \bx) + \eta^2 \psi^2(\eta, \bx) & = 2 \xi \eta \psi(\xi, \bx)\psi(\eta, \bx) + \left[\xi \psi(\xi, \bx) - \eta \psi(\eta, \bx)\right]^2 \\
& = 2 \xi \eta \psi(\xi, \bx)\psi(\eta, \bx) + \left[ \int_\eta^\xi \frac{\partial \left(r \psi \right)}{\partial r}(r, \bx) dr \right]^2
\end{split}
\end{equation}
with $\bx = \frac{\by}{|\by|} \in \bbS$. We fix $\xi$ and integrate w.r.t. $\eta \in [1, 1+\eps]$ to obtain
\begin{equation}
\label{eq:poinc2}
\begin{split}
\eps \xi^2 \psi^2(\xi, \bx) + \int_1^{1+\eps} \eta^2 \psi^2(\eta, \bx)\,d\eta & = 2 \xi \psi(\xi, \bx) \int_1^{1+\eps} \eta \psi(\eta, \bx)\,d\eta \\
&\; + \int_1^{1+\eps} \left[ \int_\eta^\xi \frac{\partial \left(r \psi \right)}{\partial r}(r, \bx) dr \right]^2 d\eta.
\end{split}
\end{equation}
With $\psi = u_r$ and $\xi = 1$, observing that $u_r(1, \bx) = 0$ (because of the boundary condition $u \cdot \vec{n} = 0$ on $\partial Q_\eps$) from \eqref{eq:poinc2} we obtain
\begin{equation}
\label{eq:poinc3}
\int_1^{1+\eps} \eta^2 u_r^2(\eta, \bx)\,d\eta = \int_1^{1+\eps} \left[ \int_\eta^\xi \frac{\partial \left(r u_r \right)}{\partial r}(r, \bx) dr \right]^2 d\eta.
\end{equation}
Applying \eqref{eq:poinc2} with $\psi = \widehat{N}_\eps u_\lambda$, we get
\begin{equation}
\label{eq:poinc4}
\begin{split}
& \eps \xi^2 \left[\widehat{N}_\eps u_\lambda(\xi, \bx)\right]^2 + \int_1^{1+\eps} \eta^2 \left[\widehat{N}_\eps u_\lambda(\eta, \bx)\right]^2 d\eta \\
&\qquad = 2 \xi \widehat{N}_\eps u_\lambda(\xi, \bx) \int_1^{1+\eps} \eta \widehat{N}_\eps u_\lambda (\eta, \bx)\,d\eta + \int_1^{1+\eps} \left[ \int_\eta^\xi \frac{\partial \left(r \widehat{N}_\eps u_\lambda \right)}{\partial r}(r, \bx) dr \right]^2 d\eta.
\end{split}
\end{equation}
Observing from Lemma~\ref{lemma2.7} for every $\psi \in L^2(Q_\eps)$
\[\int_1^{1+\eps} r \widehat{N}_\eps \psi (r, \bx)\,dr = 0 \qquad \forall \bx \in \bbS,\]
and since the first term on the LHS of \eqref{eq:poinc4} is positive we can simplify \eqref{eq:poinc4} as follows
\begin{equation}
\label{eq:poinc5}
\int_1^{1+\eps} \eta^2 \left[\widehat{N}_\eps u_\lambda(\eta, \bx)\right]^2 d\eta \le \int_1^{1+\eps} \left[ \int_\eta^\xi \frac{\partial \left(r \widehat{N}_\eps u_\lambda \right)}{\partial r}(r, \bx) dr \right]^2 d\eta.
\end{equation}
Similarly for $\psi = \widehat{N}_\eps u_\varphi$, we have
\begin{equation}
\label{eq:poinc6}
\int_1^{1+\eps} \eta^2 \left[\widehat{N}_\eps u_\varphi(\eta, \bx)\right]^2 d\eta \le \int_1^{1+\eps} \left[ \int_\eta^\xi \frac{\partial \left(r \widehat{N}_\eps u_\varphi \right)}{\partial r}(r, \bx) dr \right]^2 d\eta.
\end{equation}
Thus, using \eqref{eq:poinc3}, \eqref{eq:poinc5} and \eqref{eq:poinc6} for each of the cases $\psi = u_r$, $\psi = \widehat{N}_\eps u_\lambda$ and $\psi = u_\varphi$, we obtain
\begin{equation}
\label{eq:poinc7}
\int_1^{1+\eps} \eta^2 \psi^2(\eta, \bx)\, d\eta \le \int_1^{1+\eps} \left[ \int_\eta^\xi \frac{\partial \left(r \psi \right)}{\partial r}(r, \bx) dr \right]^2 d\eta.
\end{equation}
Using the Cauchy-Schwarz inequality, we find
\begin{equation}
\label{eq:poinc8}
\begin{split}
\int_1^{1+\eps} \eta^2 \psi^2(\eta, \bx)\, d\eta & \le \int_1^{1+\eps}|\xi - \eta|\,d\eta \int_1^{1+\eps} \left|\frac{\partial \left(r \psi \right)}{\partial r}\right|^2 dr \le \eps^2 \int_1^{1+\eps} \left|\frac{\partial \left(r \psi \right)}{\partial r}\right|^2 dr \\
& \le 2\eps^2 \int_1^{1+\eps} \left|\frac{\partial \psi }{\partial r}\right|^2 r^2 dr + 2 \eps^2 \int_1^{1+\eps} |\psi|^2 dr \\
& \le 2\eps^2 \int_1^{1+\eps} \left|\frac{\partial \psi }{\partial r}\right|^2 r^2 dr + 2 \eps^2 \int_1^{1+\eps} |r \psi|^2 dr,
\end{split}
\end{equation}
the last inequality follows since $r \ge 1$. On rearranging, we obtain
\begin{equation}
\label{eq:poinc9}
\left(1 - 2 \eps^2\right) \int_1^{1+\eps} r^2 \psi^2(r, \lambda, \varphi)\, dr \le 2\eps^2 \int_1^{1+\eps} \left|\frac{\partial \psi }{\partial r}\right|^2 r^2 dr,
\end{equation}
which implies for $0 \le \eps < \frac12$
\begin{equation}
\label{eq:poinc10}
\int_1^{1+\eps} r^2 \psi^2(r, \lambda, \varphi)\, dr \le 4\eps^2 \int_1^{1+\eps} \left|\frac{\partial \psi }{\partial r}\right|^2 r^2 dr.
\end{equation}
We then integrate w.r.t. $\lambda$ and $\varphi$ to obtain
\begin{equation}
\label{eq:poinc11}
\int_{Q_\eps} \psi^2(\by)\, d\by \le 4\eps^2 \int_{Q_\eps} \left|\frac{\partial \psi }{\partial r}\right|^2 d\by.
\end{equation}
Adding \eqref{eq:poinc11} for $\psi = u_r$, $\psi = \widehat{N}_\eps u_\lambda$ and $\psi = \widehat{N}_\eps u_\varphi$; using finally
\[\int_{Q_\eps}\left|\frac{\partial}{\partial r} \widetilde{N}_\eps u\right|^2 d\by \le \|\nabla \widetilde{N}_\eps u\|^2_{\lL^2(Q_\eps)}  = \|\curl \widetilde{N}_\eps u\|^2_{\lL^2(Q_\eps)} \qquad \forall\, u \in \rV_\eps,\]
we conclude the proof of the inequality \eqref{eq:2.24}. \hfill $\square$

\noindent \textbf{Proof of Lemma~\ref{lemma2.16}:} We will prove the lemma for smooth vector fields $u \in \mathcal{C}^\infty(Q_\eps)$. By Lemma~2.3 \cite{[TZ97]} there exists a constant $c_0 > 0$ s.t.
\begin{equation}
\label{eq:lad1}
\|\widetilde{N}_\eps u\|_{\lL^6(Q_\eps)} \le c_0 \|\nabla \left(\widetilde{N}_\eps u\right)\|_{\lL^2(Q_\eps)}.
\end{equation}
By Lemma~6.1 \cite[Eq.~6.11, Pg 359]{[DL76]} for vector fields $\v \in \mathcal{C}^1(Q_\eps)$ with $\v \cdot \vec{n} = 0$ on $\partial Q_\eps$, we have a constant $c_2 > 0$ s.t.
\begin{equation}
\label{eq:lad2}
\|\nabla \v\|^2_{\lL^2(Q_\eps)} \le 2 \left[\|\ddiv \v\|^2_{\lL^2(Q_\eps)} + \|\curl \v\|^2_{\lL^2(Q_\eps)} + c_2\|\v\|^2_{\lL^2(Q_\eps)}\right].
\end{equation}
Also by Poincar\'e inequality \eqref{eq:2.24} for all $u \in \rV_\eps$, we have
\begin{equation}
\label{eq:lad3}
\|\widetilde{N}_\eps u\|^2_{\lL^2(Q_\eps)} \le 4 \eps^2 \|\curl \widetilde{N}_\eps u\|^2_{\lL^2(Q_\eps)}.
\end{equation}
Using \eqref{eq:lad2} with $\v = \widetilde{N}_\eps u$ along with the fact that $\ddiv \v = 0$ if $\ddiv u = 0$ and $\v \cdot \vec{n} = 0$ on $\partial Q_\eps$ if $u \cdot \vec{n} = 0$ on $\partial Q_\eps$ and combining it with the Poincar\'e inequality \eqref{eq:lad3}, we obtain
\[\|\widetilde{N}_\eps u\|^2_{\lL^6(Q_\eps)} \le 2c_0^2 \left[\|\widetilde{N}_\eps u\|^2_{\rV_\eps} + 4c_2 \eps^2 \|\widetilde{N}_\eps u \|^2_{\rV_\eps} \right].\]
Therefore, choosing $c_1 = \sqrt{2} c_0 \left(1 + 4 c_2 \right)^{1/2}$ we establish \eqref{eq:2.25} for every $u \in \mathcal{C}^1(Q_\eps)$ with $\ddiv u = 0$ on $Q_\eps$ and $u \cdot \vec{n} = 0$ on $\partial Q_\eps$. We finish the proof using density argument. \hfill $\square$
\vspace{6pt}

\noindent \textbf{Proof of Lemma~\ref{lem:lap-ret-vec}}
Let $u$ be a tangential vector field defined on $\bbS$, $u = (0, u_\lambda, u_\varphi)$. Then using the definition of the map $\oldReps$ (see \eqref{eq:2.36}), for $Q_\eps \ni \by = r \bx$, $r \in (1,1+\eps)$ and $\bx \in \bbS$, we have
\[\oldReps[u](\by) = \left(0, \newReps\left[u_\lambda\right](\by), \newReps\left[u_\varphi\right](\by)\right) = \left(0, \frac{u_\lambda(\bx)}{r}, \frac{u_\varphi(\bx)}{r}\right).\]
Using the definitions of Laplace ($\Delta$) for vector fields in spherical coordinates, Laplace--Beltrami ($\nabla^2$) for scalars, tangential Laplace ($\DDelta$) for tangential vector fields and Laplace--Beltrami ($\Delta^\prime$) for the scalar defined on $\bbS$, we have following relations:
\[\Delta\left(\oldReps u\right) = \left(\left(\Delta\left(\oldReps u\right)\right)_r, \left(\Delta\left(\oldReps u\right)\right)_\lambda, \left(\Delta\left(\oldReps u\right)\right)_\varphi\right),\]
where
\begin{align*}
    \left(\Delta\left(\oldReps u\right)\right)_r & = - \frac{2}{r^3}\frac{\pa u_\lambda}{\pa \lambda} - 2 \frac{\cot \lambda}{r^3} u_\lambda - \frac{2}{r^3 \sin \lambda} \frac{\pa u_\varphi}{\pa \varphi},\\
    \left(\Delta\left(\oldReps u\right)\right)_\lambda & = \nabla^2\left(\frac{u_\lambda}{r}\right) - \frac{u_\lambda}{r^3 \sin^2 \lambda} - \frac{2 \cos \lambda}{r^3 \sin^2 \lambda} \frac{\pa u_\varphi}{\pa \varphi}, \\
    \left(\Delta\left(\oldReps u\right)\right)_\varphi & = \nabla^2\left(\frac{u_\varphi}{r}\right) - \frac{u_\varphi}{r^3 \sin^2 \lambda} + \frac{2 \cos \lambda}{r^3 \sin^2 \lambda} \frac{\pa u_\lambda}{\pa \varphi},\\
    \nabla^2\left(\frac{u_\lambda}{r}\right) & = \frac{1}{r^3} \Delta^\prime u_\lambda, \qquad \nabla^2\left(\frac{u_\varphi}{r}\right)  = \frac{1}{r^3} \Delta^\prime u_\varphi.
\end{align*}
If $\ddivS u = 0$, then by the definition of $\ddivS$,
\[\frac{1}{\sin \lambda} \frac{\pa}{\pa \lambda}(u_\lambda \sin \lambda) + \frac{1}{\sin \lambda} \frac{\pa u_\varphi}{\pa \varphi} = 0\]
which is equivalent to
\[\frac{\pa u_\lambda}{\pa \lambda} + u_\lambda\cot{\lambda}  + \frac{1}{\sin \lambda} \frac{\pa u_\varphi}{\pa \varphi} = 0.\]
Hence using all the above relations, we obtain
\begin{align*}
    \|\Delta \left(\oldReps u\right)\|^2_{\LQeps} & = \int_{Q_\eps}\left|\Delta\left(\oldReps[u](\by)\right)\right|^2\,d\by \\
    & = \inteps \int_\bbS \frac{1}{r^6}\bigg(\left|\Delta^\prime u_\lambda(\bx) - \frac{u_\lambda(\bx)}{ \sin^2 \lambda} - \frac{2 \cos \lambda}{ \sin^2 \lambda} \frac{\pa u_\varphi(\bx)}{\pa \varphi}\right|^2 \\
    & \qquad \qquad + \left|\Delta^\prime u_\varphi(\bx) - \frac{u_\varphi(\bx)}{\sin^2 \lambda} + \frac{2 \cos \lambda}{\sin^2 \lambda} \frac{\pa u_\lambda(\bx)}{\pa \varphi}\right|^2\bigg)r^2 d\sigma(\bx)dr \\
    & = - \frac{1}{3 r^3}\bigg|^{1+\eps}_1 \|\DDelta u\|^2_\LS = \frac{\eps^3 + 3 \eps^2 + 3 \eps}{3(1+\eps)^3}\|\DDelta u\|^2_\LS.
\end{align*}
Since $\eps \in (0,1)$, the inequality \eqref{eq:lap-ret-vec1} holds. \hfill $\square$


\section{Compactness}
\label{sec:appA}

We use the following theorem by Simon \cite[Theorem~5]{[Simon87]} to obtain a strongly converging subsequence$\colon$

\begin{theorem}\label{thm4.1}
Let $T>0$, and let us assume the embedding of the Banach spaces
$X \xhookrightarrow{\text{compact}} B \xhookrightarrow{} Y$.
Let $(w_\eps)_{\eps>0}$ be a family of functions of $L^p(0,T;X)$,
$1\le p \le \infty$, with the extra condition
$(w_\eps)_{\eps>0} \subset \ccal([0,T]; Y)$ if $p = \infty$, such that
\begin{itemize}
\item[(H1)] $(w_\eps)_{\eps>0}$ is bounded in $L^p(0,T;X)$
\item[(H2)] $\|w_\eps(x,t+\theta) - w_\eps(x,t)\|_{L^p(0,T-\theta;Y)} \rightarrow 0$
as $\theta \rightarrow 0$, uniformly for all $w_\eps$.
\end{itemize}
Then the family $(w_\eps)_{\eps>0}$ possesses a cluster point in $L^p(0,T;B)$ and also
in $\ccal([0,T]; B)$ if $p = \infty$ as $\eps \rightarrow 0$.
\end{theorem}

\subsection{Skorokhod Theorem and Aldous condition}
\label{sec:appA.1}

Let $E$ be a separable Banach space with the norm $\|\cdot\|_E$ and let $\calB(E)$ be its Borel $\sigma$-field. The family of probability measures on $(E, \calB(E))$ will be denoted by $\mathcal{P}$. The set of all bounded and continuous $E$-valued functions is denoted by $\ccal_b(E)$.

\begin{definition}
\label{defnA.1.1}
The family $\mathcal{P}$ of probability measures on $\left(E, \calB(E)\right)$ is said to be tight if for arbitrary $\varepsilon > 0$ there exists a compact set $K_\varepsilon \subset E$ such that
\[\mu(K_\varepsilon) \ge 1 - \varepsilon, \quad \mbox{for all}\,\, \mu \in \mathcal{P}.\]
\end{definition}

We used the following Jakubowski's generalisation of the Skorokhod Theorem, in the form given by Brze\'{z}niak and Ondrej\'{a}t \cite[Theorem~C.1]{[BO11]}, see also \cite{[Jakubowski98]}, as our topological space $\mathcal{Z}_T$ is not a metric space.

\begin{theorem}
\label{thmA.1.4}
Let $\mathcal{X}$ be a topological space such that there exists a sequence $\{f_m\}_{m \in \mathbb{N}}$ of continuous functions $f_m : \mathcal{X} \to \R$ that separates points of $\mathcal{X}$. Let us denote by $\mathcal{S}$ the $\sigma$-algebra generated by the maps $\{f_m\}$. Then
\begin{itemize}
\item[a)] every compact subset of $\mathcal{X}$ is metrizable,
\item[b)] if $(\mu_m)_{m \in \mathbb{N}}$ is a tight sequence of probability measures on $(\mathcal{X}, \mathcal{S})$, then there exists a subsequence $(m_k)_{k \in \mathbb{N}}$, a probability space $(\Omega, \mathcal{F}, \mathbb{P})$ with $\mathcal{X}$-valued Borel measurable variables $\xi_k, \xi$ such that $\mu_{m_k}$ is the law of $\xi_k$ and $\xi_k$ converges to $\xi$ almost surely on $\Omega$.
\end{itemize}
\end{theorem}

Let $(\mathbb{S}, \varrho)$ be a separable and complete metric space.

\begin{definition}
\label{defnA.1.5}
Let $u \in \ccal([0,T]; \mathbb{S})$. The modulus of continuity of $u$ on $[0,T]$ is defined by
\[m(u, \delta) :=  \sup_{s,t \in [0,T],\,|t - s|\le \delta} \varrho (u(t), u(s)), \quad \delta > 0.\]
\end{definition}

Let $(\Omega, \mathcal{F}, \mathbb{P})$ be a probability space with filtration $\mathbb{F}:= (\mathcal{F}_t)_{t \in [0,T]}$ satisfying the usual conditions, see \cite{[Metivier82]},
and let $(X_n)_{n \in \mathbb{N}}$ be a sequence of continuous $\mathbb{F}$-adapted $\mathbb{S}$-valued processes.

\begin{definition}
\label{defnA.1.6}
We say that the sequence $(X_n)_{n \in \mathbb{N}}$ of $\mathbb{S}$-valued random variables satisfies condition $[\mathbf{T}]$ iff $\forall\, \varepsilon >0, \forall\, \eta > 0,\, \exists\, \delta > 0$:
\begin{equation}
\label{eq:A.1.1}
\sup_{n \in \mathbb{N}} \mathbb{P}\left\{m(X_n, \delta) > \eta\right\} \le \varepsilon.
\end{equation}
\end{definition}

\begin{lemma}\cite[Lemma~2.4]{[BM14]}
\label{lemmaA.1.7}
Assume that $(X_n)_{n \in \mathbb{N}}$ satisfies condition $[\mathbf{T}]$. Let $\mathbb{P}_n$ be the law of $X_n$ on $\ccal([0,T]; \mathbb{S})$, $n \in \mathbb{N}$. Then for every $\varepsilon > 0$ there exists a subset $A_\varepsilon \subset \ccal([0,T]; \mathbb{S})$ such that
\[\sup_{n \in \mathbb{N}} \mathbb{P}_n(A_\varepsilon) \ge 1 - \varepsilon\]
and
\begin{equation}
\label{eq:A.1.2}
\lim_{\delta \to 0} \sup_{u \in A_\varepsilon} m(u, \delta) = 0.
\end{equation}
\end{lemma}

Now we recall the Aldous condition $[\mathbf{A}]$, which is connected with condition $[\mathbf{T}]$ (see \cite{[Metivier88]} and \cite{[Aldous78]}). This condition allows to investigate the modulus of continuity for the sequence of stochastic processes by means of stopped processes.

\begin{definition}[Aldous condition]
\label{defnA.1.8}
A sequence $(X_n)_{n \in \mathbb{N}}$ satisfies condition $[\mathbf{A}]$ iff $\,\forall\, \varepsilon > 0$, $\forall\, \eta > 0$, $\exists \, \delta > 0$ such that for every sequence $(\tau_n)_{n \in \mathbb{N}}$ of $\mathbb{F}$-stopping times with $\tau_n \le T$ one has
\[\sup_{n \in \mathbb{N}} \sup_{0 \le \theta \le \delta} \mathbb{P}\left\{\varrho(X_n(\tau_n + \theta), X_n(\tau_n)) \ge \eta \right\} \le \varepsilon.\]
\end{definition}

\begin{lemma}
\label{lemmaA.1.9}
\cite[Theorem~3.2]{[Metivier88]}
Conditions $[\mathbf{A}]$ and $[\mathbf{T}]$ are equivalent.
\end{lemma}

\subsection{Tightness criterion}
\label{sec:appA.2}
Now we formulate the compactness criterion analogous to the result due to Mikulevicus and Rozowskii \cite{[MR05]}, Brze\'zniak and Motyl \cite{[BM14]} for the space $\mathcal{Z}_T$, see also \cite[Lemma~4.2]{[BD18]}.

\begin{lemma}
\label{lemmaA.2.1}
Let $\mathcal{Z}_T$, $\mathcal{T}$ be as defined in \eqref{eq:5.32}. Then a set $\kcal \subset \mathcal{Z}_T$ is $\mathcal{T}$-relatively compact if the following three conditions hold
\begin{itemize}
\item[(a)] $\sup_{u \in \kcal} \sup_{s \in [0,T]} \|u(s)\|_{\LS} < \infty,$
\item[(b)] $\sup_{u \in \kcal} \int_0^T \|u(s)\|^2_\rV\,ds < \infty\,$, i.e. $\kcal$ is bounded in $L^2(0,T; \rV)$,
\item[(c)] $\lim_{\delta \to 0} \sup_{u \in \kcal} \sup_{\underset{|t-s| \le \delta}{s,t \in [0,T]}}\|u(t) - u(s)\|_{\rD(\rA^{-1})} = 0.$
\end{itemize}
\end{lemma}

Using Section~\ref{sec:appA.1} and the compactness criterion from Lemma~\ref{lemmaA.2.1} we obtain the following corollary.

\begin{corollary}[Tightness criterion]
\label{corA.2.2}
Let $(\alpha_\eps)_{\eps > 0}$ be a sequence of continuous $\mathbb{F}$-adapted $\mathrm{H}$-valued processes such that
\begin{itemize}
\item[(a)] there exists a constant $C_1 > 0$ such that
\[\sup_{\eps > 0}\,\E \left[ \sup_{s \in [0,T]} \|\ale(s)\|^2_{\mathrm{H}} \right] \le C_1,\]

\item[(b)] there exists a constant $C_2 > 0$ such that
\[\sup_{\eps > 0}\,\E \left[ \int_0^T \|\curlS \ale(s)\|^2_{\LS}\,ds \right] \le C_2,\]

\item[(c)] $(\ale)_{\eps > 0}$ satisfies the Aldous condition $[\mathbf{A}]$ in $\rD(\rA^{-1})$.
\end{itemize}
Let $\widetilde{\mathbb{P}}_\eps$ be the law of $\ale$ on $\mathcal{Z}_T$. Then for every $\delta > 0$ there exists a compact subset $K_\delta$ of $\mathcal{Z}_T$ such that
\[ \sup_{\eps > 0} \widetilde{\mathbb{P}}_\eps(K_\delta) \ge 1 - \delta.\]
\end{corollary}


\section{Kuratowski Theorem and proof of Lemma~\ref{lemma5.7}}
\label{s:AppB}
This appendix is dedicated to the proof of Lemma~\ref{lemma5.7}. We will first recall the Kuratowski Theorem \cite{[Kuratowski52]} in the next subsection and prove some related results which will be used later to prove Lemma~\ref{lemma5.7} in Subsection~\ref{s:AppB.2}.

\subsection{Kuratowski Theorem and related results}
\label{s:AppB.1}

\begin{theorem}
\label{thmB.1.1}
Assume that $X_1, X_2$ are the Polish spaces with their Borel $\sigma$-fields denoted respectively by $\bcal(X_1), \bcal(X_2)$. If $\varphi \colon X_1 \to X_2$ is an injective Borel measurable map then for any $E_1 \in \bcal(X_1)$, $E_2 := \varphi(E_1) \in \bcal(X_2)$.
\end{theorem}

Next two lemmas are the main results of this appendix. For the proof of Lemma~\ref{lemmaB.1.2} please see \cite[Appendix~B]{[BD19]}.

\begin{lemma}
\label{lemmaB.1.2}
Let $X_1, X_2$ and $Z$ be topological spaces such that $X_1$ is a Borel subset of $X_2$. Then $X_1 \cap Z$ is a Borel subset of $X_2 \cap Z$, where $X_2 \cap Z$ is a topological space too, with the topology given by
\begin{equation}\label{eq:B.1.1}
\tau(X_2 \cap Z) = \left\{ A \cap B : A \in \tau(X_2), B \in \tau(Z)\right\}.
\end{equation}
\end{lemma}

\subsection{Proof of Lemma~\ref{lemma5.7}}
\label{s:AppB.2}

In this subsection we recall Lemma~\ref{lemma5.7} and prove it using the results from previous subsection.

\begin{lemma}
\label{lemmaB.2.1}
Let $T > 0$ and $\mathcal{Z}_T$ be as defined in \eqref{eq:5.32}. Then, the following sets $\ccal([0,T];\rH) \cap \mathcal{Z}_T$, $L^2(0,T; \rV) \cap \mathcal{Z}_T$ are Borel subsets of $\mathcal{Z}_T$.
\end{lemma}

\begin{proof}
First of all $\ccal([0,T]; \rH) \subset \ccal([0,T]; \rD(\rA^{-1})) \cap L^2(0,T; \rH)$. Secondly, $\ccal([0,T]; \rH)$ and $\ccal([0,T]; \rD(\rA^{-1})) \cap L^2(0,T; \rH)$ are Polish spaces. And finally, since $\rH$ is continuously embedded in $\rD(\rA^{-1})$, the map
\[i \colon \ccal([0,T]; \rH) \to \ccal([0,T]; \rD(\rA^{-1})) \cap L^2(0,T; \rH),\]
is continuous and hence Borel. Thus, by application of the Kuratowski Theorem (see Theorem~\ref{thmB.1.1}), $\ccal([0,T]; \rH)$ is a Borel subset of $\ccal([0,T]; \rD(\rA^{-1})) \cap L^2(0,T; \rH)$. Therefore, by Lemma~\ref{lemmaB.1.2}, $\ccal([0,T]; \rH) \cap {\mathcal{Z}}_T$ is a Borel subset of $\ccal([0,T]; \rD(\rA^{-1})) \cap L^2(0,T; \rH) \cap {\mathcal{Z}}_T$ which is equal to ${\mathcal{Z}}_T$.

Similarly we can show that $L^2(0,T; \rV) \cap {\mathcal{Z}}_T$ is a Borel subset of ${\mathcal{Z}}_T$. $L^2(0,T; \rV) \hookrightarrow L^2(0,T; \rH)$ and both are Polish spaces thus by application of the Kuratowski Theorem, $L^2(0,T; \rV)$ is a Borel subset of $L^2(0,T; \rH)$. Finally, we can conclude the proof of lemma by Lemma~\ref{lemmaB.1.2}.
\end{proof}


\begin{thebibliography}{10}

\bibitem{[Aldous78]} D. Aldous.
\newblock Stopping times and tightness.
\newblock Ann. Probability, 6 (1978), no. 2, pp. 335--340.

\bibitem{[Arvin96]} J. D. Avrin.
\newblock Large-eigenvalue global existence and regularity results for the Navier--Stokes
equation.
\newblock J. Differential Equations, 127 (1996), pp. 365--390.

\bibitem{Aub98}
T.~Aubin.
\newblock {Some nonlinear problems in Riemannian geometry}.
\newblock Springer-Verlag, New York, 1998.

\bibitem{[BV83]} A.V. Babin and M. I. Vishik.
\newblock Attractors of partial differential equations and estimate of their dimension.
\newblock Russian Math. Survey, 38 (1983), pp. 151--213.

\bibitem{[BD18]} Z. Brze\'{z}niak and G. Dhariwal.
\newblock Stochastic constrained Navier--Stokes equations on $\mathbb{T}^2$.
\newblock Submitted (2019). \href{https://arxiv.org/abs/1701.01385}{arXiv:1701.01385}

\bibitem{[BD19]} Z. Brze\'{z}niak and G. Dhariwal.
\newblock Stochastic tamed Navier--Stokes equations on $\R^3$: the existence and the uniqueness of solutions and the existence of an invariant measure.
\newblock To appear in J. Mathematical Fluid Mechanics (2019). \href{https://arxiv.org/abs/1904.13295}{arXiv:1904.13295}

\bibitem{[BDL19]} Z. Brze\'{z}niak, G. Dhariwal and Q. T. Le Gia.
\newblock Stochastic Navier--Stokes equations on a thin spherical domain: Existence of a martingale solution (In preparation).

\bibitem{[BGJ13]} Z. Brze\'{z}niak, B. Goldys and T. Jegaraj.
\newblock Weak solutions of a stochastic Landau-Lifshitz-Gilbert equation.
\newblock Appl. Math. Research eXpress 2013 (2013), no. 1, pp. 1--33.

\bibitem{[BGL15]} Z. Brze\'{z}niak, B. Goldys and Q. T. Le Gia.
\newblock Random dynamical systems generated by stochastic Navier--Stokes equations on a rotating sphere.
\newblock J. Math. Anal. Appl. 426 (2015), pp. 505--545.

\bibitem{[BGL18]} Z. Brze\'{z}niak, B. Goldys and Q. T. Le Gia.
\newblock Random attractors for the stochastic Navier--Stokes equations on the 2D unit sphere.
\newblock J. Math. Fluid. Mech., 20 (2018), pp. 227--253.

\bibitem{[BM13]} Z. Brze\'{z}niak and E. Motyl.
\newblock Existence of a martingale solution of the stochastic Navier--Stokes equations in unbounded 2D and 3D domains.
\newblock J. Differential Equations 254 (2013), no. 4, pp. 1627--1685.

\bibitem{[BM14]} Z. Brze\'{z}niak and E. Motyl.
\newblock The existence of martingale solutions to the stochastic Boussinesq equations.
\newblock Global and Stochastic Analysis 1 (2014), no. 2, pp. 175--216.

\bibitem{[BMO17]} Z. Brze\'{z}niak, E. Motyl and M. Ondrej\'{a}t.
\newblock Invariant measure for the stochastic Navier--Stokes equations in unbounded 2D domains.
\newblock Ann. Probab. 45 (2017), no. 5, pp. 3145--3201.

\bibitem{[BO11]} Z. Brze\'{z}niak and M. Ondrej\'{a}t.
\newblock Stochastic wave equations with values in Riemanninan manifolds.
\newblock Stochastic partial differential equations and applications, Quaderni di Matematica 25 (2011), pp. 65--97.

\bibitem{[Cattabriga61]} L. Cattabriga.
\newblock Su un problema al contorno relativo al sistema di equazioni di Stokes.
\newblock Rend. Semin. Mat. Univ. Padova 31 (1961), pp. 308--340.

\bibitem{[Ciarlet90]} Ph. Ciarlet.
\newblock Plates and junctions in elastic multi-structures. An asymptotic analysis.
\newblock (1990), Masson, Paris and Springer-Verlag, New York.

\bibitem{[DL76]} G. Duvaut and J. L. Lions.
\newblock Inequalities in mechanics and physics.
\newblock Springer, Berlin, Heidelberg (1976).

\bibitem{[GT91]} J. M. Ghidaglia and R. Temam.
\newblock Lower bound on the dimension of the attractor for the Navier--Stokes equations in space dimension 3.
\newblock Mechanics, analysis and geometry: 200 years after Lagrange (1991), pp. 33--60, North-Hollan Delta Ser., North-Holland, Amsterdam.

\bibitem{[Grigoryan00]} A. Grigoryan.
\newblock Heat Kernel and Analysis on Manifolds (2000).
\newblock Amer. Math. Soc., Providence, RI.

\bibitem{[HR92a]} J. K. Hale and G. Raugel.
\newblock A damped hyperbolic equation on thin domains.
\newblock Trans. Amer. Math. Soc., 329 (1992), pp. 185--219.

\bibitem{[HR92b]} J. K. Hale and G. Raugel.
\newblock Partial differential equations on thin domains.
\newblock Differential equations and mathematical physics (Birmingham, AL, 1990), (1992), pp. 63--97, Math. Sci. Engrg., 186, Academic Press, Boston, MA.

\bibitem{[HR92c]} J. K. Hale and G. Raugel.
\newblock Reaction-diffusion equation on thin domains.
\newblock J. Math. Pures Appl., 71 (1992), pp. 33--95

\bibitem{IbraPelin2009} R. N. Ibragimov and D. E. Pelinovsky.
\newblock{Incompressible Viscous Fluid Flows in a Thin Spherical Shell.}
\newblock J. Math. Fluid Mech., 11 (2009), pp. 60--90.

\bibitem{Ibra2011} R. N. Ibragimov.
\newblock{Nonlinear viscous fluid patterns in a thin rotating spherical domain and applications.}
\newblock Phys. Fluids 23, 123102 (2011).

\bibitem{IbraIbra2011} N.~H.~Ibragimov and R.~N.~Ibragimov.
\newblock{Integration by quadratures of the nonlinear Euler equations modeling atmospheric flows in a thin rotating spherical shell.}
\newblock Phys. Lett. A 375, 3858 (2011).

\bibitem{[Iftime99]}  D. Iftimie.
\newblock The 3D Navier--Stokes equations seen as a perturbation of the 2D Navier--Stokes equations.
\newblock Bull. Soc. Math. France, 127 (1999), pp. 473--517.

\bibitem{[IR01]} D. Iftimie and G. Raugel.
\newblock Some results on the Navier--Stokes equations in thin 3D domains.
\newblock J. Differential Equations, 169 (2001), pp. 281--331.

\bibitem{[Ilin91]} A. A. Il'in.
\newblock The Navier--Stokes and Euler equations on two dimensional manifolds.
\newblock Math. USSR Sb., 69 (1991), pp. 559--579.

\bibitem{[Ilin94]} A. A. Il'in.
\newblock Partially dissipative semigroups generated by the Navier--Stokes system on two dimensional manifolds and their attractors.
\newblock Russian Acad. Sci. Sb. Math., 78 (1994), pp. 47--76.

\bibitem{[IF89]} A. A. Il'in and A. N. Filatov.
\newblock On unique solvability of the Navier--Stokes equations on the two dimensional sphere. \newblock Sov. Math. Dokl., 38 (1989), pp. 9--13.

\bibitem{[Jakubowski98]} A. Jakubowski.
\newblock The almost sure Skorokhod representation for subsequences in nonmetric spaces.
\newblock Teor. Veroyatn. Primen.,  42 (1998), no. 1, pp. 209--216; translation in Theory Probab. Appl., 42 (1998), no. 1, pp. 167--174.

\bibitem{[Kru14]} R. Kruse.
\newblock Strong and Weak Approximation of Semilinear Stochastic Evolution Equations.
\newblock Lecture Notes Math. 2093. Springer, Cham, 2014.

\bibitem{[Kuratowski52]} K. Kuratowski.
\newblock Topologie, Vol. I (French)3‘eme ed. (1952).
\newblock Monografie Matematyczne, Tom XX, Polskie Towarzystwo Matematyczne, Warszawa.

\bibitem{LionTemamWang92} J.~L. Lions, R.~Temam and S.~Wang.
\newblock{New formulations of the primitive equations of atmosphere and applications.}
\newblock Nonlinearity 5 (1992), pp. 237--288.

\bibitem{LionTemamWang92a} J.~L. Lions, R. Temam and S. Wang.
\newblock{On the equations of the large-scale ocean.}
\newblock Nonlinearity 5 (1992), 1007--1053.

\bibitem{[LTW95]} J. L. Lions, R. Temam and S. H. Wang.
\newblock Mathematical Theory for the coupled atmosphere-ocean models.
\newblock J. Math. Pures Appl. (9), 74 (1995), no. 2, pp. 105--163.

\bibitem{[MRR95]} J. E. Marsden, T. S. Raitu and G. Raugel.
\newblock Les \'{e}quation d\`{}Euler dans des coques sph\'eriques minces.
\newblock C. R. Acad. Sci. Paris, 321 (1995), pp. 1201--1206.

\bibitem{[Metivier82]} M. M\'{e}tivier.
\newblock Semimartingales: A course on stochastic processes (1982).
\newblock Walter de Gruyter $\&$ and Co., Berlin-New York.

\bibitem{[Metivier88]} M. M\'{e}tivier.
\newblock Stochastic partial differential equations in infinite dimensions (1988).
\newblock Scuola Normale Superiore, Pisa.

\bibitem{[MR05]} R. Mikulevicius and B. L. Rozovskii.
\newblock Global $L^2$-solutions of stochastic Navier--Stokes equations.
\newblock Ann. Prob. 33 (2005), no. 1, pp. 137--176.

\bibitem{[MTZ97]}  I. Moise, R. Temam, and M. Ziane.
\newblock Asymptotic analysis of the Navier--Stokes equations in thin domains.
\newblock Topol. Methods Nonlinear Anal., 10 (1997), pp. 249--282.

\bibitem{[Motyl14]} E. Motyl.
\newblock Stochastic hydrodynamic-type evolution equations driven by L\'{e}vy noise in
3D unbounded domains -- abstract framework and applications.
\newblock Stochastic Process. Appl., 124 (2014), pp. 2052--2097.

\bibitem{[Pedlosky87]} J. Pedlosky.
\newblock Geophysical Fluid Dynamics.
\newblock 2nd Edition (1987), Springer-Verlag, New York.

\bibitem{[RS93]} G. Raugel and G. R. Sell.
\newblock Navier--Stokes equations on thin 3D domains. I. Global attractors and global regularity of solutions.
\newblock J. Amer. Math. Soc., 6 (1993), pp. 503--568.

\bibitem{[RS94]} G. Raugel and G. R. Sell.
\newblock Navier--Stokes equations on thin 3D domains. II. Global regularity of spatially periodic solutions.
\newblock Nonlinear partial differential equations and their applications, Coll\`{e}ge de France Seminar, Vol. XI (1994), pp. 205--247, Longman, Harlow.

\bibitem{[Saito05]} J. Saito.
\newblock Boussinesq equations in thin spherical domains.
\newblock Kyushu J. Math., 59 (2005), pp. 443--465.

\bibitem{[Serrin59]}
J.~Serrin,
\newblock Mathematical principles of classical fluid mechanics, 
\newblock Encly. of Physics, vol. 8, pp. 125--263, Springer-Verlag, 1959.

\bibitem{[Simon87]}
J. Simon.
\newblock Compact sets in $L^p(0,T;B)$.
\newblock Ann. Mat. Pura Appl. (4), 146 (1987), pp. 65--97.

\bibitem{[Taylor92]}M.~E.~Taylor.
\newblock Analysis on Morrey spaces and applications to Navier--Stokes and other evolution equations.
\newblock Comm. Partial Differential Equations 17 (1992), pp. 1407--1456

\bibitem{[Temam97]}
R. Temam.
\newblock Infinite-dimensional Dynamical Systems in Mechanics and Physics.
\newblock 2nd Edition (1997), Springer, New York.

\bibitem{[Temam00]}
R. Temam.
\newblock Navier--Stokes Equations: Theory and Numerical Analysis (2000).
\newblock American Mathematical Society; UK Edition

\bibitem{[TW93]} R. Temam, S. Wang.
\newblock Inertial forms of Navier--Stokes equations on the sphere.
\newblock J. Funct. Anal., 117 (1993), pp. 215--242.

\bibitem{[TZ96]} R. Temam and M. Ziane.
\newblock Navier--Stokes equations in three-dimensional thin domains
with various boundary conditions.
\newblock Adv. Differential Equations 1 (1996), pp. 499--546.

\bibitem{[TZ97]}
R. Temam and M. Ziane.
\newblock Navier--Stokes equations in thin spherical domains,
\newblock Contemp. Math. 209 (1997), 281--314.
\end{thebibliography}
\end{document}